\pdfoutput=1
\documentclass[10pt, a4paper]{amsart}
\usepackage{amssymb}
\usepackage[all]{xy}
\usepackage[T1]{fontenc}
\usepackage{mathpazo}
\usepackage{tgpagella} 
\usepackage{graphicx}
\date{22 March 2016}
\title[Commutative DG Rings]{Duality and Tilting for Commutative DG Rings}
\usepackage{hyperref}
\hypersetup{colorlinks=false}

\author{Amnon Yekutieli}
\address{Department of  Mathematics,
Ben Gurion University, Be'er Sheva 84105, Israel}
\email{amyekut@math.bgu.ac.il}

\thanks{
Primary: 13D09; Secondary: 13D07, 18G10, 16E45}

\thanks{{\em Grants}: Supported by Israel Science Foundation grant no.\ 
253/13.}

\thanks{{\em Notice regarding publication}: The author has decided not to
submit the present version of the paper to a peer reviewed journal. This 
decision is based on frustrating experience with submissions of earlier 
versions.}

\newtheorem{thm}[equation]{Theorem}
\newtheorem{cor}[equation]{Corollary}
\newtheorem{prop}[equation]{Proposition}
\newtheorem{lem}[equation]{Lemma}
\theoremstyle{definition}
\newtheorem{dfn}[equation]{Definition}
\newtheorem{rem}[equation]{Remark}
\newtheorem{exa}[equation]{Example}

\newtheorem{conv}[equation]{Convention}
\newtheorem{conj}[equation]{Conjecture}

\numberwithin{equation}{section}

\newcommand{\iso}{\xrightarrow{\simeq}}

\newcommand{\xar}{\xrightarrow}
\newcommand{\opn}{\operatorname}
\newcommand{\cat}[1]{\operatorname{\mathsf{#1}}}

\newcommand{\cd}{\,{\cdot}\,}

\newcommand{\rmitem}[1]{\item[\text{\textup{(#1)}}]}
\newcommand{\mfrak}[1]{\mathfrak{#1}}
\newcommand{\mcal}[1]{\mathcal{#1}}

\newcommand{\mrm}[1]{\mathrm{#1}}

\newcommand{\OO}{\mcal{O}}
\newcommand{\MM}{\mcal{M}}

\newcommand{\RR}{\mcal{R}}

\newcommand{\Ga}{\Gamma}
\newcommand{\si}{\sigma}
\newcommand{\la}{\lambda}
\newcommand{\de}{\delta}

\newcommand{\al}{\alpha}
\newcommand{\be}{\beta}
\newcommand{\ga}{\gamma}

\newcommand{\Om}{\Omega}
\newcommand{\p}{\mfrak{p}}
\newcommand{\q}{\mfrak{q}}
\newcommand{\m}{\mfrak{m}}

\renewcommand{\a}{\mfrak{a}}

\newcommand{\K}{\mathbb{K}}
\newcommand{\R}{\mathbb{R}}
\newcommand{\Q}{\mathbb{Q}}
\newcommand{\Z}{\mathbb{Z}}
\newcommand{\N}{\mathbb{N}}

\newcommand{\Hom}{\mcal{H} \! om}

\newcommand{\tup}[1]{\textup{#1}}
\newcommand{\bsym}[1]{\boldsymbol{#1}}

\newcommand{\ot}{\otimes}

\newcommand{\til}[1]{\tilde{#1}}

\renewcommand{\d}{\mathrm{d}}
\newcommand{\pa}{\partial}

\newcommand{\smfrac}[2]{{\textstyle \frac{#1}{#2}}}

\newcommand{\lb}{\linebreak}

\begin{document}

\begin{abstract}
We consider commutative DG rings (better known as nonpositive \lb strongly 
commutative associative unital DG algebras). For such a DG ring $A$ we define 
the notions of {\em perfect, tilting,  dualizing, Cohen-Macaulay 
and rigid DG $A$-modules}.
{\em Geometrically perfect} DG modules are defined by a local condition on 
$\opn{Spec} \bar{A}$, where $\bar{A}$ is the commutative ring $\opn{H}^0(A)$.
{\em Algebraically perfect} DG modules are those that can be obtained from $A$ 
by finitely many shifts, direct summands and cones. 
Tilting DG modules are those that have inverses w.r.t.\ the derived tensor 
product; their isomorphism classes form the derived Picard group 
$\opn{DPic}(A)$. 
Dualizing DG modules are a generalization of Grothendieck's 
original definition (and here $A$ has to be {\em cohomologically 
pseudo-noetherian}). Cohen-Macaulay DG modules are the duals (w.r.t.\ a given 
dualizing DG module) of finite $\bar{A}$-modules. Rigid DG $A$-modules, 
relative to a commutative base ring $\K$, are defined using the {\em squaring 
operation}, and this is a generalization of Van den Bergh's original 
definition. 

The techniques we use are the standard ones of derived categories, with a few 
improvements. We introduce a new method for studying DG $A$-modules: \v{C}ech 
resolutions of DG $A$-modules corresponding to open coverings of
$\opn{Spec} \bar{A}$. 

Here are some of the new results obtained in this paper:
\begin{itemize}
\item A DG $A$-module is geometrically perfect iff it is algebraically perfect.

\item The canonical group homomorphism $\opn{DPic}(A) \to \opn{DPic}(\bar{A})$ 
is bijective. 

\item The group $\opn{DPic}(A)$ acts simply transitively on the set of 
isomorphism classes of dualizing DG $A$-modules.

\item Cohen-Macaulay DG modules are insensitive to cohomologically surjective 
DG ring homomorphisms. 

\item Rigid dualizing DG $A$-modules  are unique up to unique rigid 
isomorphisms. 
\end{itemize}

The functorial properties of Cohen-Macaulay DG modules that we establish 
here are needed for our work on rigid dualizing complexes over 
commutative rings, schemes and Deligne-Mumford stacks. 

We pose several conjectures regarding existence and uniqueness of rigid DG 
modules over commutative DG rings.  
\end{abstract}

\maketitle

\tableofcontents

\setcounter{section}{-1}
\section{Introduction}

In this paper we are interested in {\em commutative DG rings}. 
This is an abbreviation for strongly commutative nonpositive associative 
unital differential graded algebras over $\Z$. Thus a commutative DG ring is a 
graded ring $A = \bigoplus_{i \leq 0} A^i$, together with a 
differential $\d$ of degree $1$, that satisfies the graded Leibniz rule. 
The multiplication satisfies 
$b \cd a = (-1)^{i j} \cd a \cd b$
for all $a \in A^i$ and $b \in A^j$, and $a \cd a = 0$ if $i$ is odd.
By default, the DG rings in the Introduction are commutative. 
Rings are viewed as DG rings concentrated in degree $0$ (so they are 
commutative by default throughout the Introduction). 

Commutative DG rings come up in the foundations of {\em derived algebraic 
geometry}, as developed by To\"en-Vezzosi \cite{TV}.
Indeed, a {\em derived stack} is a stack of groupoids
on the site of commutative DG rings (with its \'etale topology). 
Some precursors of this point of view are the papers \cite{Hi2}, 
\cite{Ke2}, \cite{KoSo} and \cite{Be}. 
Another role of commutative DG rings is as {\em resolutions of commutative 
rings}. We shall say more on this role below, since it was the immediate 
motivation for writing the present paper. 

In our paper we study {\em perfect, tilting, dualizing, 
Cohen-Macaulay and rigid DG modules} over commutative DG 
rings. We give definitions that generalize the familiar definitions for rings. 
Many of the results that hold for rings, continue to hold in the much more 
complicated DG setting. Later in the Introduction we discuss the main 
definitions and results of the paper, and a couple of conjectures.  
  
But first let us explain the problem that motivated our work on commutative DG 
rings. {\em Rigid dualizing complexes} over noetherian
commutative rings are the foundation of a new approach to 
Grothendieck Duality on schemes and Deligne-Mumford stacks. 
See the papers \cite{YZ1}, \cite{YZ2}, \cite{Ye3}, \cite{Ye6}, \cite{Ye8} and 
\cite{Ye9}. Rigid dualizing complexes were introduced by Van den Bergh in the 
context of noncommutative rings over a base field $\K$; see \cite{VdB}.
But for the new approach to Grothendieck Duality in algebraic geometry, 
that should apply also to the arithmetic setup, we are interested in 
commutative rings over a base commutative ring $\K$ that is {\em not a field}. 

In this context, the definition of a rigid complex relies on the more primitive 
notion of the {\em square of a complex}. Given a commutative $\K$-ring $A$, 
we choose a  {\em K-flat commutative DG ring resolution} $\til{A} \to A$ 
relative $\K$. Such resolutions exist, because we work with strongly 
commutative DG rings (see Remark \ref{rem:580} for a discussion of this issue). 
For any complex of $A$-modules $M$, its square is the complex of $A$-modules
\begin{equation} \label{eqn:100}
\opn{Sq}_{A / \K}(M) := \opn{RHom}_{\til{A} \ot_{\K} \til{A}}
(A, M \ot_{\K}^{\mrm{L}} M) . 
\end{equation}
The question of independence of $\opn{Sq}_{A / \K}(M)$ of the 
choice of resolution $\til{A}$ is very subtle; and in fact there was a mistake 
in the original proof in \cite{YZ1} 
(which has since been corrected; see \cite{Ye5}). 
A rigid dualizing complex over $A$ relative to $\K$ is a pair 
$(R, \rho)$, consisting of a dualizing complex $R$, and an isomorphism 
$\rho : R \iso \opn{Sq}_{A / \K}(R)$ in the derived category $\cat{D}(A)$. 

In the course of writing the new paper \cite{Ye7} -- which corrects further 
mistakes in the earlier paper \cite{YZ1}, and extends it  -- we realized 
that we need Theorem \ref{thm:61}, that deals with Cohen-Macaulay DG modules.
This is explained in Remark \ref{rem:120}.
The parts of the present paper leading to Section \ref{sec:cm} set the stage 
for the definition of Cohen-Macaulay DG modules and the proof of Theorem 
\ref{thm:61}. 

From a wider perspective, we expect that the results in this paper shall find 
further applications in algebra and geometry. One such possible application 
would be the development of 
a theory of {\em rigid dualizing complexes in derived algebraic 
geometry} (of either flavor: Lurie's or To\"en's). Indeed, a commutative DG 
ring is an affine derived scheme; so what is needed is a way to sheafify our 
constructions. That should be facilitated by  the {\em \'etale descent} and 
{\em \'etale codescent} properties of {\em rigid residue complexes} (see 
\cite{Ye6}). 

Our work in the present paper should be easily accessible to anyone with a 
working knowledge of the derived category of modules over a ring (e.g.\ from 
the 
book \cite{We}). This is because the methods we use are basically the same; 
there are only slight modifications. The necessary tools to upgrade from 
rings to DG rings (such as K-injective resolutions in place of injective 
resolutions) are recalled in Section \ref{sec:resol} of our paper. We do not 
resort at all to the daunting technicalities of $\mrm{E}_{\infty}$ rings. 
Likewise, we do not touch simplicial methods or Quillen model structures. 

Let us now describe the work in this paper. 
Consider a commutative DG ring $A = \bigoplus_{i \leq 0} A^i$. Its cohomology 
$\opn{H}(A) = \bigoplus_{i \leq 0} \opn{H}^i(A)$ is a commutative 
graded ring. We use the notation $\bar{A} := \opn{H}^0(A)$. There is a 
canonical homomorphism of DG rings $A \to \bar{A}$. 
We say that $A$ is {\em cohomologically pseudo-noetherian} if $\bar{A}$ is a 
noetherian ring, and $\opn{H}^i(A)$ is a finite (i.e.\ finitely generated) 
$\bar{A}$-module for every $i$. (Note that this is weaker than the condition 
that the ring $\opn{H}(A)$ is noetherian.)

The category of DG $A$-modules is denoted by $\cat{C}(A)$. It is a DG 
category, and its derived category, gotten by inverting the 
quasi-isomorphisms, is denoted by $\cat{D}(A)$. There are full triangulated 
subcategories $\cat{D}^{+}(A)$, $\cat{D}^{-}(A)$ and $\cat{D}^{\mrm{b}}(A)$
of $\cat{D}(A)$, made up of the DG modules $M$ with bounded below, 
bounded above and bounded cohomologies, respectively. 
The full subcategory of $\cat{D}(A)$ on the DG modules $M$, whose cohomology 
modules $\opn{H}^i(M)$ are finite over $\bar{A}$, is denoted by
$\cat{D}_{\mrm{f}}(A)$. As usual, for any boundedness condition $\star$ we let 
$\cat{D}^{\star}_{\mrm{f}}(A) := \cat{D}_{\mrm{f}}(A) 
\cap \cat{D}^{\star}(A)$. If $A$ is cohomologically
pseudo-noetherian, then the categories $\cat{D}^{\star}_{\mrm{f}}(A)$ are 
triangulated, and $A \in \cat{D}^{-}_{\mrm{f}}(A)$.
In case $A$ is a ring, then 
$\cat{D}(A) = \cat{D}(\cat{Mod} A)$, the derived category of $A$-modules. 

In Section \ref{sec:resol} we recall some facts on DG modules. We mention 
several kinds of resolutions of DG modules, and special attention is paid to 
semi-free resolutions. 
Section \ref{sec:coh-dim} is about various notions of 
cohomological dimension for DG modules and derived functors.
For instance, in Definition \ref{dfn:11} we introduce the projective and  
injective dimensions of a DG $A$-module $M$ {\em relative to a subcategory}
$\cat{E} \subseteq \cat{D}(A)$. 
In Section \ref{sec:lifting} we study the reduction functor 
$\cat{D}(A) \to \cat{D}(\bar{A})$, 
$M \mapsto \bar{A} \ot^{\mrm{L}}_{A} M$. 
We show that projective $\bar{A}$-modules can be lifted to DG $A$-modules. 

In Section \ref{sec:localiz} we discuss localization of
a commutative DG ring $A$ on $\opn{Spec} \bar{A}$. We introduce the {\em 
\v{C}ech resolution} $\opn{C}(M; \bsym{a})$ of a DG $A$-module $M$,
associated to a {\em covering sequence} 
$\bsym{a} = (a_1, \ldots, a_n)$ of $\bar{A}$. 
In case there is a decomposition 
$\opn{Spec} \bar{A} = \coprod_{i = 1}^n \opn{Spec} \bar{A}_i$
into open-closed subsets, we show that there are canonically defined DG rings 
$A_1, \ldots, A_n$, and a DG ring 
quasi-iso\-morph\-ism $A \to \prod_{i = 1}^n A_i$, that in $\opn{H}^0$ recovers 
the decomposition
$\bar{A} \cong \prod_{i = 1}^n \bar{A}_i$.

The topic of Section \ref{sec:perfect} is {\em perfect} DG modules. 
A DG $A$-module $P$ is called {\em geometrically perfect} if locally on 
$\opn{Spec} \bar{A}$ it is isomorphic, in the derived category, to a finite 
semi-free DG module. See Definition \ref{lem:34} for the precise formulation. 
This definition appears to be completely new for DG rings.
When $A$ is a ring (so that $A = A^0 = \bar{A})$, this definition coincides 
with the one in \cite[Expos\'e I]{SGA-6}, since a finite semi-free DG module 
over a ring is just a bounded complex of finite free modules. 

A DG $A$-module $P$ is called  {\em algebraically perfect} if it can be 
finitely built from $A$ by shifts, direct summands and cones. In other words, 
if $P$ belongs to the epaisse subcategory of $\cat{D}(A)$ classically generated 
by $A$, in the sense of \cite{BV}. This definition (without the qualification 
``algebraically'') was already used in \cite{ABIM}. 

When $A$ is a ring, it is known that geometrically perfect complexes are the 
same as algebraically perfect complexes; see \cite[Expos\'e I]{SGA-6}. 
Therefore they are just called ``perfect complexes''.  
 
Here is the main result Section \ref{sec:perfect}. It is a combination of 
Theorem \ref{thm:50}, Theorem \ref{thm:74} and Corollary \ref{cor:550}.

\begin{thm} \label{thm:63}
Let $A$ be a commutative DG ring, and let $P$ be a DG $A$-module. 
The following four conditions are equivalent\tup{:}
\begin{enumerate}
\rmitem{i} The DG $A$-module $P$ is geometrically perfect. 

\rmitem{ii} The DG $A$-module $P$ is in $\cat{D}^{-}(A)$, and the DG 
$\bar{A}$-module $\bar{A} \ot^{\mrm{L}}_A P$ is geometrically perfect.

\rmitem{iii} For any $M, N \in \cat{D}(A)$, the canonical morphism 
\[ \opn{RHom}_A(P, M) \ot^{\mrm{L}}_A N \to 
\opn{RHom}_A(P, M \ot^{\mrm{L}}_A N) \]
in $\cat{D}(A)$ is an isomorphism.

\rmitem{iv} The DG $A$-module $P$ is algebraically perfect. 
\end{enumerate}
If $A$ is cohomologically pseudo-noetherian, then the four conditions above are 
equivalent to\tup{:}
\begin{enumerate}
\rmitem{v} The DG $A$-module $P$ is in $\cat{D}^{-}_{\mrm{f}}(A)$,
and it has finite projective dimension relative to $\cat{D}(A)$.
\end{enumerate}
\end{thm}

When $A$ is a commutative ring, conditions (i) and (ii) are essentially the 
same; and (as already mentioned above) the equivalence of conditions 
(i) and (iv) was proved in \cite{SGA-6}. 
But for DG rings this is a new result. In light of Theorem \ref{thm:63} we can 
unambiguously talk about ``perfect DG modules''.  

Section \ref{sec:tilting} is about {\em tilting} DG modules. 
A DG $A$-module $P$ is said to be tilting if there is some DG module
$Q$ such that $P \ot^{\mrm{L}}_A Q \cong A$ in $\cat{D}(A)$. The DG module
$Q$ is called a quasi-inverse of $P$. 
The next theorem is repeated as Theorem \ref{thm:76}.

\begin{thm} \label{thm:420}
Let $A$ be a commutative DG ring, and let $P$ be a DG $A$-module. The following 
four conditions are equivalent\tup{:}
\begin{enumerate}
\rmitem{i}  The DG $A$-module $P$ is tilting .

\rmitem{ii} The functor $P \ot^{\mrm{L}}_A -$ is an equivalence of 
$\cat{D}(A)$.

\rmitem{ii} The functor $\opn{RHom}_A(P, -)$ is an equivalence 
of $\cat{D}(A)$.

\rmitem{iv}  The DG $A$-module $P$ is perfect, and the adjunction 
morphism $A \to \opn{RHom}_A(P, P)$ in $\cat{D}(A)$ is an isomorphism.
\end{enumerate} 
\end{thm}

A combination of Theorems \ref{thm:420} and \ref{thm:63} implies that the 
DG $A$-module $Q := \lb \opn{RHom}_A(P, A)$
is a quasi-inverse of the tilting DG module $P$.

As in \cite{Ye2}, we define the {\em commutative derived Picard group} 
$\opn{DPic}(A)$ to be the group whose elements are the 
isomorphism classes of tilting DG $A$-modules, and the multiplication is 
induced by $- \ot^{\mrm{L}}_A -$.  

If $A \to B$ is a homomorphism of DG rings,  
then the operation $P \mapsto B \ot^{\mrm{L}}_A P$ induces a group homomorphism 
$\opn{DPic}(A) \to \opn{DPic}(B)$.
The next result is Theorem \ref{thm:40} in the body of the paper. 

\begin{thm} \label{thm:60}
Let $A$ be a commutative DG ring. The canonical group homomorphism 
\[ \opn{DPic}(A) \to \opn{DPic}(\bar{A}) \]
is bijective.
\end{thm}

In an earlier version of the paper we had a finiteness condition: we required 
$\opn{Spec} \bar{A}$ to have finitely many connected components. 
But, as noticed by Negron \cite{Ng}, this condition is superfluous.

It is known that the commutative derived Picard group of the ring $\bar{A}$ has 
this structure: 
\[ \opn{DPic}(\bar{A}) \cong 
\opn{Pic}(\bar{A}) \times \opn{F}_{\mrm{lc}}(\opn{Spec} \bar{A}, \Z)  \ . \]
Here $\opn{F}_{\mrm{lc}}(\opn{Spec} \bar{A}, \Z)$
is the group of locally constant functions $\opn{Spec} \bar{A} \to \Z$, and 
$ \opn{Pic}(\bar{A})$ is the usual (commutative) Picard group. See
Theorem \ref{thm:315}, due to Negron, that refines earlier results in 
\cite{Ye2}, \cite{RZ} and \cite {Ye4}.

In Section \ref{sec:dualizing} we talk about {\em dualizing} DG modules. 
Here $A$ is a cohomologically pseudo-noetherian commutative DG ring.
A DG $A$-module $R \in \cat{D}^{+}_{\mrm{f}}(A)$ 
is called dualizing if it has finite injective dimension relative to 
$\cat{D}(A)$, and the adjunction morphism 
$A \to \opn{RHom}_A(R, R)$ is an isomorphism. 
Note that when $A$ is a ring, this is precisely the original definition found 
in \cite{RD}; but for a DG ring there are several possible notions of 
injective dimension, and the correct one has to be used. See  
Definition \ref{dfn:11}(2) and Remark \ref{rem:50}.
Note also that $R$ need not have bounded cohomology -- see 
Corollary \ref{cor:380}  and Example \ref{exa:400}.
For comparisons to dualizing DG modules, as defined previously in \cite{Hi1}, 
\cite{FIJ} and \cite{Lu2}, see Example \ref{exa:113}, Proposition \ref{prop:80}
and Remark \ref{rem:316} respectively.

A DG ring $A$ is called  {\em tractable} if it is cohomologically 
pseudo-noetherian, and there is a homomorphism $\K \to A$ from a finite 
dimensional regular noetherian commutative ring $\K$, such that the induced 
homomorphism $\K \to \bar{A}$ is essentially finite type. Such a homomorphism
$\K \to A$ is called a {\em traction} for $A$. 

The next result is a combination of Theorem \ref{thm:51} and Corollary 
\ref{cor:50}. When $A$ is a ring, this was proved by Grothendieck 
\cite[Sections V.3 and V.10]{RD}.

\begin{thm} \label{thm:90}
Let $A$ be a tractable  commutative DG ring. Then\tup{:}
\begin{enumerate}
\item $A$ has a dualizing DG module. 
\item The operation $(P, R) \mapsto P \ot^{\mrm{L}}_A R$,
for a tilting DG module $P$ and a dualizing DG module $R$, induces a simply 
transitive action of the group $\opn{DPic}(A)$ on the set of isomorphism 
classes of dualizing DG $A$-modules. 
\end{enumerate}
\end{thm}

In particular, if $\bar{A}$ is a local ring, then by Theorems \ref{thm:60} 
and \ref{thm:315} we have \lb $\opn{DPic}(A) \cong \Z$. Thus any two dualizing 
DG $A$-modules $R, R'$ satisfy $R' \cong R[m]$ for an integer $m$. 

A combination of Theorems \ref{thm:60} and \ref{thm:90} yields (see Corollary 
\ref{cor:80}):

\begin{cor}
If $A$ is a tractable commutative DG ring, then the operation 
$R \mapsto \lb \opn{RHom}_A(\bar{A}, R)$ induces a bijection 
\[ \frac{ \{  \tup{dualizing DG} \, A \tup{-modules} \} }
{ \tup{isomorphism} } \iso 
\frac{ \{ \tup{dualizing DG} \, \bar{A} \tup{-modules} \}  }
{ \tup{isomorphism} } . \]
\end{cor}

Here is a result that is quite surprising. It relies 
on a theorem of J{\o}rgensen \cite{Jo}, who proved it in the local case (i.e.\ 
when $\bar{A}$ is a local ring). 

\begin{thm} 
Let $A$ be a cohomologically bounded tractable commutative DG ring. 
If $\bar{A}$ is a perfect DG $A$-module, then the canonical homomorphism 
$A \to \bar{A}$ is a quasi-isomorphism. 
\end{thm}

This is repeated (in slightly stronger form) as Theorem \ref{thm:73} in the 
body of the paper. See Remark \ref{rem:70} for an interpretation of this 
theorem.

Section \ref{sec:cm} of the paper is about {\em Cohen-Macaulay} DG modules.
The definition does not involve regular sequences of course; nor does it 
involve 
vanishing of local cohomologies as in \cite{RD} (even though it could probably 
be stated in this language). Instead we use a fact discovered in \cite{YZ3}: 
for 
a noetherian scheme $X$ with dualizing complex $\RR$, a complex 
$\MM \in \cat{D}^{\mrm{b}}_{\mrm{c}}(\cat{Mod} \OO_X)$ 
is CM (in the sense of \cite{RD}, for the dimension function determined by 
$\RR$) iff $\mrm{R} \Hom_{\OO_X}(\MM, \RR)$ 
is (isomorphic to) a coherent sheaf. 
In \cite{YZ3} the CM complexes inside 
$\cat{D}^{\mrm{b}}_{\mrm{c}}(\cat{Mod} \OO_X)$ 
were also called {\em perverse coherent sheaves}. 

With the explanation above, the next definition makes sense. Let $R$ be a 
dualizing DG $A$-module. A DG module $M \in \cat{D}^{\mrm{b}}_{\mrm{f}}(A)$
is called {\em CM with respect to $R$} if 
$\opn{RHom}_A(M, R) \in \cat{D}^{0}_{\mrm{f}}(A)$.
Here $\cat{D}^{0}_{\mrm{f}}(A)$ is the full subcategory of 
$\cat{D}(A)$ consisting of DG modules with finite cohomology concentrated in 
degree $0$; and we know that it is equivalent to the category 
$\cat{Mod}_{\mrm{f}} \bar{A}$ of finite $\bar{A}$-modules. 

The next theorem is repeated, in slightly greater generality, 
as Theorem \ref{thm:251}.

\begin{thm} \label{thm:61}
Let $f : A \to B$ be a homomorphism between tractable commutative DG rings,
such that $\opn{H}^0(f) : \bar{A} \to \bar{B}$
is surjective. Let $R_B$ be a dualizing DG $B$-module, and let  
$M, N \in \cat{D}^{\mrm{b}}_{\mrm{f}}(B)$. 
\begin{enumerate}
\item If $M$ is CM w.r.t.\ $R_B$, and there is an isomorphism
$\opn{rest}_f(M) \cong \opn{rest}_f(N)$ in $\cat{D}(A)$, then 
$N$ is also CM w.r.t.\ $R_B$. 

\item If $M$ and $N$ are both CM w.r.t.\ $R_B$, then the homomorphism 
\[ \opn{rest}_f : \opn{Hom}_{\cat{D}(B)}(M, N) \to 
\opn{Hom}_{\cat{D}(A)} \bigl( \opn{rest}_f(M), \opn{rest}_f(N) \bigr) \]
is bijective. 
\end{enumerate}
\end{thm}

In the theorem, $\opn{rest}_f : \cat{D}(B) \to \cat{D}(A)$ is the restriction 
functor. As already mentioned, Theorem \ref{thm:61} is needed in \cite{Ye7}.

In our paper \cite{Ye5} we introduce the {\em squaring operation} for 
commutative DG rings. The construction goes like this. The input is a 
homomorphism $A \to B$ of commutative DG rings. To this datum we associate a 
functor $\opn{Sq}_{B / A}$ from $\cat{D}(B)$ to itself. The formula is a bit 
more general than (\ref{eqn:100}): we choose any K-flat DG ring resolution 
$\til{A} \to \til{B}$ of $A \to B$ (see Section \ref{sec:rigid}), and we define 
\begin{equation} \label{eqn:505}
\opn{Sq}_{B / A}(M) 
:= \opn{RHom}_{\til{B} \ot_{\til{A}} \til{B}}(B, M \ot_{\til{A}}^{\mrm{L}} M)
\in \cat{D}(B) . 
\end{equation}
The proof that this definition does not depend on the resolution 
$\til{A} \to \til{B}$ is quite difficult. 

A {\em rigid DG module} over $B$ relative to $A$ is a pair $(M, \rho)$, 
consisting of a DG $B$-module $M$, and a {\em rigidifying isomorphism}
$\rho : M \iso \opn{Sq}_{B / A}(M)$
in $\cat{D}(B)$. There is a notion of rigid morphism between rigid DG modules
(Definition \ref{dfn:501}).  
Here are our results on rigid DG modules. The first is a DG version of
\cite[Theorem 0.2]{YZ1}, and it is repeated as Theorem \ref{thm:501}.

\begin{thm} \label{thm:505}
Let $A \to B$ be a homomorphism of commutative DG rings, and let 
$(M, \rho)$ be a rigid DG module over $B$ relative to $A$. Assume that the 
adjunction morphism $B \to \opn{RHom}_{B}(M, M)$ in $\cat{D}(B)$
is an isomorphism. Then the only rigid automorphism of $(M, \rho)$ is the 
identity. 
\end{thm}

Next is a DG version of the uniqueness in
\cite[Theorem 1.1]{YZ2}(1). It is repeated as Theorem \ref{thm:502}.

\begin{thm} \label{thm:506}
Let $A$ be a tractable  commutative DG ring, with traction $\K \to A$.  
Suppose  $(R, \rho)$ and $(R', \rho')$ are rigid dualizing DG modules over $A$ 
relative to $\K$. Then there is a unique rigid isomorphism 
$(R, \rho) \cong (R', \rho')$.
\end{thm}

Thus we can call $(R, \rho)$ {\em the} rigid dualizing DG module of $A$ 
relative to $\K$. 

\begin{conj} \label{conj:505}
In the situation of Theorem \tup{\ref{thm:506}}, the rigid dualizing 
DG module over $A$ relative to $\K$ exists.
\end{conj}

Indeed, we give more detailed conjectures in Section \ref{sec:rigid}, that 
would imply Conjecture \ref{conj:505}. When $A$ is a ring, these assertions 
were already proved in \cite{YZ1} (with some corrections in \cite{Ye7}).
In case $A$ has bounded cohomology, Conjecture \ref{conj:505} was very recently 
proved by Shaul \cite{Sh2}. 

We end the Introduction with another conjecture on rigid DG modules. 

\begin{conj} \label{conj:506}
In the situation of Theorem \tup{\ref{thm:506}}, let $(M, \rho)$
be a rigid DG module over $A$ relative to $\K$. Assume that 
$M \in \cat{D}^+_{\mrm{f}}(A)$, and that $M$ is nonzero on each connected 
component of $\opn{Spec} \bar{A}$. Then $M$ is a dualizing DG $A$-module. 
\end{conj}

When $A$ is a ring and 
$M \in \cat{D}^{\mrm{b}}_{\mrm{f}}(A)$, this was proved in \cite{YZ2} 
and \cite{AIL}. More on this conjecture in Remark \ref{rem:520}.

\medskip \noindent
{\bf Acknowledgments}. 
I wish to thank Peter J{\o}rgensen, Bernhard Keller, Vladimir Hinich,
Liran Shaul, James Zhang, Amnon Neeman, Dennis Gaitsgory, 
Rishi Vyas, Matan Prasma, Jacob Lurie, Benjamin Antieau, John Palmieri, 
Pieter Belmans, \lb Michel Vaqui\'e, Joseph Lipman, Srikanth Iyengar and Cris 
Negron for helpful discussions.

\section{DG Modules and their Resolutions} \label{sec:resol}
 
A {\em DG ring} (usually called an associative unital DG algebra over $\Z$) is 
a graded ring $A = \bigoplus_{i \in \Z} A^i$, with differential $\d$ of degree 
$1$, satisfying the graded Leibniz rule 
\[ \d(a \cd b) = \d(a) \cd b + (-1)^i \cd a \cd \d(b) \]
for $a \in A^i$ and $b \in A^j$. 
A homomorphism of DG rings is a degree $0$ ring homomorphism 
that commutes with the differentials. 
In this way we get a category, denoted by $\cat{DGR}$. 
Rings are viewed as DG rings concentrated in degree $0$. 
For a DG ring $A$, the cohomology 
$\opn{H}(A) = \bigoplus_{i \in \Z} \opn{H}^i(A)$ is a graded ring. 

The {\em opposite} of the DG ring $A$ is the DG ring $A^{\mrm{op}}$,
which is the same graded abelian group as $A$, with the same differential, but 
the multiplication $\cd^{\mrm{op}}$ is reversed in the graded sense. Namely, 
for 
elements $a \in A^i$ and $b \in A^j$, their product in $A^{\mrm{op}}$ is 
\[ a \, \cd^{\mrm{op}} \, b := (-1)^{i j} \cd b \cd a . \]

A left DG $A$-module is a graded left $A$-module $M = \bigoplus_{i \in \Z} 
M^i$, with differential $\d$ satisfying the graded Leibniz rule. 
If $M$ is a left DG $A$-module then 
$\opn{H}(M) = \bigoplus_{i \in \Z} \opn{H}^i(M)$
is a graded left $\opn{H}(A)$-module.
Note that a right DG $A$-module is the same as a left DG $A^{\mrm{op}}$-module.

\begin{conv} \label{conv:585}
By default, in this paper all DG modules are left DG modules.
\end{conv} 

From Section \ref{sec:localiz} onwards our DG rings will be commutative, 
and for them the distinction between left and right DG modules becomes 
negligible. 

\begin{dfn} \label{dfn:365}
Let $A$ be a DG ring. 
The category of (left) DG $A$-modules, with $A$-linear homomorphisms of degree 
$0$ that commute with differentials, is denoted by $\cat{DGMod} A$,
or by its abbreviation $\cat{C}(A)$. 
The derived category, gotten from $\cat{DGMod} A$ by inverting 
quasi-isomorphisms, is denoted by 
$\cat{D}(\cat{DGMod} A)$, or by the abbreviation $\cat{D}(A)$.
\end{dfn}

For information on $\cat{D}(A)$ see \cite[Section 10]{BL}, 
\cite[Section 2]{Ke1} or 
\cite[\href{http://stacks.math.columbia.edu/tag/09KV}{Section 09KV}]{SP}.
If $A$ is a ring, then 
$\cat{C}(A) = \cat{C}(\cat{Mod} A)$,
the category of complexes of $A$-modules; and 
$\cat{D}(A) = \cat{D}(\cat{Mod} A)$,
the usual derived category of the abelian category $\cat{Mod} A$. 

\begin{dfn} \label{dfn:105}
Let $A = \bigoplus_{i \in \Z} A^i$ be a DG ring.
\begin{enumerate}
\item We say that $A$ is  {\em a nonpositive DG ring} if 
$A^i = 0$ for all $i > 0$.

\item If $A$ is nonpositive DG ring, then we write 
$\bar{A} := \opn{H}^0(A)$, which is a ring. There is a canonical DG ring 
homomorphism $A \to \bar{A}$.
\end{enumerate}
\end{dfn}

The full subcategory of $\cat{DGR}$ on the nonpositive DG rings is 
denoted by $\cat{DGR}^{\leq 0}$.

\begin{conv} \label{conv:305}
By default, in this paper all DG rings are nonpositive; i.e.\ we work 
inside $\cat{DGR}^{\leq 0}$. 
\end{conv}

One of the important advantages of nonpositive DG rings
is that the differential $\d$ of any DG $A$-module $M$ is $A^0$-linear. This 
implies that the two smart truncations
\begin{equation} \label{eqn:425}
\opn{smt}^{\geq i}(M) := \bigl( \cdots \to 0 \to \opn{Coker}(\d|_{M^{i - 1}})
\to M^{i + 1} \to M^{i + 2} \to  \cdots \bigr)
\end{equation}
and
\begin{equation} \label{eqn:426}
\opn{smt}^{\leq i}(M) := \bigl( \cdots \to  M^{i - 2} \to  M^{i - 1} \to
\opn{Ker}(\d|_{M^{i}}) \to 0 \to  \cdots \bigr)
\end{equation}
remain within $\cat{C}(A)$; and there are functorial homomorphisms 
$M \to \opn{smt}^{\geq i}(M)$ and 
$\opn{smt}^{\leq i}(M) \to M$ in $\cat{C}(A)$, inducing isomorphisms in 
$\opn{H}^{\geq i}$ and $\opn{H}^{\leq i}$ respectively.
Note that these are the truncations $\tau_{\leq n}$ and $\tau_{\geq n}$ from 
\cite[\href{http://stacks.math.columbia.edu/tag/0118}{Section 0118}]{SP},
that are variants of the truncations $\si_{> n}$ and $\si_{\leq n}$ from 
\cite[Section I.7, page 69]{RD}.
Warning: the two stupid truncations might fail to work in this context -- these 
truncated complexes of abelian groups might not be DG $A$-modules. 
 
Recall that for a subset $S \subseteq \Z$, its infimum is
$\opn{inf}(S) \in \Z \cup \{ \pm \infty \}$,
where $\opn{inf}(S) = + \infty$ iff $S = \varnothing$. 
Likewise the supremum is $\opn{sup}(S) \in \Z \cup \{ \pm \infty \}$,
where $\opn{sup}(S) = - \infty$ iff $S = \varnothing$.
For $i, j \in \Z \cup \{ \infty \}$, the expressions 
$i + j$ and $-i - j$ have obvious values in 
$\Z \cup \{ \pm \infty \}$. And for 
$i, j \in \Z \cup \{ \pm \infty \}$, 
the expression $i \leq j$ has an obvious meaning.

Let $M = \bigoplus_{i \in \Z} M^i$ be a graded abelian group. We write  
\begin{equation} \label{eqn:371}
\inf (M) := \opn{inf}\, \{ i \mid M^i \neq 0 \} \quad \tup{and} \quad 
\sup (M) := \opn{sup}\, \{ i \mid  M^i \neq 0 \} .
\end{equation}
The amplitude of $M$ is 
\begin{equation} \label{eqn:372}
\opn{amp} (M) := \sup (M) - \inf (M) \in \N \cup \{ \pm \infty\} .
\end{equation}
(For $M = 0$ this reads $\inf (M) = \infty$, $\sup (M) = -\infty$ 
and $\opn{amp} (M) = -\infty$.)
Thus $M$ is bounded (resp.\ bounded above, resp.\ bounded below) iff 
$\opn{amp} (M) < \infty$ (resp.\ $\opn{sup} (M) < \infty$, resp.\
$\opn{inf} (M) > -\infty$).

Given $i_0 \leq i_1$ in $\Z \cup \{ \pm \infty \}$,
the {\em integer interval} with these endpoints is the set of integers
\begin{equation} \label{eqn:580}
[i_0, i_1] := \{ i \in \Z  \mid i_0 \leq i \leq i_1 \}  .
\end{equation}
The integer interval $[i_0, i_1]$ is said to be bounded (resp.\ bounded above, 
resp.\ \lb bounded below) if $i_0, i_1 \in \Z$ (resp.\  $i_1 \in \Z$, resp.\ 
$i_0 \in \Z$). 
The {\em length} of this interval is $i_1 - i_0 \in \N \cup \{ \infty \}$. 
Of course the interval has finite length iff it is bounded. 
We write $-[i_0, i_1] := [-i_1, -i_0]$.
Given a second integer interval $[j_0, j_1]$, we let 
\[ [i_0, i_1] + [j_0, j_1] := [i_0 + j_0, i_1 + j_1] . \]
For the empty interval $\varnothing$, the sum is 
$[i_0, i_1] + \varnothing := \varnothing$.

\begin{dfn}  \label{dfn:550}
Let $M = \bigoplus_{i \in \Z} M^i$ be a graded abelian group.
\begin{enumerate}
\item We say that $M$ is concentrated in an integer interval $[i_0, i_1]$
if 
\[  \{ i \in \Z \mid M^i \neq 0 \} \subseteq [i_0, i_1] . \]

\item The {\em concentration} of $M$ is the smallest integer interval 
$\opn{con}(M)$ in which $M$ is concentrated.
\end{enumerate}
\end{dfn}

In other words, if 
$i_0 = \opn{inf} (M) \leq i_1 = \opn{sup} (M)$, 
then the concentration of $M$ is the interval
$\opn{con}(M) = [i_0, i_1]$,
and the amplitude $\opn{amp}(M)$ is the length of $\opn{con}(M)$.
Furthermore, $\opn{con}(M) = \varnothing$ iff $M = 0$. 

Given an integer interval $[i_0, i_1]$, we 
denote by $\cat{D}^{[i_0, i_1]}(A)$ the full subcategory of $\cat{D}(A)$ 
consisting of the DG modules $M$ whose cohomologies $\opn{H}(M)$ are 
concentrated in this interval; namely 
$\opn{con}(\opn{H}(M)) \subseteq [i_0, i_1]$.
For $i \in \Z$ we write 
$\cat{D}^{i}(A) := \cat{D}^{[i, i]}(A)$.
The subcategory $\cat{D}^{[i_0, i_1]}(A)$ is additive, but not triangulated. 
Similarly we have the subcategory 
$\cat{C}^{[i_0, i_1]}(A) \subseteq \cat{C}(A)$,
consisting of the DG modules $M$ that are concentrated in the integer interval
$[i_0, i_1]$; namely $\opn{con}(M) \subseteq [i_0, i_1]$.

\begin{dfn} \label{dfn:370}
Let $A$ be a DG ring. 
\begin{enumerate}
\item A DG $A$-module $M$ is said to be {\em cohomologically bounded} 
(resp.\  {\em cohomologically bounded above}, 
resp.\  {\em cohomologically bounded below})
if the  graded module $\opn{H}(M)$ is bounded (resp.\  bounded above, 
resp.\ bounded below).
We denote by $\cat{D}^{\mrm{b}}(A)$, $\cat{D}^{-}(A)$ and $\cat{D}^{+}(A)$
the corresponding full subcategories of $\cat{D}(A)$.

\item The full subcategory of $\cat{D}(A)$ consisting of the DG modules $M$,
whose cohomology modules $\opn{H}^i(M)$ are finite over $\bar{A}$,  is denoted 
by $\cat{D}_{\mrm{f}}(A)$.

\item  For any boundedness condition $\star$ we write 
$\cat{D}^{\star}_{\mrm{f}}(A) := \cat{D}_{\mrm{f}}(A) \cap \cat{D}^{\star}(A)$.
\end{enumerate}
\end{dfn}

Thus we have 
\[ \cat{D}^{\mrm{b}}(A) = \bigcup_{-\infty < i_0 \leq i_1 < \infty }
\cat{D}^{[i_0, i_1]}(A) \ , \]
etc. The categories $\cat{D}^{\mrm{b}}(A)$, $\cat{D}^{-}(A)$ and 
$\cat{D}^{+}(A)$ are triangulated. If $\bar{A}$ is left 
noetherian, then the categories 
$\cat{D}^{\mrm{b}}_{\mrm{f}}(A)$, $\cat{D}^{-}_{\mrm{f}}(A)$ and 
$\cat{D}^{+}_{\mrm{f}}(A)$ are also triangulated.

Given a DG $A$-module $M$, its shift by an integer $i$ is the DG module 
$M[i]$, whose \lb $j$-th graded component is $M[i]^j := M^{i+j}$.
Elements of $M[i]^j$ are denoted by $m[i]$, with $m \in M^{i+j}$.
The differential of $M[i]$ is 
$\d_{M[i]}(m[i]) := (-1)^i \cd \d_M(m)[i]$.
The left action of $A$ on $M[i]$ is also twisted by $\pm 1$, as follows:
$a \cd m[i] := (-1)^{k i} \cd (a \cd m)[i]$ 
for $a \in A^k$. The right action remains untwisted:
$m[i] \cd a := (m \cd a)[i]$.
See \cite[Section 1]{Ye5} for a detailed study of the shift operation, 
including an explanation of the sign that appears in the left action. 

We now recall some resolutions of DG $A$-modules. 
A DG module $N$ is called acyclic if $\opn{H}(N) = 0$. 
A DG $A$-module $M$ is called {\em K-projective} (resp.\ {\em 
K-injective}), if for any acyclic DG $A$-module $N$, the DG 
$\Z$-module $\opn{Hom}_A(M, N)$ (resp.\ $\opn{Hom}_A(N, M)$) is also 
acyclic. The DG $A$-module $M$ is called {\em K-flat} if for any acyclic DG 
$A^{\mrm{op}}$-module $N$, the DG $\Z$-module $N \ot_A M$ is acyclic. 
It is easy to see that K-projective implies K-flat. 
More information about the operations 
$\opn{Hom}_A(-, -)$ and $- \ot_A -$, 
and the various resolutions, see \cite[Section 1]{Ye5}.

For a cardinal number $r$ (possibly infinite) we denote by $M^{\oplus r}$ the 
direct sum of $r$ copies of $M$.
Recall that a DG $A$-module $P$ is a {\em free DG module} if 
$P \cong \bigoplus_{i \in \Z} A[-i]^{\oplus r_i}$, where 
$r_i$ are cardinal numbers.
We say that $P$ is a {\em finite free DG module} if $\sum_i r_i < \infty$
(for some such isomorphism).  
 
\begin{dfn} \label{dfn:310}
Let $P$ be a DG $A$-module. 
A {\em semi-free filtration} of $P$ is an ascending filtration 
$\{ \nu_j(P) \}_{j \in \Z}$ by DG $A$-submodules
$\nu_j(P) \subseteq P$, such that 
$\nu_{-1}(P) = 0$, $P = \bigcup_j \nu_j(P)$, and each 
$\opn{gr}^\nu_j(P) :=  \nu_j(P) / \nu_{j-1}(P)$ is a free DG $A$-module.
\end{dfn}

\begin{dfn} \label{dfn:100}
Let $P$ be a DG $A$-module, with semi-free filtration 
$\{ \nu_j(P) \}_{j \in \Z}$.
\begin{enumerate}
\item The filtration $\{ \nu_j(P) \}_{j \in \Z}$ is said to have length $l$
if 
\[ l = \opn{inf}\, \{ j \in \N \mid  \nu_{j}(P) = P  \}
\in \N \cup \{ \infty \}  . \]

\item The filtration $\{ \nu_j(P) \}_{j \in \Z}$ is called {\em pseudo-finite} 
if each free DG $A$-module $\opn{gr}^\nu_j(P)$ is finite, and 
$\lim_{j \to  \infty} \sup (\opn{gr}^\nu_j(P)) = -\infty$.

\item The filtration $\{ \nu_j(P) \}_{j \in \Z}$ on $P$ is called 
{\em finite} if it is pseudo-finite and has finite length. 
\end{enumerate}
\end{dfn}

\begin{dfn} \label{dfn:311}
A DG $A$-module $P$ is called a {\em semi-free} (resp.\ {\em pseudo-finite 
semi-free}, resp.\ {\em finite semi-free}) {\em DG module} if it admits a 
semi-free (resp.\ pseudo-finite semi-free, resp.\  finite semi-free)  
filtration. The {\em semi-free length} of $P$ is defined to be the minimum of 
the lengths of its semi-free filtrations.  
\end{dfn}

Note that a semi-free DG module need not be bounded above.
If $P$ is a semi-free DG module then it is K-projective; see
\cite[Section 3.1]{Ke1}.

Recall that our DG rings are always nonpositive. 
The next proposition gives another characterization of pseudo-finite semi-free
DG modules. The graded ring gotten from $A$ by forgetting the differential is 
denoted by $A^{\natural}$. Likewise for DG modules.

\begin{prop} \label{prop:100}
Let $P$ be a DG $A$-module. 
\begin{enumerate}
\item $P$ is pseudo-finite semi-free iff there are $i_1 \in \Z$
and $r_i \in \N$, such that 
$P^{\natural} \cong \lb \bigoplus_{i \leq i_1} A^{\natural}[-i]^{\oplus r_i}$
as graded $A^{\natural}$-modules. 

\item $P$ is finite semi-free iff there is an isomorphism of 
graded $A^{\natural}$-modules as in item \tup{(1)} above, and $i_0 \in \Z$,
such that $r_i = 0$ for all $i < i_0$.
\end{enumerate}
\end{prop} 

\begin{proof}
Given an isomorphism 
$P^{\natural} \cong \bigoplus_{i \leq i_1} A^{\natural}[-i]^{\oplus r_i}$,
define  
\[ \nu_j(P) := \bigoplus_{i_1 - j \leq i \leq i_1} A[-i]^{\oplus r_i} 
\subseteq P . \]
This is a pseudo-finite semi-free filtration, of length $\leq i_1 - i_0$
in the finite case. The converse is clear, and so is item (2). 
\end{proof}

\begin{dfn} \label{dfn:390} 
Let $[i_0, i_1]$ be an integer interval (possibly unbounded). 
\begin{enumerate}
\item Let $P$ be a free graded $A^{\natural}$-module. We say that $P$ has a
{\em basis concentrated in $[i_0, i_1]$}  
if there is an isomorphism of graded 
$A^{\natural}$-modules
\[ P \cong \bigoplus_{i \in [i_0, i_1]} A^{\natural}[-i]^{\oplus r_i} \]
for some cardinal numbers $r_i$.

\item A DG $A$-module $M$ is said to be {\em generated} in the integer interval
$[i_0, i_1]$ if there is a surjection of graded $A^{\natural}$-modules
$P \to M^{\natural}$, where $P$ is a free graded $A^{\natural}$-module
with a basis concentrated in $[i_0, i_1]$.
\end{enumerate}
\end{dfn}

\begin{prop} \label{prop:390} 
Let $M$ be a DG $A$-module generated in the integer interval $[i_0, i_1]$.
\begin{enumerate}
\item If $N \in \cat{C}^{[j_0, j_1]}(A^{\mrm{op}})$,
then 
\[ N \ot_A M \in \cat{C}^{[j_0, j_1] + [i_0, i_1]}(\Z)
= \cat{C}^{[j_0 + i_0, j_1 + i_1]}(\Z). \]

\item If $N \in \cat{C}^{[j_0, j_1]}(A)$,
then 
\[ \opn{Hom}_A(M, N) \in \cat{C}^{[j_0, j_1] - [i_0, i_1]}(\Z)
= \cat{C}^{[j_0 - i_1, j_1 - i_0]}(\Z). \]
\end{enumerate}
\end{prop}

The easy proof is left to the reader. 

\begin{dfn} \label{dfn:391}
Let $A$ be a DG ring.
\begin{enumerate}
\item $A$ is called {\em cohomologically left noetherian} if the graded ring 
$\opn{H}(A)$ is left noetherian.

\item $A$ is called {\em cohomologically left pseudo-noetherian} if it 
satisfies these two conditions:
\begin{enumerate}
\rmitem{i} The ring  $\bar{A} = \opn{H}^0(A)$ is left noetherian.

\rmitem{ii} For every $i$ the left $\bar{A}$-module $\opn{H}^i(A)$ is finite 
(i.e.\ finitely generated). 
\end{enumerate}
\end{enumerate} 
\end{dfn}

We shall be mostly interested in  cohomologically left pseudo-noetherian DG 
rings. Of course if $A$ is cohomologically left noetherian, then it is 
cohomologically left pseudo-noetherian. These conditions are equivalent when 
$A$ is cohomologically bounded.

As usual, by a semi-free resolution of a DG module $M$ we mean a 
quasi-iso\-mor\-phism $P \to M$ in $\cat{C}(A)$, where $P$ is semi-free. 
Likewise we talk about K-injective resolutions $M \to I$.

\begin{prop} \label{prop:101}
Let $M$ be a DG $A$-module. 
\begin{enumerate}
\item There is a semi-free resolution $P \to M$ such that 
$\opn{sup}(P) = \opn{sup}(\opn{H}(M))$.

\item If the DG ring $A$ is cohomologically left pseudo-noetherian, and if
$M \in \cat{D}^{-}_{\mrm{f}}(A)$, then there is a pseudo-finite semi-free 
resolution $P \to M$ such that $\opn{sup}(P) = \opn{sup}(\opn{H}(M))$.

\item There is a K-injective resolution $M \to I$ such that 
$\opn{inf}(I) = \opn{inf}(\opn{H}(M))$.
\end{enumerate}
\end{prop}

\begin{proof}
(1) See \cite[Theorem 3.1]{Ke1}, \cite{AFH}, or 
\cite[\href{http://stacks.math.columbia.edu/tag/09KK}{Section 09KK}]{SP},
noting that \lb $\opn{sup}(A) = 0$ (if $A$ is nonzero). 

\medskip \noindent
(2) In this case the construction in item (1) can be made 
with finitely many basis elements in each degree. 
See \cite[Theorem 3.1]{Ke1} or \cite{AFH}.

\medskip \noindent
(3) See \cite[Theorem 3.2]{Ke1}, \cite{AFH}, or 
\cite[\href{http://stacks.math.columbia.edu/tag/09KQ}{Section 09KQ}]{SP}.
The point is that any DG $A$-module can be embedded in a product of shifts of 
the DG $A$-module \lb $\opn{Hom}_{\Z}(A, \Q / \Z)$. 
\end{proof}

\begin{rem} \label{rem:90}
Suppose $A$ is a ring. A DG $A$-module $P$ is pseudo-finite semi-free iff it is 
a bounded above complex of finite free $A$-modules. Now according to 
\cite{SGA-6} or 
\cite[\href{http://stacks.math.columbia.edu/tag/064Q}{Definition 064Q}]{SP},
a DG $A$-module $M$ is called pseudo-coherent if it is quasi-isomorphic to
some pseudo-finite semi-free DG module $P$. This explains the choice of the 
name ``pseudo-finite semi-free DG module''.

The name ``pseudo-noetherian DG ring'' was chosen due to the close 
relation to pseudo-finite semi-free resolutions; see Proposition 
\ref{prop:101}(2).
\end{rem}

\section{Cohomological Dimension} \label{sec:coh-dim}

We continue with the conventions of Section \ref{sec:resol}, namely our DG 
rings are nonpositive, and the DG modules are acted upon from the left. 
The concentration $\opn{con}(M)$ of a graded module $M$ was 
introduced in Definition \ref{dfn:550}.

\begin{dfn} \label{dfn:10}
Let $A$ and $B$ be DG rings, and let
$\cat{E} \subseteq \cat{D}(A)$ be a full subcategory.
\begin{enumerate}
\item Let  $F : \cat{E} \to \cat{D}(B)$ be an additive functor,
and let $[d_0, d_1]$ be an integer interval. We say that {\em $F$ has 
cohomological displacement at most $[d_0, d_1]$} if 
\[ \opn{con} \bigl( \opn{H}(F(M)) \bigr) \subseteq
\opn{con} \bigl( \opn{H}(M) \bigr) + [ d_0, d_1] \]
for every $M \in \cat{E}$.

\item Let $F : \cat{E}^{\mrm{op}} \to \cat{D}(B)$ be an additive functor,
and let $[d_0, d_1]$ be an integer interval. We say that {\em $F$ has 
cohomological displacement at most $[d_0, d_1]$} if 
\[ \opn{con} \bigl( \opn{H}(F(M)) \bigr) \subseteq
- \opn{con} \bigl( \opn{H}(M) \bigr) + [ d_0, d_1] \]
for every $M \in \cat{E}$.

\item Let $F$ be as in item (1) or (2).
The {\em cohomological displacement of $F$} is the smallest integer interval 
$[d_0, d_1]$ for which $F$ has cohomological displacement at most $[d_0, d_1]$.
If $d_0 \in \Z$ (resp.\ $d_1 \in \Z$, resp.\ $d_0, d_1 \in \Z$) then $F$ is 
said to have {\em bounded below} (resp.\ {\em bounded above}, 
resp.\ {\em bounded}) {\em cohomological displacement}.

\item Let $[d_0, d_1]$ be the cohomological displacement of $F$. 
The {\em cohomological dimension} of $F$ is 
$d := d_1 - d_0 \in \N \cup \{ \infty \}$. If $d \in \N$, then $F$ is said 
to have {\em finite cohomological dimension}.
\end{enumerate}
\end{dfn}

Note that if $\cat{E}' \subseteq \cat{E}$, and $F$ has 
cohomological displacement at most $[d_0, d_1]$, then $F|_{\cat{E}'}$ also 
has cohomological displacement at most $[d_0, d_1]$.

\begin{exa} \label{exa:391}
Consider a commutative ring $A = B$, and let $\cat{E} := \cat{D}^{}(A)$.
For the covariant case (item (1) in Definition \ref{dfn:10}) take a nonzero 
projective module $P$, and 
let $F := \opn{Hom}_A(P \oplus P[1], -)$. 
Then $F$ has cohomological 
displacement $[0, 1]$. 
For the contravariant case (item (2)) take a nonzero 
injective module $I$, and let 
$F := \lb \opn{Hom}_A(-, I \oplus I[1])$. 
Then $F$ has cohomological displacement $[-1, 0]$. 
In both cases the cohomological dimension of $F$ is $1$. 
\end{exa}

\begin{rem}
Suppose $A$ and $B$ are rings. If $\cat{E} = \cat{D}^+(A)$ and 
$F = \mrm{R} F_0$ (or $\cat{E} = \cat{D}^-(A)$ and $F = \mrm{L} F_0$) for some 
additive functor $F_0 : \cat{Mod} A \to \cat{Mod} B$, then the cohomological 
dimension of $F$ is the usual cohomological dimension of $F_0$. 

Assume that $\cat{E} = \cat{D}(A)$ and $F$ is a triangulated 
functor. The functor $F$ has bounded below (resp.\ above) cohomological 
displacement iff it is way-out right (resp.\ left), in the sense of 
\cite[Section I.7]{RD}.
\end{rem}

\begin{dfn} \label{dfn:11}
Let $A$ be a DG ring, let $M \in \cat{D}(A)$, and let $[d_0, d_1]$ be an integer
interval of length $d := d_1 - d_0$.
\begin{enumerate}
\item Given a full subcategory $\cat{E} \subseteq \cat{D}(A)$, 
we say that $M$ has {\em projective concentration $[d_0, d_1]$
and projective dimension $d$ relative to $\cat{E}$} if the functor
\[ \opn{RHom}_A(M, -) : \cat{E} \to \cat{D}(\Z) \]
has cohomological displacement $-[d_0, d_1]$ relative to $\cat{E}$.

\item Given a full subcategory $\cat{E} \subseteq \cat{D}(A)$, 
we say that $M$ has {\em injective concentration $[d_0, d_1]$ and 
injective dimension $d$ relative to $\cat{E}$} if the functor
\[ \opn{RHom}_A(-, M) : \cat{E}^{\mrm{op}} \to \cat{D}(\Z) \]
has  cohomological displacement $[d_0, d_1]$ relative to $\cat{E}$.

\item Given a full subcategory $\cat{E} \subseteq \cat{D}(A^{\mrm{op}})$, 
we say that $M$ has {\em flat concentration $[d_0, d_1]$  
and flat dimension $d$ relative to $\cat{E}$} if the functor
\[ - \ot^{\mrm{L}}_A M : \cat{E} \to \cat{D}(\Z) \]
has cohomological displacement $[d_0, d_1]$ relative to $\cat{E}$.
\end{enumerate}
\end{dfn}

\begin{exa} \label{exa:390}
Continuing with the setup of Example \ref{exa:391}, the DG module
$P \oplus P[1]$ (resp.\ $I \oplus I[1]$) has projective (resp.\ injective)
concentration $[-1, 0]$ relative to $\cat{D}(A)$. 
\end{exa}

\begin{exa} \label{exa:520}
Let $A$ be a DG ring, and consider the free DG module 
$P := A \in \cat{D}(A)$.
The functor 
\[ F := \opn{RHom}_{A}(P, -) : \cat{D}(A) \to \cat{D}(\Z) \]
is isomorphic to the forgetful functor 
$\opn{rest}_{A / \Z}$, so it has cohomological displacement 
$[0, 0]$ and cohomological dimension $0$ relative to $\cat{D}(A)$. Thus the DG 
module $P$ has projective concentration $[0, 0]$ and projective dimension $0$
relative to $\cat{D}(A)$.
Note however that the cohomology $\opn{H}(P)$ could be unbounded below. 
\end{exa}

The interval of generation of a DG module was introduced in Definition 
\ref{dfn:390}. 

\begin{prop} \label{prop:393}
Let $M \in \cat{D}(A)$. Assume there is an isomorphism 
$P \cong M$ in $\cat{D}(A)$, where $P$ is a K-flat \tup{(}resp.\ 
K-projective\tup{)} DG $A$-module generated in the integer interval 
$[d_0, d_1]$. Then  $M$ has flat \tup{(}resp.\ projective\tup{)} concentration 
at most $[d_0, d_1]$ relative to  $\cat{D}(A^{\mrm{op}})$
\tup{(}resp.\ $\cat{D}(A)$\tup{)}. 
\end{prop}

\begin{proof}
First consider the K-flat case. Take any $N \in \cat{D}(A^{\mrm{op}})$. 
After applying smart truncation, we can assume that 
$\opn{con}(N) = \opn{con}(\opn{H}(N))$.
Now  $N \ot_A^{\mrm{L}} M \cong N \ot_A P$, 
and by Proposition \ref{prop:390} the bounds on 
$N \ot_A P$ are as claimed. 

Next consider the K-projective case. Take any $N \in \cat{D}(A)$. As above, 
we can assume that 
$\opn{con}(N) = \opn{con}(\opn{H}(N))$.
We know that 
$\opn{RHom}_A(M, N) \cong \opn{Hom}_A(P, N)$.
The bounds on the DG module 
$\opn{Hom}_A(P, N)$ are as claimed, by Proposition \ref{prop:390}. 
\end{proof}

\begin{prop} \label{prop:35}
If $M \in \cat{D}^{[d_0, d_1]}(A)$, then $M$ has projective concentration at 
most  \lb $[-\infty, d_1]$ relative to $\cat{D}(A)$, 
injective concentration at most $[d_0, \infty]$  relative to $\cat{D}(A)$, 
and flat concentration at most  $[-\infty, d_1]$ relative to 
$\cat{D}(A^{\mrm{op}})$.
\end{prop}

\begin{proof}
First let's assume that  $M \neq 0$ and $d_1 < \infty$.
We know that $M$ admits a semi-free resolution $P \to M$ with 
$\opn{sup}(P) = \opn{sup}(\opn{H}(M)) \leq d_1$. 
Now we can use Proposition \ref{prop:393} for the flat and projective 
concentrations.

Now let's assume that $M \neq 0$ and $d_0 > -\infty$.
By Proposition \ref{prop:101} there is a K-injective 
resolution $M \to I$ with 
$\opn{inf}(I) = \opn{inf}(\opn{H}(M)) \geq d_0$. For any 
$N \in \cat{D}(A)$ we have 
$\opn{RHom}_A(N, M) \cong \opn{Hom}_A(N, I)$, and hence the bound on the 
injective concentration. 
\end{proof}

\begin{rem} \label{rem:50}
Let us write <att> for either of the attributes projective, 
injective or flat.
When $A$ is a ring,
$\cat{E} = \cat{D}^{0}(A) \approx \cat{Mod} A$
and $M \in \cat{Mod} A$,
we recover the usual definition of <att> dimension of modules in ring theory.
Furthermore, in the ring case, a DG $A$-module $M$ has <att> concentration in a 
bounded integer interval $[d_0, d_1]$ relative to $\cat{D}^{0}(A)$, iff $M$
isomorphic in $\cat{D}(A)$ to a complex $P$ of <att> $A$-modules with 
$\opn{con}(P) \subseteq [d_0, d_1]$. 
This implies that $M$ has finite <att> dimension relative to $\cat{D}(A)$. 
We do not know if this -- namely the <att> dimension of $M$ relative to 
$\cat{D}(A)$ is the same as that relative $\cat{D}^{0}(A)$ -- is true when $A$ 
is a genuine DG ring. 
\end{rem}

The next two theorems are variations of the opposite (in the categorical sense) 
of \cite[Proposition I.7.1]{RD}, the ``Lemma on Way-Out Functors''.
The canonical homomorphism $A \to \bar{A}$ lets us 
view any $\bar{A}$-module as a DG $A$-module.

\begin{thm} \label{thm:70}
Let $A$ and $B$ be DG rings, let 
$F, G : \cat{D}(A) \to \cat{D}(B)$ 
be triangulated functors, and let $\eta : F \to G$ be a morphism of 
triangulated functors. Assume that $\eta_M : F(M) \to G(M)$ is an 
isomorphism for every $M \in \cat{Mod} \bar{A}$. 
\begin{enumerate}
\item The morphism $\eta_M$ is an isomorphism for every 
$M \in \cat{D}^{\mrm{b}}(A)$.

\item If $F$ and $G$ have bounded above cohomological displacements, then 
$\eta_M$ is an isomorphism for every $M \in \cat{D}^{-}(A)$.

\item If $F$ and $G$ have finite cohomological dimensions, then 
$\eta_M$ is an isomorphism for every $M \in \cat{D}(A)$.
\end{enumerate}
\end{thm}

\begin{proof}
(1) The proof is by induction on $j := \opn{amp}(\opn{H}(M))$. 
If $j = 0$ then $M$ is isomorphic to a shift an object of $\bar{A}$, so 
$\eta_M$ is an isomorphism. If $j > 0$, then using smart truncation 
we obtain a distinguished triangle
$M' \to M \to M'' \xar{\vartriangle}$ in $\cat{D}(A)$ such that  
$\opn{amp}(\opn{H}(M'')) < j$ and $\opn{amp}(\opn{H}(M')) < j$. 
Since $\eta_{M''}$ and $\eta_{M'}$ are isomorphisms, so is 
$\eta_{M}$.

\medskip \noindent
(2) Here we assume that $F$ and $G$ have cohomological displacements at 
most \lb $[-\infty, d_1]$ for some integer $d_1$. 
Take any $M \in \cat{D}^{-}(A)$. 
In order to prove that $\eta_M$ is an isomorphism it suffices to show that 
\[ \opn{H}^i(\eta_M) : \opn{H}^i(F(M)) \to \opn{H}^i(G(M)) \]
is bijective for every $i \in \Z$. 

Fix an integer $i$. Let $M' \to M \to M'' \xar{\vartriangle}$ be a 
distinguished triangle such that 
$\opn{sup}(\opn{H}(M')) \leq i - d_1 - 2$
and 
$\opn{inf}(\opn{H}(M'')) \geq i - d_1 - 1$. This can be obtained using smart 
truncation.

The  cohomologies of $F(M')$ and $G(M')$ are concentrated in the degree range 
$\leq i - 2$. The distinguished triangle induces a commutative 
diagram of $\bar{A}$-modules with exact rows:
\[ \UseTips \xymatrix @C=5ex @R=5ex {
\opn{H}^{i}(F(M'))
\ar[r]
\ar[d]_{ \opn{H}^{i}(\eta_{M'}) }
&
\opn{H}^{i}(F(M))
\ar[r]
\ar[d]_{ \opn{H}^{i}(\eta_{M}) }
&
\opn{H}^{i}(F(M''))
\ar[r]
\ar[d]_{ \opn{H}^{i}(\eta_{M''}) }
&
\opn{H}^{i+1}(F(M'))
\ar[d]_{ \opn{H}^{i+1}(\eta_{M'}) }
\\
\opn{H}^{i}(G(M'))
\ar[r]
&
\opn{H}^{i}(G(M))
\ar[r]
&
\opn{H}^{i}(G(M''))
\ar[r]
&
\opn{H}^{i+1}(G(M')) \ .
} \]
The four terms involving $M'$ are zero. Since $M''$ has bounded cohomology, we 
know by part (1) that $\opn{H}^{i}(\eta_{M''})$ is an isomorphism. Therefore 
$\opn{H}^{i}(\eta_{M})$ is an isomorphism.

\medskip \noindent
(3) Here we assume that $F$ and $G$ have finite cohomological dimensions. 
So $F$ and $G$ have cohomological displacements at most $[d_0, d_1]$ for some 
$d_0 \leq d_1$ in $\Z$.
Take any $M \in \cat{D}(A)$, and fix $i \in \Z$. We want to show that 
$\opn{H}^i(\eta_M)$ is an isomorphism. 
Using smart truncations of $M$ we obtain a distinguished triangle
$M' \to M \to M'' \xar{\vartriangle}$ in $\cat{D}(A)$, such that
$\opn{sup}(\opn{H}(M')) \leq i - d_0 + 1$
and 
$\opn{inf}(\opn{H}(M'')) \geq i - d_0 + 2$.
The cohomologies of $F(M'')$ and $G(M'')$ are concentrated in degrees 
$\geq i + 2$. 
We have a commutative diagram of $\bar{A}$-modules with exact rows:
\[ \UseTips \xymatrix @C=5ex @R=5ex {
\opn{H}^{i-1}(F(M''))
\ar[r]
\ar[d]_{ \opn{H}^{i-1}(\eta_{M''}) }
&
\opn{H}^{i}(F(M'))
\ar[r]
\ar[d]_{ \opn{H}^{i}(\eta_{M'}) }
&
\opn{H}^{i}(F(M))
\ar[r]
\ar[d]_{ \opn{H}^{i}(\eta_{M}) }
&
\opn{H}^{i}(F(M''))
\ar[d]_{ \opn{H}^{i}(\eta_{M''}) }
\\
\opn{H}^{i-1}(G(M''))
\ar[r]
&
\opn{H}^{i}(G(M'))
\ar[r]
&
\opn{H}^{i}(G(M))
\ar[r]
&
\opn{H}^{i}(G(M''))
} \]
The four terms involving $M''$ are zero here. Since $M'$ has bounded above
cohomology, we know by part (2) that $\opn{H}^{i}(\eta_{M'})$ is an 
isomorphism. Therefore 
$\opn{H}^{i}(\eta_{M})$ is an isomorphism.
\end{proof}

\begin{thm} \label{thm:10}
Let $A$ and $B$ be DG rings, let 
$F, G : \cat{D}(A) \to \cat{D}(B)$ 
be triangulated functors, and let $\eta : F \to G$ be a morphism of 
triangulated functors. 
Assume that $A$ is cohomologically left pseudo-noetherian,
and $\eta_A : F(A) \to G(A)$ is an isomorphism.
\begin{enumerate}
\item If $F$ and $G$ have bounded above cohomological displacements, then 
$\eta_M : F(M) \to G(M)$ is an isomorphism for every 
$M \in \cat{D}^{-}_{\mrm{f}}(A)$.

\item If $F$ and $G$ have finite cohomological dimensions, then 
$\eta_M : F(M) \to G(M)$ is an isomorphism for every 
$M \in \cat{D}_{\mrm{f}}(A)$.
\end{enumerate}
\end{thm}

\begin{proof}
Step 1. Consider a finite free DG $A$-module $P$, i.e.\ 
$P \cong \bigoplus_{k = 1}^r A[-i_k]$ in $\cat{C}(A)$ for some 
$i_1, \ldots, i_r \in \Z$. 
Because the functors $F, G$ are triangulated, and $\eta_A$ is an isomorphism, 
it follows that $\eta_P$ is an isomorphism.

\medskip \noindent
Step 2. Now let $P$ be a finite semi-free DG $A$-module,
with finite semi-free filtration $\{ \nu_j(P) \}$ of length $j_1$
(see Definition \ref{dfn:100}). 
We prove that $\eta_{P}$ is an isomorphism by induction on $j_1$.
For $j_1 = 0$ this is step 1. Now assume $j_1 \geq 1$. 
Write $P' := \nu_{j_1 - 1}(P)$ and $P'' := \opn{gr}^\nu_{j_1}(P)$, so there is 
a distinguished triangle
$P' \to P \to P'' \xar{\vartriangle}$ in $\cat{D}(A)$. 
According to step 1 and the induction hypothesis, the morphisms 
$\eta_{P''}$ and $\eta_{P'}$ are isomorphisms. 
Hence $\eta_P$ is an isomorphism.

\medskip \noindent
Step 3. Here we assume that $F$ and $G$ have cohomological displacements at 
most $[-\infty, d_1]$ for some integer $d_1$. 
Take any $M \in \cat{D}^{-}_{\mrm{f}}(A)$. 
In order to prove that $\eta_M$ is an isomorphism it suffices to show that 
\[ \opn{H}^i(\eta_M) : \opn{H}^i(F(M)) \to \opn{H}^i(G(M)) \]
is bijective for every $i \in \Z$.  

We may assume that $M$ is nonzero. 
Let $i_1 := \opn{sup}(\opn{H}(M))$, which is an integer. 
There exists a pseudo-finite semi-free resolution $P \to M$ 
such that $\opn{sup}(P) = i_1$; see Proposition \ref{prop:101}.
We will prove that $\opn{H}^i(\eta_P)$ is an isomorphism for every $i$. 
Fix a pseudo-finite semi-free filtration $\{ \nu_j(P) \}$ of $P$. 

Take an integer $j$, and define $P' := \nu_{j}(P)$ and 
$P'' := P / \nu_{j}(P)$. So there is a  distinguished triangle
$P' \to P \to P'' \xar{\vartriangle}$
in $\cat{D}(A)$. The DG module $P''$ is concentrated in the degree range 
$\leq i_1 - j - 1$, and hence so is its cohomology.
Thus the  cohomologies of $F(P'')$ and $G(P'')$ are concentrated in the 
degree range $\leq i_1 - j - 1 + d_1$, or in other words 
$\opn{H}^i(F(P'')) = \opn{H}^i(G(P'')) = 0$ for all 
$i > i_1 - j - 1 + d_1$. 
On the other hand the DG module $P'$ is finite semi-free. 
The distinguished triangle above induces a commutative 
diagram of $\bar{A}$-modules with exact rows:
\begin{equation} \label{eqn:37}
\UseTips \xymatrix @C=5ex @R=5ex {
\opn{H}^{i-1}(F(P''))
\ar[r]
\ar[d]_{ \opn{H}^{i-1}(\eta_{P''}) }
&
\opn{H}^{i}(F(P'))
\ar[r]
\ar[d]_{ \opn{H}^{i}(\eta_{P'}) }
&
\opn{H}^{i}(F(P))
\ar[r]
\ar[d]_{ \opn{H}^{i}(\eta_{P}) }
&
\opn{H}^{i}(F(P''))
\ar[d]_{ \opn{H}^{i}(\eta_{P''}) }
\\
\opn{H}^{i-1}(G(P''))
\ar[r]
&
\opn{H}^{i}(G(P'))
\ar[r]
&
\opn{H}^{i}(G(P))
\ar[r]
&
\opn{H}^{i}(G(P''))
}
\end{equation}
For any $i > i_1 + d_1 - j$ the modules in this diagram involving $P''$ 
are zero. By step 2 we know that $\opn{H}^{i}(\eta_{P'})$ is an 
isomorphism for every $i$. Therefore 
$\opn{H}^{i}(\eta_{P})$ is an isomorphism for every 
$i > i_1 + d_1 - j$. Since $j$ can be made arbitrarily large, we conclude that 
$\opn{H}^{i}(\eta_{P})$ is an isomorphism for every $i$.

\medskip \noindent
Step 4. Here we assume that $F$ and $G$ have finite cohomological dimensions. 
So $F$ and $G$ have cohomological displacements at most $[d_0, d_1]$ for some 
$d_0 \leq d_1$ in $\Z$. 
This step is very similar to the proof of  Theorem \ref{thm:70}(3).

Take any $M \in \cat{D}^{}_{\mrm{f}}(A)$. Fix $i \in \Z$. We want to show that 
$\opn{H}^i(\eta_M)$ is an isomorphism. 
Using smart truncations of $M$, we obtain a distinguished triangle
$M' \to M \to M'' \xar{\vartriangle}$ in $\cat{D}_{\mrm{f}}(A)$, such that
$M' \in \cat{D}^{[-\infty, i - d_0 + 1]}_{\mrm{f}}(A)$ and
$M'' \in \cat{D}^{[i - d_0 + 1, \infty]}_{\mrm{f}}(A)$.
We get a commutative diagram like (\ref{eqn:37}).
The modules $\opn{H}^{i-1}(F(M''))$, $\opn{H}^{i}(F(M''))$, 
$\opn{H}^{i-1}(G(M''))$ and $\opn{H}^{i}(G(M''))$ are zero. 
Since $M' \in \cat{D}^{-}_{\mrm{f}}(A)$, step 3 says that \lb 
$\opn{H}^{i}(\eta_{M'})$ is an isomorphism. Therefore 
$\opn{H}^{i}(\eta_{M})$ is an isomorphism.
\end{proof}

The next theorem is a variant of the opposite of \cite[Proposition I.7.3]{RD}.
For a boundedness condition $\star$, we denote by $- \star$ the opposite 
boundedness condition. Thus if $\cat{D}^{\star}$ denotes either 
$\cat{D}^{}$, $\cat{D}^{+}$, $\cat{D}^{-}$ or $\cat{D}^{\mrm{b}}$, 
then $\cat{D}^{- \star}$ denotes either
$\cat{D}^{}$, $\cat{D}^{-}$, $\cat{D}^{+}$ or $\cat{D}^{\mrm{b}}$ 
respectively.

\begin{thm} \label{thm:30}
Let $A$ and $B$ be DG rings, and let 
$F : \cat{D}(A)^{\mrm{op}} \to \cat{D}(B)$ 
be a triangulated functor.
Assume that $A$ and $B$ are cohomologically left pseudo-noetherian, 
and $F(A) \in \cat{D}^+_{\mrm{f}}(B)$.
\begin{enumerate}
\item If $F$ has bounded below cohomological displacement, then 
$F(M) \in \cat{D}^+_{\mrm{f}}(B)$  for every $M \in \cat{D}^{-}_{\mrm{f}}(A)$.

\item If $F$ has finite cohomological dimension, then  
$F(M) \in \cat{D}^{- \star}_{\mrm{f}}(B)$  for every 
$M \in \cat{D}^{\star}_{\mrm{f}}(A)$.
\end{enumerate}
\end{thm}

\begin{proof}
The proof is very similar to that of Theorem \ref{thm:10}. We just outline the 
necessary changes. 

\medskip \noindent
Steps 1-2. $P$ is a finite semi-free DG $A$-module. By induction on the 
length $j_1$ of a semi-free filtration, we prove that 
$F(P) \in \cat{D}^+_{\mrm{f}}(B)$.
 
\medskip \noindent
Step 3. Here $F$ has cohomological displacement $[d_0, \infty]$ for some 
$d_0 \in \Z$, and $M \in \cat{D}^{-}_{\mrm{f}}(A)$.
Let $P \to M$ be a pseudo-finite semi-free resolution, with 
$\opn{sup}(P) = i_1$.
Take any $j \in \Z$, and consider the distinguished triangle 
$P' \to P \to P'' \xar{\vartriangle}$ where $P' := \nu_{j}(P)$. 
There is an exact sequence of $\bar{B}$-modules 
\begin{equation} \label{eqn:38}
\UseTips \xymatrix @C=5ex @R=5ex {
\opn{H}^{i-1}(F(P''))
\ar[r]
&
\opn{H}^{i}(F(P'))
\ar[r]
&
\opn{H}^{i}(F(P))
\ar[r]
&
\opn{H}^{i}(F(P'')) .
}
\end{equation}
By step 2 we know that 
$\opn{H}^{i}(F(P')) \in \cat{Mod}_{\mrm{f}} \bar{B}$ for every $i$.
If $i \leq d_0 - i_1 + j$ then 
$\opn{H}^{i-1}(F(P'')) = \opn{H}^{i}(F(P'')) = 0$. 
Therefore $\opn{H}^{i}(F(P)) \in \cat{Mod}_{\mrm{f}} \bar{B}$.
But $j$ can be made arbitrarily large. 
This proves that 
$F(M) \in \cat{D}^{}_{\mrm{f}}(B)$.
But on the other hand we know that 
$\opn{H}^{i}(F(M)) = 0$ for all $i < -i_1 + d_0$. 
We conclude that 
$F(M) \in \cat{D}^+_{\mrm{f}}(B)$.

\medskip \noindent
Step 4. Here $F$ has cohomological displacement at 
most $[d_0, d_1]$ for some $d_0 \leq d_1$ in $\Z$,
$M \in \cat{D}^{}_{\mrm{f}}(A)$, and $i \in \Z$. We truncate $M$ to obtain 
a distinguished triangle
$M' \to M \to M'' \xar{\vartriangle}$ in $\cat{D}_{\mrm{f}}(A)$, such that
$M' \in \cat{D}^{[-\infty, j + 1]}_{\mrm{f}}(A)$ and
$M'' \in \cat{D}^{[j + 1, \infty]}_{\mrm{f}}(A)$
for $j := i - d_1 - 3$.
We get an exact sequence like (\ref{eqn:38}).
The modules $\opn{H}^{i - 1}(F(M''))$ and $\opn{H}^{i}(F(M''))$ are zero, and 
$\opn{H}^{i}(F(M')) \in \cat{Mod}_{\mrm{f}} \bar{B}$ by step 3. 
Therefore $\opn{H}^{i}(F(M)) \in \cat{Mod}_{\mrm{f}} \bar{B}$. 
The condition on the boundedness of $\opn{H}(F(M))$ is established like in step 
3.
\end{proof}

\section{Reduction and Lifting} \label{sec:lifting}

Recall that by default all DG modules are left DG modules, and all DG rings are 
nonpositive (Convention \ref{conv:305}).   
In this section we study the canonical DG ring homomorphism 
$A \to \bar{A}$, and the corresponding reduction functor 
$\cat{D}(A) \to \cat{D}(\bar{A})$, 
$M \mapsto \bar{A} \ot^{\mrm{L}}_{A} M$.
We do not make any finiteness assumptions on the cohomology modules 
$\opn{H}^i(M)$.

A triangulated functor $F$ is called {\em conservative} if for any object 
$M$, $F(M) = 0$ implies $M = 0$; or equivalently, if for any morphism 
$\phi$, $F(\phi)$ is an isomorphism implies $\phi$ is an isomorphism.
Cf.\ \cite[Section 1.4]{KaSc}. The following result is analogous to the 
Nakayama Lemma (cf.\ Remark \ref{rem:317}).

\begin{prop} \label{prop:92}
Let $A$ be a DG ring. The reduction functor 
\[ \bar{A} \ot^{\mrm{L}}_{A} - : \cat{D}^{-}(A) \to \cat{D}^-(\bar{A}) \]
is conservative. 
\end{prop}

\begin{proof}
Take $M \in \cat{D}^{-}(A)$ not isomorphic to $0$, and let 
$i_1 := \opn{sup}(\opn{H}(M))$.
We can find a K-flat resolution (e.g.\ a semi-free resolution) $P \to M$ over 
$A$ such that $\opn{sup}(P) = i_1$.
Then $\bar{A} \ot^{\mrm{L}}_A M \cong \bar{A} \ot^{}_A P$, and  
(by the ``K\"unneth trick'')
\[ \opn{H}^{i_1}(\bar{A} \ot^{\mrm{L}}_A M) \cong 
\opn{H}^{i_1}(\bar{A} \ot^{}_A P) \cong 
\bar{A} \ot^{}_A \opn{H}^{i_1}(P) \cong \opn{H}^{i_1}(M) \]
is nonzero. Hence $\bar{A} \ot^{\mrm{L}}_A M$ is nonzero. 
\end{proof}

Given a homomorphism $\phi : P \to Q$ in $\cat{C}(A)$, we denote by 
$\opn{cone}(\phi)$ the corresponding cone, which is also an object of 
$\cat{C}(A)$. 

\begin{lem} \label{lem:33}
Suppose $\phi : P \to Q$ is a homomorphism in $\cat{C}(A)$,
where $P$ and $Q$ are pseudo-finite \tup{(}resp.\ finite\tup{)} semi-free DG 
modules. Then $\opn{cone}(\phi)$ is a pseudo-finite \tup{(}resp.\ finite\tup{)} 
semi-free DG module. 
\end{lem}

\begin{proof}
Clear from Proposition \ref{prop:100}.
\end{proof}

\begin{prop} \label{prop:305}  
Let $M \in \cat{D}^-(A)$. We write
$\bar{M} := \bar{A} \ot^{\mrm{L}}_A M \in \cat{D}^-(\bar{A})$. 
\begin{enumerate}
\item If $\bar{M}$ is isomorphic in $\cat{D}(\bar{A})$ to 
$\bar{A}^{\oplus r}$ for some cardinal number $r$, then 
$M$ is isomorphic in $\cat{D}(A)$ to $A^{\oplus r}$.

\item If $\bar{M}$ is isomorphic in $\cat{D}(\bar{A})$ to a semi-free DG 
$\bar{A}$-module $\bar{P}$ of semi-free length $d$, then $M$ is 
isomorphic in $\cat{D}(A)$ to a semi-free DG $A$-module $P$ of semi-free 
length $d$.
\end{enumerate}
\end{prop}

Observe that the DG $\bar{A}$-module $\bar{P}$ in (2) above is nothing but a
bounded complex of free $\bar{A}$-modules; cf.\ Proposition \ref{prop:100}.
The semi-free length was introduced in Definition \ref{dfn:311}.

\begin{proof}
Step 1. 
In view of Proposition \ref{prop:92} we can assume that $\bar{A}$ and $M$ are 
nonzero. Define $i_1 := \opn{sup}(\opn{H}(M))$. 
By replacing $M$ with a suitable resolution of it, we can assume that $M$ is a 
K-flat DG $A$-module satisfying $\opn{sup}(M) =  i_1$.
After that we can also assume that $\bar{M} = \bar{A} \ot_A M$. 
The K\"unneth formula says that
\[ \opn{H}^{i_1}(\bar{M}) \cong 
\opn{H}^{i_1}(\bar{A} \ot_{A} M) \cong \bar{A} \ot_{A} \opn{H}^{i_1}(M)
\cong \opn{H}^{i_1}(M) \]
as $\bar{A}$-modules. Therefore  
$\opn{sup}(\opn{H}(\bar{M})) = \opn{sup}(\bar{M}) = i_1$.

We are given an isomorphism $\bar{\phi} : \bar{P} \to \bar{M}$ in 
$\cat{D}(\bar{A}^0)$, where $\bar{P}$ is a bounded complex of free 
$\bar{A}$-modules. 
Since $\bar{P}$ is K-projective, we can assume that the isomorphism 
$\bar{\phi} : \bar{P} \to \bar{M}$ in $\cat{D}(\bar{A})$
is in fact a quasi-isomorphism in $\cat{C}(\bar{A})$.
The proof continues by induction on $j := \opn{amp}(\bar{P}) \in \N$.
Note that in item (1) we have $i_1 = 0$, $\bar{P} = \bar{A}^{\oplus r}$ and 
$j = 0$. 

\medskip \noindent
Step 2. In this step we assume that $j = 0$. This means that the 
only nonzero term of $\bar{P}$ is
in degree $i_1$, and it is 
the free module $\bar{P}^{i_1} \cong \bar{A}^{\oplus r_1}$
for some cardinal number $r_1$. In other words, 
$\bar{P} \cong \bar{A}[-i_1]^{\oplus r_1}$ in $\cat{C}(\bar{A})$.
So $\opn{H}^{i_1}(\bar{M}) \cong \bar{P}^{i_1}$, and 
$\opn{H}^{i}(\bar{M}) = 0$ for all $i \neq i_1$.
Recall the isomorphism of $\bar{A}$-modules
$\opn{H}^{i_1}(\bar{M}) \cong \opn{H}^{i_1}(M)$ from Step 1.  
There are canonical surjections 
$M^{i_1} \to \opn{H}^{i_1}(M)$ and
$\bar{M}^{i_1} \to \opn{H}^{i_1}(\bar{M})$.
We can write the quasi-isomorphism $\bar{\phi}$ as 
$\bar{\phi} : \bar{A}[-i_1]^{\oplus r_1} \to \bar{M}$.
 
Consider the nonderived reduction functor
$F : \cat{C}(A) \to \cat{C}(\bar{A})$, $F(-) := \bar{A} \ot_{A} -$.
Let $P := A[-i_1]^{\oplus r_1}$, a free DG $A$-module
satisfying $F(P) \cong \bar{P}$.  
There exists a homomorphism 
$\phi : P \to M$ in $\cat{C}(A)$ that lifts 
the quasi-isomorphism
$\bar{\phi} : \bar{P} \to \bar{M}$,
namely $\bar{\phi} = F(\phi)$. 
Now the DG modules $P$ and $M$ are K-flat, so 
$\bar{\phi} = \mrm{L} F(\phi)$. 
Since $\bar{\phi}$ is an isomorphism, and since the functor 
$\mrm{L} F$ is conservative, we conclude that $\phi$ is an isomorphism.
This proves item (1).

\medskip \noindent
Step 3. Here we suppose that $j \geq 1$.
Let $i_2 := \opn{sup} (\bar{P})$, which is of course $\geq i_1$. 
Say $\bar{P}^{i_2} \cong \bar{A}^{\oplus r_2}$ for some natural number $r_2$.
Define DG modules $\bar{P}' := \bar{A}[-i_2]^{\oplus r_2}$ and 
$P' := A[-i_2]^{\oplus r_2}$; these satisfy
$\bar{P}' \cong \bar{A} \ot_A P'$. 
The inclusion $\bar{P}^{i_2} \subseteq \bar{P}$ is viewed as a DG module 
homomorphism $\bar{\al} : \bar{P}' \to \bar{P}$.
We also have a quasi-isomorphism 
$\bar{\phi} : \bar{P} \to \bar{M}$ and an equality 
$\bar{M} = \bar{A} \ot_A M$ in $\cat{C}(\bar{A})$.
In this way we obtain a homomorphism 
$\bar{\psi} :  \bar{P}' \to \bar{M}$,
$\bar{\psi} := \bar{\phi} \circ \bar{\al}$. 
Because $P'$ is a free DG $A$-module, there is a homomorphism 
$\psi :  P' \to M$ in $\cat{C}(A)$ lifting $\bar{\psi}$, namely 
$\bar{\psi} = F(\psi)$, where $F$ is the functor $\bar{A} \ot_A -$. 

Let $M'' \in \cat{C}(A)$ be the cone of $\psi$, so there is a distinguished 
triangle 
\begin{equation} \label{eqn:30}
P' \xar{\psi} M \xar{\chi} M'' \xar{} P'[1] 
\end{equation}
in $\cat{D}(A)$. Define $\bar{M}'' := F(M'')$ and $\bar{\chi} := F(\chi)$,  
which are an object and a morphism in $\cat{C}(\bar{A})$, respectively.  
Since all three DG modules in this triangle are K-flat, it follows that 
\[ \bar{P}' \xar{\bar{\psi}} \bar{M} \xar{\bar{\chi}} \bar{M}'' 
\xar{} \bar{P}'[1] \]
is a distinguished triangle in $\cat{D}(\bar{A})$.
On the other hand, let $\bar{P}''$ be the cokernel of the inclusion 
$\bar{\al} : \bar{P}' \to \bar{P}$. So there is a distinguished triangle
\[ \bar{P}' \xar{\bar{\al}} \bar{P} \xar{\bar{\be}} \bar{P}'' 
\xar{} \bar{P}'[1] \]
in $\cat{D}(\bar{A})$. 
Consider the diagram of solid arrows in $\cat{D}(\bar{A})$~: 
\[  \UseTips \xymatrix @C=5ex @R=5ex {
\bar{P}'
\ar[r]^{\bar{\al}} 
\ar[d]^{=}
&
\bar{P} 
\ar[r]^{\bar{\be}} 
\ar[d]^{\bar{\phi}} 
&
\bar{P}'' 
\ar[r]^{}
\ar@{-->}[d]^{\bar{\phi}''} 
&
\bar{P}'[1]
\ar[d]^{=}
\\
\bar{P}'
\ar[r]^{\bar{\psi}} 
&
\bar{M} 
\ar[r]^{\bar{\chi}} 
&
\bar{M}'' 
\ar[r]^{}
&
\bar{P}'[1]
} \]
The square on the left is commutative, and therefore it extends to an 
isomorphism of distinguished triangles. So there is an isomorphism
$\bar{\phi}'' : \bar{P}'' \to \bar{M}''$
in $\cat{D}(\bar{A})$.

Finally, the complex $\bar{P}''$ is a bounded complex of finite free 
$\bar{A}$-modules, of amplitude $j-1 \geq 0$. According to the induction 
hypothesis (step 2 for $j > 1$, and step 1 for $j = 1$) there is an isomorphism 
$\phi'' : P'' \iso M''$ in $\cat{D}(A)$ for some finite semi-free DG 
$A$-module $P''$. From (\ref{eqn:30}) we obtain a distinguished triangle 
\[ P' \xar{\psi} M \xar{\psi} P'' \xar{\ga} P'[1] \]
in $\cat{D}(A)$. We can assume that $\ga$ is a homomorphism in 
$\cat{C}(A)$. Turning this triangle we get a distinguished triangle 
\[ P''[-1] \xar{- \ga[-1]} P' \xar{\psi} M \xar{\vartriangle} . \]
Define $P$ to be the cone on the homomorphism 
$- \ga[-1] : P''[-1] \to P'$. Then $P \cong M$ in $\cat{D}(A)$,
and by Lemma \ref{lem:33}, $P$ is a finite semi-free DG $A$-module. 
\end{proof}

We learned the next proposition from B. Antieau and J. Lurie. Cf.\ 
\cite[Corollary 7.2.2.19]{Lu1} for a more general statement. 

\begin{prop} \label{prop:310}
Let $\bar{P}$ be a projective $\bar{A}$-module. Then there exists a DG 
$A$-module $P$ with these properties:
\begin{enumerate}
\rmitem{i} $P$ is a direct summand, in $\cat{D}(A)$, of a direct sum of copies 
of $A$.
\rmitem{ii} 
$\bar{A} \ot^{\mrm{L}}_A P \cong \bar{P}$
in $\cat{D}(\bar{A})$.
\end{enumerate} 
\end{prop}

\begin{proof}
Say $\bar{P}$ is a direct summand, in $\cat{Mod} \bar{A}$, of a free 
$\bar{A}$-module $\bar{F}$. So $\bar{P}$ is the image of an idempotent 
endomorphism $\bar{\phi} : \bar{F} \to \bar{F}$. Let $F$ be the direct sum 
in $\cat{C}(A)$ of copies of $A$, as many as there are copies of $\bar{A}$ in 
$\bar{F}$. There is a canonical surjection $F \to \bar{F}$. Choose any 
homomorphism $\phi : F \to F$ in $\cat{C}(A)$ lifting $\bar{\phi}$.
Note that many such $\phi$ exist; and they aren't necessarily idempotents. Let 
$P$ be the homotopy colimit construction on $\phi$. Namely we let
\[ \Phi : \bigoplus\nolimits_{i \in \N} \, F \ \to \ 
\bigoplus\nolimits_{i \in \N} \, F \]
be the homomorphism
\[  \Phi(a_0, a_1, a_2, ...) := (a_0, a_1 - \phi(a_0), a_2 - \phi(a_1), ...) \]
in $\cat{C}(A)$, and then we define 
$P := \opn{cone}(\Phi) \in \cat{C}(A)$. 
Because $P$ is K-flat over $A$, we see that  
$\bar{A} \ot^{\mrm{L}}_A P \cong \bar{A} \ot_A P \cong 
\opn{cone}(\bar{\Phi})$,
where 
\[ \bar{\Phi} : \bigoplus\nolimits_{i \in \N} \, \bar{F} \ \to \ 
\bigoplus\nolimits_{i \in \N} \, \bar{F} \]
is the homotopy colimit construction on $\bar{\phi}$.
An easy calculation, using the fact that $\bar{\phi}$ is an idempotent with 
image $\bar{P}$,  shows that 
$\opn{cone}(\bar{\Phi}) \cong \bar{P}$ in $\cat{D}(\bar{A})$. 
This proves (ii).

As for (i): say $\bar{Q}$ is the other direct summand of $\bar{F}$
in $\cat{Mod} \bar{A}$.
We can lift it to a DG $A$-module $Q$ as above. Now
$\bar{P} \oplus \bar{Q} \cong \bar{F}$
in $\cat{D}(\bar{A})$. By Proposition \ref{prop:305}(1) we deduce that
$P \oplus Q \cong F$ in $\cat{D}(A)$.
\end{proof}

\section{Localization of Commutative DG Rings} \label{sec:localiz}

In this section we specialize to DG rings satisfying a commutativity 
condition. Such a DG ring $A$ can be localized on $\opn{Spec} \bar{A}$. 

\begin{dfn} \label{dfn:305}
Let $A = \bigoplus_{i \in \Z} A^i$ be a DG ring. 
\begin{enumerate}
\item $A$ is called {\em weakly commutative} if 
$b \cd a =  (-1)^{i j} \cd a \cd b$ 
for all $a \in A^{i}$ and $b \in A^{j}$.

\item $A$ is called {\em strongly commutative} if it is weakly commutative, and 
if $a \cd a = 0$ for all $a \in A^{i}$ with $i$ odd.

\item $A$ is called a {\em commutative DG ring} if it is  strongly commutative 
and nonpositive. 
\end{enumerate}
\end{dfn}

If $A$ is weakly commutative, then any left DG 
$A$-module $M$ can be viewed as a
right DG $A$-module. The formula for the right action is this: 
$m \cd a := (-1)^{i j} \cd a \cd m$ 
for $a \in A^i$ and $m \in M^j$.

The full subcategory of $\cat{DGR}$ on the commutative DG rings 
is denoted by $\cat{DGR}^{\leq 0}_{\mrm{sc}}$.

\begin{conv} \label{conv:100}
By default, all DG rings from here on in the paper are 
commutative (Definition \ref{dfn:305}(3)), unless explicitly stated otherwise; 
i.e.\ we work inside $\cat{DGR}^{\leq 0}_{\mrm{sc}}$.
In particular all rings are commutative by default. 
\end{conv}

Since our DG rings are now commutative, we can talk about cohomologically 
pseudo-noetherian DG rings, instead of cohomologically left pseudo-noetherian 
ones (Definition \ref{dfn:391}).

\begin{rem} \label{rem:580}
In \cite{YZ1}, and in earlier versions of this paper, we used the name 
``super-commutative'' for what we now call ``strongly commutative''. The book 
\cite{ML} uses the term ``strictly commutative''. 

Of course when $2$ is invertible in the ring $A^0$ (e.g.\ in characteristic 
$0$), there is no difference between weakly commutative and strongly 
commutative DG rings.

The reader may wonder why we chose to work in the category 
$\cat{DGR}^{\leq 0}_{\mrm{sc}}$ of strongly commutative nonpositive DG rings, 
and not in the category $\cat{DGR}^{\leq 0}_{\mrm{wc}}$ of
weakly commutative ones. Here is the reason. 

Consider any nontrivial commutative 
base ring $\K$ (e.g.\ $\K = \Z$), and the corresponding slice category 
$\cat{DGR}^{\leq 0}_{\mrm{sc}} / \K$.
Let $X$ be a set of nonpositive graded variables. 
The strongly commutative polynomial ring 
in $X$ (see \cite[Definition 3.10]{Ye5}) is denoted by $\K_{\mrm{sc}}[X]$. 
It is easy to see that $\K_{\mrm{sc}}[X]$ is a flat $\K$-module. 
A DG ring $A \in \cat{DGR}^{\leq 0}_{\mrm{sc}} / \K$
is called semi-free if $A^{\natural} \cong \K_{\mrm{sc}}[X]$ for some $X$.
Such a semi-free DG ring $A$ is a bounded above complex of flat $\K$-modules, 
and thus it is a K-flat DG $\K$-module. We had shown in \cite{YZ1} that there 
are enough semi-free resolutions in $\cat{DGR}^{\leq 0}_{\mrm{sc}} / \K$.
Therefore there are enough K-flat resolutions in 
$\cat{DGR}^{\leq 0}_{\mrm{sc}} / \K$.

On the other hand, let us examine the category 
$\cat{DGR}^{\leq 0}_{\mrm{wc}} / \Z$.
The weakly commutative polynomial ring $\Z_{\mrm{wc}}[x]$ over $\Z$, 
in a single variable $x$ of degree $-1$, has this structure as a graded 
$\Z$-module:
\[ \Z_{\mrm{wc}}[x] = \Z \oplus (\Z \cd x) \oplus 
\bigl( \smfrac{\Z}{(2)} \cd x^2 \bigr) \oplus 
\bigl( \smfrac{\Z}{(2)} \cd x^3 \bigr)  \oplus \cdots . \] 
We see that it is not flat over $\Z$. Moreover, it is not hard to see that if
$A$ is any object in $\cat{DGR}^{\leq 0}_{\mrm{wc}} / \Z$
such that $A^{\natural} \cong \Z_{\mrm{wc}}[x]$,
then $A$ is not a K-flat DG $\Z$-module. 
Similarly, it seems that any semi-free object 
$A \in \cat{DGR}^{\leq 0}_{\mrm{wc}}$ that has a nontrivial degree $-1$ 
component, also fails to be a K-flat DG 
$\Z$-module. The upshot is that, most likely, {\em there are not enough 
K-flat resolutions in $\cat{DGR}^{\leq 0}_{\mrm{wc}} / \Z$}.
\end{rem}

\begin{prop} \label{prop:311}
Let $A$ be a DG ring. The action of $A^0$ on $\cat{C}(A)$ 
makes $\cat{D}(A)$ into an $\bar{A}$-linear category.
\end{prop}

\begin{proof}
Let $\phi : M \to N$ be a morphism in  $\cat{C}(A)$; namely 
$\phi$ is a degree $0$ cocycle in the DG abelian group 
$\opn{Hom}_A(M, N)$. For any $a \in A^{-1}$, 
the homomorphism $\d(a) \cd \phi = \d(a \cd \phi)$ 
is a degree $0$ coboundary in $\opn{Hom}_A(M, N)$, and so it
vanishes in the homotopy category $\cat{K}(A)$. Therefore $\cat{K}(A)$ is 
an $\bar{A}$-linear category, and hence so is its localization 
$\cat{D}(A)$.
\end{proof}

\begin{dfn} \label{dfn:40}
Let $A$ be a DG ring, with canonical homomorphism
$\pi : A \to \bar{A} = \opn{H}^0(A)$. Given a multiplicatively closed 
subset $S$ of $\bar{A}$, the set
$\til{S} := \pi^{-1}(S) \cap A^0$ is a 
multiplicatively closed subset of $A^0$. Define the ring 
$A^0_S :=  \til{S}^{-1} \cd A^0$ 
(this is the usual localization), and the DG ring 
$A_S := A^0_S \ot_{A^0} A$.
There is a canonical DG ring homomorphism 
$\la_S : A \to A_S$.

If $S = \{ s^i \}_{i \in \N}$ for some element $s \in \bar{A}$, 
then we also use the notation $A_s := A_S$.
\end{dfn}

There is the usual localization $\bar{A}_S = S^{-1} \cd \bar{A}$ of the 
ring $\bar{A}$ w.r.t.\ $S$. We get a graded ring 
$\opn{H}(A)_S := \bar{A}_S \ot_{\bar{A}} \opn{H}(A)$. 
There are graded ring homomorphisms 
$\opn{H}(A) \to \opn{H}(A)_S$ and 
$\opn{H}(\la_S) :\opn{H}(A) \to \opn{H}(A_S)$.

\begin{prop} \label{prop:300}
Let $A$ be a DG ring and $S \subseteq \bar{A}$ a multiplicatively closed 
subset. 
\begin{enumerate}
\item There is a unique isomorphism of graded $\opn{H}(A)$-rings 
$\opn{H}(A_S) \cong \opn{H}(A)_S$.

\item For any DG $A$-module $M$ there is a unique isomorphism of graded 
$\opn{H}(A_S)$-modules 
\[ \opn{H}(A_S \ot_A M) \cong \opn{H}(A)_S \ot_{\opn{H}(A)} \opn{H}(M)  \]
that is compatible with the homomorphisms from $\opn{H}(M)$.

\item If $A$ is cohomologically pseudo-noetherian, then so is $A_S$.
\end{enumerate}
\end{prop}

\begin{proof}
Items (1-2) are true because the ring homomorphism $\la_S : A^0 \to A^0_S$ is 
flat. Item (3) is an immediate consequence of (1-2). 
\end{proof}

\begin{dfn}  \label{dfn:70}
Let $A$ be a ring, and let $\bsym{a} = (a_1, \ldots, a_n)$ be a 
sequence of elements of $A$. 
We call $\bsym{a}$ a {\em covering sequence} of $A$ if
$\sum_{i = 1}^n A \cd a_i = A$. 
\end{dfn}
 
Consider the spectrum $\opn{Spec} A$ of the ring $A$. For an element $a \in A$ 
we identify the principal open set \
$\{ \p \in \opn{Spec} A \mid a \notin \p \}$
with the scheme $\opn{Spec} A_a$, where $A_a$ is the localization of $A$ 
w.r.t.\ $a$. Clearly a sequence 
$\bsym{a} = (a_1, \ldots, a_n)$ in $A$ is a covering sequence iff
$\opn{Spec} A = \bigcup_{i = 1}^n \, \opn{Spec} A_{a_i}$.

Let $A$ be a DG ring, and let $\bsym{a} = (a_1, \ldots, a_n)$ be a 
covering sequence of $\bar{A}$.
For any  strictly increasing sequence
$\bsym{i} = (i_0, \ldots, i_p)$ of length $p$ in the integer interval $[0, n]$, 
i.e.\  $1 \leq i_0 < \cdots < i_p \leq n$, we define the ring 
\begin{equation} \label{eqn:341}
\opn{C}(A^0; \bsym{a})(\bsym{i}) := 
A^0_{a_{i_0}} \ot_{A^0} \cdots \ot_{A^0} A^0_{a_{i_p}} , 
\end{equation}
where $A^0_{a_{0}}, \ldots, A^0_{a_{p}}$ are the localizations from Definition 
\ref{dfn:40}. Next, for any $p \in [0, n - 1]$ we let 
\begin{equation} \label{eqn:342}
\opn{C}^p(A^0; \bsym{a}) := \bigoplus_{\bsym{i}} \, 
\opn{C}(A^0; \bsym{a})(\bsym{i}) \ , 
\end{equation}
where the sum is on all strictly increasing sequences $\bsym{i}$ of length $p$.
Finally we define the DG $A^0$-module 
\begin{equation} \label{eqn:340}
\opn{C}(A^0; \bsym{a}) := 
\bigoplus_{p = 0}^{n - 1} \, \opn{C}^p(A^0; \bsym{a}) .
\end{equation}
The differential 
$\opn{C}^p(A^0; \bsym{a}) \to \opn{C}^{p+1}(A^0; \bsym{a})$
is $\sum_{\bsym{i}, k} (-1)^k \cd \la_{\bsym{i}, k}$,
where $\bsym{i}$ runs over the strictly increasing sequences of length $p + 1$,
$k \in [0, p + 1]$, 
$\pa_k(\bsym{i})$ is the sequence obtained from $\bsym{i}$ by omitting 
$i_k$, and 
$\la_{\bsym{i}, k} : \opn{C}(A^0; \bsym{a})(\pa_k(\bsym{i})) \to 
\opn{C}(A^0; \bsym{a})(\bsym{i})$
is the canonical ring homomorphism.

\begin{dfn}  \label{dfn:72}
Let $A$ be a DG ring, and let  $\bsym{a} = (a_1, \ldots, a_n)$ be a 
covering sequence of $\bar{A}$.
For a DG $A$-module $M$, the {\em \v{C}ech DG module} of $M$ is the DG 
$A$-module
\[ \opn{C}(M; \bsym{a}) := 
\opn{C}(A^0; \bsym{a}) \ot_{A^0} M . \]
\end{dfn}

There is a canonical DG module homomorphism 
$\opn{c}_M : M \to \opn{C}(M; \bsym{a})$, 
sending $m \in M$ to $\sum_i 1_i \ot m \in \opn{C}(M; \bsym{a})$,
where $1_i$ is the element 
$1 \in A^0_{a_{i}} \subseteq \opn{C}^0(A^0; \bsym{a})$.

\begin{prop} \label{prop:70}
Let $\bsym{a} = (a_1, \ldots, a_n)$ be a covering sequence of $\bar{A}$, and 
let $M$ be a DG $A$-module. Then the homomorphism 
$\opn{c}_M : M \to \opn{C}(M; \bsym{a})$ is a quasi-isomorphism. 
\end{prop}

\begin{proof}
Since $\opn{C}(A; \bsym{a})$ is a K-flat DG $A$-module, and 
$\opn{C}(M; \bsym{a}) \cong \opn{C}(A; \bsym{a}) \ot_A M$, we see that 
$\opn{C}(-; \bsym{a})$ is a triangulated functor from $\cat{D}(A)$ to 
itself. 
The homomorphism $\opn{c}_M$ is a quasi-isomorphism iff it is an isomorphism 
in $\cat{D}(A)$.

The cohomological dimension of the functor $\opn{C}(-; \bsym{a})$ is finite: 
it is at most $n-1$. 
According to Theorem \ref{thm:70}(3) it suffices to check that $\opn{c}_M$
is a quasi-isomorphism for 
$M \in \cat{Mod} \bar{A}$. But in this case 
$\opn{C}(M; \bsym{a}) \cong \opn{C}(\bar{A}; \bsym{a}) \ot_{\bar{A}} M$,
so $\opn{C}(M; \bsym{a})$ is the usual \v{C}ech complex for the covering of
$\opn{Spec} \bar{A}$ determined by the sequence $\bsym{a}$. 
In geometric language (cf.\ \cite[Section III.4]{Ha}), writing 
$X := \opn{Spec} \bar{A}$ and $U_i := \opn{Spec} \bar{A}_{a_i}$,
and letting $\mcal{M}$ denote the quasi-coherent $\OO_X$-module corresponding 
to $M$, we have $M \cong \Ga(X, \MM)$ and
$\opn{C}(M; \bsym{a}) \cong \opn{C}(\{ U_i \}, \MM)$. 
By 
\cite[\href{http://stacks.math.columbia.edu/tag/01X9}{Lemma 01X9}]{SP},
the homomorphism 
$\opn{c}_M : M \to \opn{C}(M; \bsym{a})$ 
is a quasi-isomorphism. 
\end{proof}

\begin{rem}
Actually the DG $A$-module $\opn{C}(A; \bsym{a})$
has more structure. There is a cosimplicial commutative ring 
$\opn{C}_{\mrm{cos}}(A^0; \bsym{a})$, 
whose degree $p$ piece is 
\[ \opn{C}_{\mrm{cos}}^p(A^0; \bsym{a}) := 
\prod_{\bsym{i}} \ \opn{C}(A^0; \bsym{a})(\bsym{i}) , \]
where $\bsym{i} = (i_0, \ldots, i_p)$ are weakly increasing
sequences in $[1, n]$. 
The \v{C}ech DG module $\opn{C}(A^0; \bsym{a})$ is the standard 
normalization 
of $\opn{C}_{\mrm{cos}}(A^0; \bsym{s})$, and as such it has a structure 
of noncommutative central DG $A^0$-ring (which is concentrated in non-negative 
degrees). Hence $\opn{C}(A; \bsym{a})$ is a noncommutative DG ring,
and $\opn{c}_A : A \to \opn{C}(A; \bsym{a})$ 
is a DG ring quasi-isomorphism. 
See \cite[Section 8]{PSY}.
\end{rem}

\begin{dfn}  \label{dfn:90}
Let $A$ be a ring, and let $\bsym{e} = (e_1, \ldots, e_n)$ be a 
sequence of elements of $A$. We call $\bsym{e}$ an {\em idempotent covering 
sequence} if each $e_i$ is an idempotent element of $A$, $e_i \cd e_j = 0$ 
for $i \neq j$, and $1 = \sum_{i = 1}^n e_i$. 
\end{dfn}

Suppose $\bsym{e} = (e_1, \ldots, e_n)$ is an idempotent covering sequence of 
the ring $A$. Of course $\bsym{e}$ is a covering sequence in the sense of 
Definition \ref{dfn:70}. 
For any $i$ there is a unique $A$-ring isomorphism
$A_{e_i} \cong A / (1 - e_i) \cd A$; 
namely the localization of $A$ with respect to the element $e_i$ is also 
the quotient ring modulo the ideal generated by the complementary idempotent 
$1 - e_i$. There is a ring isomorphism 
\begin{equation} \label{eqn:585}
A \iso \prod_{i = 1}^n \ A_{e_i} \,  .
\end{equation}
The scheme $\opn{Spec} A_{e_i}$ is an open-closed subscheme of $\opn{Spec} A$, 
and 
\begin{equation} \label{eqn:551}
\opn{Spec} A = \coprod_{i = 1}^n \ \opn{Spec} A_{e_i} \, . 
\end{equation}

\begin{dfn} \label{dfn:585}
Let $\bsym{e} = (e_1, \ldots, e_n)$ be an idempotent covering sequence of 
the ring $A$. The {\em $\bsym{e}$-induced decomposition} of $A$ is 
the ring isomorphism (\ref{eqn:585}). 
\end{dfn}

\begin{dfn} \label{dfn:560}
Let $A$ be a ring. The set of locally constant functions 
$f : \opn{Spec} A \to \Z$ (for the Zariski topology on $\opn{Spec} A$) shall be 
denoted by $\opn{F}_{\mrm{lc}}(\opn{Spec} A, \Z)$. 
It is an abelian group by pointwise addition.  
\end{dfn}

\begin{prop} \label{prop:550}
Let $A$ be a ring. There is a bijection between the set 
$\opn{F}_{\mrm{lc}}(\opn{Spec} A, \Z)$,
and the set of pairs $(\bsym{e}, \bsym{k})$, consisting of 
an idempotent covering sequence $\bsym{e} = (e_1, \ldots, e_n)$
of $A$, and a nondecreasing sequence of integers 
$\bsym{k} = (k_1, \ldots, k_n)$. 
This bijection sends a function $f \in \opn{F}_{\mrm{lc}}(\opn{Spec} A, \Z)$
to the pair $(\bsym{e}, \bsym{k})$ that satisfies
$f^{-1}(k_i)  = \opn{Spec} A_{e_i}$.
\end{prop}

\begin{proof}
Consider a locally constant function $f : \opn{Spec} A \to \Z$.
Then $f$ is continuous for the discrete topology on $\Z$. 
Since the topological space $\opn{Spec} A$ is quasi-compact, the image of $f$ 
must be finite. This gives rise to a finite decomposition of 
$\opn{Spec} A$ into open-closed subsets, which must be of the form 
(\ref{eqn:551}) for some idempotent covering sequence $\bsym{e}$. 
After suitable renumbering we get the pair  $(\bsym{e}, \bsym{k})$.

The converse is clear. 
\end{proof}

\begin{exa} \label{exa:550}
Often $\opn{Spec} A$ of the ring $A$ has finitely many connected 
components; say $n$ of them. Prototypes are: 
\begin{enumerate}
\item $A$ is a noetherian ring.
\item $A$ is a semilocal ring. 
\item $A$ is the ring of continuous (resp.\ differentiable) 
functions $X \to \R$, where $X$ is a connected topological space (resp.\ a 
connected differentiable manifold). Here $n = 1$. 
\end{enumerate}

An ordering of the connected components of $\opn{Spec} A$
gives rise to an idempotent covering sequence $\bsym{e} = (e_1, \ldots, e_n)$.
The sequence of delta functions $(\de_1, \ldots, \de_n)$
is a basis of the group $\opn{F}_{\mrm{lc}}(\opn{Spec} A, \Z)$, 
which is thus isomorphic to $\Z^n$.
\end{exa}

\begin{dfn}  \label{dfn:551}
Let $A$ be a ring, and assume that $\opn{Spec} A$ has finitely many 
connected components. A {\em connected component idempotent covering sequence}
of $A$ is an idempotent covering sequence
$\bsym{e} = (e_1, \ldots, e_n)$, such that each $\opn{Spec} A_{e_i}$
is nonempty and connected. 
\end{dfn}

Clearly a connected component idempotent covering sequence of $A$ is unique up 
to permutation. 

\begin{prop} \label{prop:20}
Let $A$ be a DG ring, and let 
$\bsym{e} = (e_1, \ldots, e_n)$ be an idempotent covering sequence of 
$\bar{A} = \opn{H}^0(A)$. For any $i$ we have the localized DG ring 
$A_i := A_{e_i}$ from Definition \tup{\ref{dfn:40}}, and the DG ring 
homomorphism $\la_i : A \to A_i$. Then the DG ring homomorphism
\[ (\la_1,  \ldots , \la_n) : A \to A_1 \times \cdots \times A_n \]
is a quasi-isomorphism. 
\end{prop}

\begin{proof}
Let's write $\bar{A}_i := \bar{A}_{e_i} = \opn{H}^0(A)_{e_i}$, so 
$\bar{A} = \prod_{i = 1}^n \, \bar{A}_i$. 
Using Proposition \ref{prop:300}(1) we obtain canonical graded ring isomorphisms
\[ \opn{H}(A) \cong 
\prod\nolimits_{i = 1}^n \, \bigl( \bar{A}_i \ot_{\bar{A}} \opn{H}(A) \bigr) 
\cong \prod\nolimits_{i = 1}^n \, \opn{H}(A_i) \cong 
\opn{H}\bigl( \prod\nolimits_{i = 1}^n \, A_i \bigr) . \] 
The composition of these isomorphisms is exactly 
$\opn{H}(\la_1,  \ldots , \la_n)$. 
\end{proof}

\begin{cor} \label{cor:22}
With $A_1, \ldots, A_n$ as in Proposition \tup{\ref{prop:20}}, the 
restriction functors 
$\opn{rest}_{\la_i} : \cat{D}(A_{i}) \to \cat{D}(A)$
induce an equivalence of triangulated categories
$\bigoplus_{i = 1}^n \,  \cat{D}(A_{i}) \to \cat{D}(A)$.
\end{cor}

\begin{proof}
This is standard. Cf.\ \cite[Proposition 1.4]{YZ1}. 
\end{proof}

\begin{dfn} \label{dfn:102}
Let $A$ be a DG ring, and let $e \in \bar{A} = \opn{H}^0(A)$
be an idempotent element. Consider the localized DG ring 
$A_{e}$ corresponding to $e$, as in Definition \ref{dfn:40}. 
The triangulated functor
\[ E : \cat{D}(A) \to \cat{D}(A) \ , \  
E(M) := A_{e} \ot_A M, \]
is called the {\em idempotent functor} corresponding to $e$.
\end{dfn}

\begin{dfn}  \label{dfn:343}
Let $A$ be a DG ring, and let 
$\bsym{e} = (e_1, \ldots, e_n)$
be an idempotent covering sequence of $\opn{Spec} \bar{A}$. 
\begin{enumerate}
\item The corresponding DG ring quasi-isomorphism 
\[ (\la_1,  \ldots , \la_n) : A \to A_1 \times \cdots \times A_n \, , \]
from Proposition \ref{prop:20} 
is called the {\em $\bsym{e}$-induced decomposition} of $A$. 

\item The corresponding functors  $E_1, \ldots, E_n$ from Definition 
\ref{dfn:102} are called the 
{\em $\bsym{e}$-induced idempotent functors} of $A$. 

\item In case $\opn{Spec} \bar{A}$ has finitely many connected components, and 
$\bsym{e}$ is one of its connected component idempotent covering sequence, then 
 the expression ``$\bsym{e}$-induced'' above is sometimes replaced by 
``connected component''. 
\end{enumerate}
\end{dfn}

\begin{prop} \label{prop:90}
Suppose $\bsym{e} = (e_1, \ldots, e_n)$ is an idempotent covering sequence of 
$\bar{A}$. Let \lb $E_1, \ldots, E_n$ be the $\bsym{e}$-induced
idempotent functors. 
\begin{enumerate}
\item We have $E_i \circ E_i \cong E_i$, $E_i \circ E_j = 0$ for $i \neq j$, 
and $\sum_{i = 1}^n E_i \cong \opn{id}_{\cat{D}(A)}$.

\item Under the equivalence of categories in Corollary \tup{\ref{cor:22}},
$\cat{D}(A_{i})$ is  the essential image of $E_i$.
\end{enumerate}
\end{prop}

\begin{proof}
This is similar to the proof of Proposition \ref{prop:20}, and is left to the 
reader.
\end{proof}

\begin{prop} \label{prop:21}
Let $f : A \to B$ be a DG ring homomorphism, and let 
$M, N \in \cat{D}(B)$.
We write $\bar{f} := \opn{H}^0(f)$ and 
$F := \opn{rest}_f : \cat{D}(B) \to \cat{D}(A)$. 
Assume that the ring homomorphism 
$\bar{f} : \bar{A} \to \bar{B}$ is surjective.
Let $\bsym{e} = (e_1, \ldots, e_n)$ be an idempotent covering sequence of 
$\bar{B}$,  let $E_1, \ldots, E_n$ be the $\bsym{e}$-induced idempotent
functors of $B$, and write $M_i := E_i(M)$ and $N_i := E_i(N)$. 
Then for any  $i \neq j$ we have
\[ \opn{Hom}_{\cat{D}(A)} \bigl( F(M_i), F(N_j) \bigl) = 0 . \]
Therefore we get a canonical isomorphism of $\bar{A}$-modules 
\[ \opn{Hom}_{\cat{D}(A)} \bigl( F(M), F(N) \bigl) \cong 
\bigoplus_{i = 1}^n \,
\opn{Hom}_{\cat{D}(A)} \bigl( F(M_i), F(N_i) \bigl) . \]
\end{prop}

\begin{proof}
For any $i$ choose some element $a_i \in \bar{A}$ such that 
$\bar{f}(a_i) = e_i$.
Consider the noncommutative rings $\opn{End}_{\cat{D}(B)}(M_i)$ and 
$\opn{End}_{\cat{D}(A)}(F(M_i))$. 
There is a commutative diagram (of noncommutative rings) 
\[  \UseTips \xymatrix @C=5ex @R=5ex {
\bar{A}
\ar[r]
\ar[d]_{ \bar{f} }
&
\opn{End}_{\cat{D}(A)}  ( F(M_i) )
\\
\bar{B}
\ar[r]
&
\opn{End}_{\cat{D}(B)} ( M_i )
\ar[u]_{F} \ . 
} \]
Cf.\ Proposition \ref{prop:311}.

Take two distinct indices $i, j$.
We know that $e_i \cd \opn{id}_{M_i} = \opn{id}_{M_i}$
in $\opn{End}_{\cat{D}(B)}(M_i)$,
and $e_i \cd \opn{id}_{M_j} = 0$ in $\opn{End}_{\cat{D}(B)}(M_j)$.
Therefore
$a_i \cd \opn{id}_{F(M_i)} = \opn{id}_{F(M_i)}$
in $\opn{End}_{\cat{D}(A)}(F(M_i))$,
and $a_i \cd \opn{id}_{F(M_j)} = 0$ in $\opn{End}_{\cat{D}(A)}(F(M_j))$.

Consider any morphism $\phi : F(M_i) \to F(N_j)$ in $\cat{D}(A)$.
Then 
\[ \begin{aligned}
& \phi =  \opn{id}_{F(M_j)} \circ \, \phi \circ \opn{id}_{F(M_i)}  = 
\opn{id}_{F(M_j)} \circ \, \phi \circ (a_i \cd \opn{id}_{F(M_i)}) \\
& \qquad = (a_i \cd \opn{id}_{F(M_j)}) \circ \phi \circ \opn{id}_{F(M_i)}  = 0 .
\end{aligned} \]
\end{proof}

\section{Perfect DG Modules} \label{sec:perfect}

Recall that all DG rings are now commutative by default (Convention 
\ref{conv:100}). In particular all rings are commutative. For a DG ring $A$, 
its reduction is $\bar{A} = \opn{H}^0(A)$. 

Let $A$ be a  DG ring and $s \in \bar{A}$ an element. The 
localization $A_s$ was defined in Definition \ref{dfn:40}. The notion of
covering sequence of $\bar{A}$ was introduced in Definition \ref{dfn:70},
and finite semi-free DG modules were introduced in Definition \ref{dfn:311}.

If $A$ is a ring, there are two definitions in the literature (see 
\cite[Expos\'e I]{SGA-6}) of a {\em perfect complex of $A$-modules}. Let us 
recall them; but in order to make a 
distinction, we shall add attributes to the name ``perfect''. The first 
definition is this: a complex $M \in \cat{D}(A)$ is {\em geometrically perfect}
if there is a covering sequence  
$(s_1, \ldots, s_n)$ of $A$, and for every $i$ there 
is an isomorphism $A_{s_i} \ot_A M \cong P_i$ in  $\cat{D}(A_{s_i})$,
 where $P_i$ is a bounded complex of finite free $A_{s_i}$-modules.  
The second definition is: a complex $M \in \cat{D}(A)$ is {\em algebraically 
perfect} if there is an isomorphism $M \cong P$ in $\cat{D}(A)$, where $P$ is a 
bounded complex of finite projective $A$-modules. 
It is known that the two definitions are equivalent - see 
\cite[Expos\'e I]{SGA-6} or 
\cite[\href{http://stacks.math.columbia.edu/tag/08CL}{[Section 08E4}]{SP}.
Therefore it is safe to use the name ``perfect complex'' without further 
qualification. 

Here are our generalizations to DG rings. 

\begin{dfn} \label{dfn:25}
Let $A$ be a DG ring, and let $M$ be a DG $A$-module. 
We say that $M$ is {\em geometrically perfect} if there is a covering sequence  
$(s_1, \ldots, s_n)$ of $\bar{A}$, and for every $i$ there 
is an isomorphism $A_{s_i} \ot_A M \cong P_i$
in $\cat{D}(A_{s_i})$, for some finite semi-free DG $A_{s_i}$-module $P_i$.
\end{dfn} 

Clearly when $A$ is a ring, we recover the classical definition; but
for DG rings this is a new definition. 

Before stating the second definition, we need to recall some terminology on 
triangulated categories. Let $\cat{D}$ be a triangulated category. 
A full subcategory $\cat{E} \subseteq \cat{D}$ is called {\em 
epaisse} if it is closed under shifts, cones and direct summands (so $\cat{E}$ 
itself is triangulated). We say that an object $P \in \cat{E}$ is a {\em 
classical generator} of  $\cat{E}$ if this is the smallest epaisse subcategory 
of $\cat{D}$ that contains $P$. In this case, any object of $\cat{E}$ can be 
obtained from $P$ by finitely many shifts, direct summands and cones. See 
\cite{Ri} and \cite{BV}. 

\begin{dfn} \label{dfn:545}
Let $A$ be a DG ring, and let $M$ be a DG $A$-module. 
We say that $M$ is {\em algebraically perfect} if $M$ belongs to the epaisse 
subcategory of $\cat{D}(A)$ classically generated by $A$. 
\end{dfn} 

It is not hard to see that when $A$ is a ring, this 
definition coincides with the definition of algebraically perfect complexes 
given before. Definition \ref{dfn:545} is not new -- it already appeared in 
\cite{ABIM} (without the attribute ``algebraically''). 

In our paper we are interested in geometrically perfect DG modules. 
However, as we shall prove in Corollary \ref{cor:550}, these turn out to be 
the same as algebraically perfect DG modules.

\begin{prop} \label{prop:81}
Let $A \to B$ be a homomorphism of DG rings, and let $M$ be a geometrically 
perfect DG $A$-module. Then $B \ot^{\mrm{L}}_A M$ is a geometrically perfect DG 
$B$-module. 
\end{prop}

\begin{proof}
Let $\bar{f} : \bar{A} \to \bar{B}$ denote the induced ring homomorphism, and 
let $\bsym{s} = \lb (s_1, \ldots, s_n)$ and $P_i$ be as in the definition. 
Define 
$t_i := \bar{f}(s_i)$ and $N := B \ot^{\mrm{L}}_A M$. Then 
$(t_1, \ldots, t_n)$ is a covering sequence of $\bar{B}$,
$Q_i := B_{t_i} \ot_{A_{s_i}} P_i$ is a finite semi-free DG $B_{t_i}$-module, 
and $B_{t_i} \ot_B N \cong Q_i$ in 
$\cat{D}(B_{t_i})$.
\end{proof}

\begin{lem}  \label{lem:34}
Let $M$ be a geometrically perfect DG $A$-module. 
\begin{enumerate}
\item $M$ belongs to $\cat{D}^{-}(A)$. 
\item If $A$ is cohomologically pseudo-noetherian, then $M$ 
belongs to $\cat{D}^{-}_{\mrm{f}}(A)$. 
\end{enumerate}
\end{lem}

\begin{proof}
Step 1. Assume $M \cong P$, where $P$ is a finite semi-free DG $A$-module. 
We know that $A \in \cat{D}^{-}(A)$, and that $A \in \cat{D}^{-}_{\mrm{f}}(A)$
in the  cohomologically pseudo-noetherian case. Now $P$ is obtained from $A$ 
by finitely many shifts and cones, and hence $P$ also belongs to 
$\cat{D}^{-}(A)$, and to $\cat{D}^{-}_{\mrm{f}}(A)$ in the cohomologically 
pseudo-noetherian case.

\medskip \noindent
Step 2. Let $s_1, \ldots, s_n \in \bar{A}$ and
$P_i \in \cat{D}(A_{s_i})$ be as in Definition \ref{dfn:25}. 
We know that for each $i$, 
$\bar{A}_{s_i} \ot_{\bar{A}} \opn{H}(M) \cong \opn{H}(P_i)$
as graded modules over $\bar{A}_{s_i}$.
Step 1 tells us that $P_i \in \cat{D}^{-}(A_{s_i})$, and that 
$P_i \in \cat{D}^{-}_{\mrm{f}}(A_{s_i})$
in the  cohomologically pseudo-noetherian case.
From the  faithfully flat ring homomorphism 
$\bar{A} \to  \prod_i \bar{A}_{s_i}$
we deduce that $\opn{H}(M)$ is bounded above (cf.\ Proposition \ref{prop:300}). 
In the cohomologically pseudo-noetherian case, descent implies that each 
$\opn{H}^j(M)$ is finite over $\bar{A}$. Cf.\ 
\cite[\href{http://stacks.math.columbia.edu/tag/066D}{Lemma 066D}]{SP},
noting that an $\bar{A}$ module is finite iff it is $0$-pseudo-coherent. 
\end{proof}

Let $L, M, N \in \cat{D}(A)$. There is a canonical morphism
\begin{equation} \label{eqn:400}
\psi_{L, M, N} : \opn{RHom}_A(L, M) \ot^{\mrm{L}}_A N \to 
\opn{RHom}_A(L, M \ot^{\mrm{L}}_A N) 
\end{equation}
in $\cat{D}(A)$, which is functorial in the three arguments. 
If we choose a K-projective resolution 
$\til{L} \to L$, and a K-flat resolution $\til{N} \to N$, 
then the morphism $\psi_{L, M, N}$ is represented by the homomorphism 
\begin{equation} \label{eqn:401}
\til{\psi}_{\til{L}, M, \til{N}} : \opn{Hom}_A(\til{L}, M) \ot_A \til{N} 
\to 
\opn{Hom}_A(\til{L}, M \ot_A \til{N})
\end{equation}
in $\cat{C}(A)$, where 
\[ \til{\psi}_{\til{L}, M, \til{N}}(\al \ot n)(l) :=
(-1)^{j k} \cd \al(l) \ot n \]
for $\al \in \opn{Hom}_A(\til{L}, M)^i$,
$n \in \til{N}^j$ and $l \in \til{L}^k$. 

\begin{lem} \label{lem:400}
$L, M, N \in \cat{D}(A)$, and assume $L$ is geometrically perfect. Then the 
morphism $\psi_{L, M, N}$ in formula \tup{(\ref{eqn:400})} is an isomorphism.
\end{lem}

\begin{proof}
Step 1. Assume $L \cong \til{L}$ in $\cat{D}(A)$, where $\til{L}$ is a finite 
semi-free DG $A$-module. Choose a  K-flat resolution $\til{N} \cong N$.
Then the homomorphism $\til{\psi}_{\til{L}, M, \til{N}}$ in (\ref{eqn:401}) is 
in fact bijective. 

\medskip \noindent
Step 2. Let $\bsym{s} = (s_1, \ldots, s_n)$ be a covering sequence of 
$\bar{A}$, and for every $i$ let us write  $A_i := A_{s_i}$. 
We assume that there are isomorphisms
$A_{i} \ot_A L \cong \til{L}_i$ in $\cat{D}(A_{i})$, such that $\til{L}_i$ is 
a finite semi-free DG $A_{i}$-module; cf.\ Definition \ref{dfn:25}. 
In this step we assume that 
$M \cong A_{i} \ot_A M$ in $\cat{D}(A)$ for some index $i$. 
Let's write $L_i := A_{i} \ot_A  L$,
$M_i := A_{i} \ot_A  M$  and $N_i := A_{i} \ot_A  N$. 
Then, and using adjunction with respect to the 
homomorphism $A \to  A_{i}$, we get isomorphisms 
\[ \begin{aligned}
& \opn{RHom}_A(L, M) \ot^{\mrm{L}}_A N \cong 
\opn{RHom}_A(L, M_i) \ot^{\mrm{L}}_A N 
\\
& \quad \cong \opn{RHom}_{A_{i}}(L_i, M_i) \ot^{\mrm{L}}_A N \cong 
\opn{RHom}_{A_{i}}(L_i, M_i) \ot^{\mrm{L}}_{A_{i}} N_i 
\end{aligned} \]
and 
\[ \opn{RHom}_A(L, M \ot^{\mrm{L}}_A N) \cong 
\opn{RHom}_A(L, M_i \ot^{\mrm{L}}_A N)
\cong \opn{RHom}_{A_{i}}(L_i, M_i \ot^{\mrm{L}}_A N_i) \]
in $\cat{D}(A)$. By step 1, the morphism 
\[ \psi_{L_i, M_i, N_i} : 
\opn{RHom}_{A_{i}}(L_i, M_i) \ot^{\mrm{L}}_{A_{i}} N_i
\to \opn{RHom}_{A_{i}}(L_i, M_i \ot^{\mrm{L}}_{A_{i}} N_i)  \]
is an isomorphism.

\medskip \noindent
Step 3. We keep the covering sequence $\bsym{s} = (s_1, \ldots, s_n)$ from step 
2. Since the \v{C}ech resolution 
$\opn{c}_N : M \to \opn{C}(M; \bsym{s})$ is a quasi-isomorphism
(Proposition \ref{prop:70}), it suffices to prove that 
$\psi_{L, M', N}$ is an isomorphism, where 
$M' :=  \opn{C}(M; \bsym{s})$. 

The DG $A^0$-module $\opn{C}(A^0; \bsym{s})$ is filtered by degree: 
\[ \mu^k(\opn{C}(A^0; \bsym{s})) := 
\bigoplus_{j \geq k}  \, \opn{C}^{j}(A^0; \bsym{s}) . \]
This is a decreasing filtration of finite length, because 
$\mu^0(\opn{C}(A^0; \bsym{s})) = \opn{C}(A^0; \bsym{s})$ and 
$\mu^{n}(\opn{C}(A^0; \bsym{s})) = 0$.
Now by definition 
$\opn{C}(M; \bsym{s}) = \opn{C}(A^0; \bsym{s}) \ot_{A^0} M$,
so we get an induced filtration of finite length
$\bigl\{ \mu^k(\opn{C}(M; \bsym{s})) \bigr\}_{k \in \Z}$ 
on the DG module $\opn{C}(M; \bsym{s})$, with 
\begin{equation} \label{eqn:410}
\mu^k \bigl( \opn{C}(M; \bsym{s}) \bigr) := 
\mu^k \bigl( \opn{C}(A^0; \bsym{s}) \bigr) \ot_{A^0} M .  
\end{equation}
For every $k$ the filtration gives rise to an 
exact sequence of DG $A$-modules, that becomes a distinguished triangle
\begin{equation} \label{eqn:411}
\mu^{k+1} \bigl( \opn{C}(M; \bsym{s}) \bigr) \to 
\mu^k \bigl( \opn{C}(M; \bsym{s}) \bigr) \to 
\opn{gr}_{\mu}^k \bigl( \opn{C}(M; \bsym{s}) \bigr) \xar{\vartriangle}
\end{equation}
in $\cat{D}(A)$. Thus to prove that $\psi_{L, M', N}$ is an isomorphism, 
it suffices to prove that 
$\psi_{L, M'_k, N}$ is an isomorphism, where 
\begin{equation} \label{eqn:412}
M'_k := \opn{gr}_{\mu}^k \bigl( \opn{C}(M; \bsym{s}) \bigr)
\cong \opn{C}^k(A^0; \bsym{s})[-k] \ot_{A^0} M . 
\end{equation}

But $M'_k$ is a finite direct sum of shifts of the DG modules 
\[ M''_{\bsym{i}} := \opn{C}(A^0; \bsym{s})(\bsym{i}) \ot_{A^0} M ; \]
see formula (\ref{eqn:342}).
Thus we reduce the problem to proving that 
$\psi_{L, M''_{\bsym{i}}, N}$ is an isomorphism. Because  
$M''_{\bsym{i}}$ satisfies the assumption in step 2, we are done. 
\end{proof}

\begin{thm} \label{thm:50}
Let $A$ be a DG ring, and let $M$ be a DG $A$-module. 
The following two conditions are equivalent\tup{:}
\begin{enumerate}
\rmitem{i} $M$ is geometrically perfect. 
\rmitem{ii} $M$ belongs to $\cat{D}^{-}(A)$, and the DG $\bar{A}$-module 
$\bar{A} \ot^{\mrm{L}}_A M$ is geometrically perfect.
\end{enumerate}
If $A$ is cohomologically pseudo-noetherian, then these two conditions 
are equivalent to\tup{:}
\begin{enumerate}
\rmitem{iii} $M$ is in $\cat{D}^{-}_{\mrm{f}}(A)$, and it has finite projective 
dimension relative to $\cat{D}(A)$.
\end{enumerate}
\end{thm}

When $A$ is a ring, i.e.\ $A = \bar{A}$, the equivalence (i) $\Leftrightarrow$ 
(ii) is almost a tautology, and the equivalence (i) $\Leftrightarrow$ (iii)
was already proved in \cite[Expos\'e I]{SGA-6}. But for a genuine DG ring this 
is a new result. 

\begin{proof}
(i) $\Leftrightarrow$ (ii): 
According to Lemma \ref{lem:34} we have $M \in \cat{D}^{-}(A)$.
Write $\bar{M} := \bar{A} \ot^{\mrm{L}}_A M$. 
Consider a covering sequence $(s_1, \ldots, s_n)$ of $\bar{A}$.
Let us write $A_i := A_{s_i}$, $M_i := A_i \ot_{A} M$ and 
$\bar{A}_i := \opn{H}^0(A_{s_i})$. 
For every $i$ there is an isomorphism 
$\bar{A}_{i} \ot_{\bar{A}} \bar{M} \cong 
\bar{A}_{i} \ot^{\mrm{L}}_{A_i} M_i$
in $\cat{D}(A_{s_i})$.
Using Proposition \ref{prop:305} we see that $M_i$ is isomorphic in 
$\cat{D}(A_i)$ to a finite semi-free DG $A_i$-module iff 
$\bar{A}_{i} \ot^{\mrm{L}}_{A_i} M_i$ is isomorphic in 
$\cat{D}(\bar{A}_{i})$ to a finite semi-free DG $\bar{A}_{i}$-module.

\medskip \noindent 
(i) $\Rightarrow$ (iii): 
Here  $A$ is cohomologically pseudo-noetherian. 
Lemma \ref{lem:34} says that $M \in \cat{D}^{-}_{\mrm{f}}(A)$.
To prove that $M$ has finite projective dimension relative to 
$\cat{D}(A)$, we have to bound 
$\opn{H} \bigl( \opn{RHom}_{A}( M , N) \bigr)$ 
in terms of $\opn{H}(N)$ for any $N \in \cat{D}(A)$.
Choose a covering sequence $(s_1, \ldots, s_n)$ of $\bar{A}$ and finite 
semi-free DG $A_i$-modules $P_i$ as in Definition \ref{dfn:25}, where 
$A_i := A_{s_i}$. 
Let $d_0 \leq d_1$ be integers such that every $P_i$ is generated in the 
integer interval $[d_0, d_1]$ (see Definition \ref{dfn:390}). 
Using Lemma \ref{lem:400} for the isomorphism $\cong^{\dag}$,
and adjunction, we obtain isomorphisms 
\[ \begin{aligned}
& 
A_i \ot_A \opn{RHom}_{A}( M , N) \cong^{\dag}
\opn{RHom}_{A}( M , A_i \ot_A N) 
\\
& \quad 
\cong \opn{RHom}_{A_i}( A_i \ot_A M , A_i \ot_A N) \cong 
\opn{Hom}_{A_i}( P_i, A_i \ot_A N) 
\end{aligned} \]
in $\cat{D}(A)$.
Using Proposition \ref{prop:390}, this proves that 
\[ \opn{con} \bigl( \opn{H} (\opn{RHom}_{A}( M , N)) \bigr) \subseteq
\opn{con} ( \opn{H}(N) ) - [d_0, d_1] . \]
We conclude that the projective dimension of $M$ relative to 
$\cat{D}(A)$ is $\leq d_1 - d_0$.

\medskip \noindent 
(iii) $\Rightarrow$ (ii): 
Here again $A$ is cohomologically pseudo-noetherian. 
Let's write $\bar{M} := \bar{A} \ot^{\mrm{L}}_A M$.
Because $M \in \cat{D}^{-}_{\mrm{f}}(A)$, we 
can find a pseudo-finite semi-free resolution $P \to M$ in 
$\cat{C}(A)$. 
Thus $\bar{M} \cong \bar{A} \ot_A P$ belongs to 
$\cat{D}^{-}_{\mrm{f}}(\bar{A})$.

For every $\bar{N} \in \cat{D}(\bar{A})$ we have, by adjunction, 
$\opn{RHom}_{\bar{A}}( \bar{M}, \bar{N}) \cong 
\opn{RHom}_{A}(M, \bar{N})$.
This shows that the projective dimension of $\bar{M}$ relative to 
$\cat{D}(\bar{A})$ is finite. But this just means that the complex 
$\bar{M}$ has finite projective dimension over the ring $\bar{A}$. 
In particular $\bar{M}$ belongs to $\cat{D}^{-}_{\mrm{f}}(\bar{A})$.
The usual syzygy argument shows that there is a quasi-isomorphism 
$\bar{P} \to \bar{M}$ in 
$\cat{C}(\bar{A})$, for some bounded complex of finite projective 
$\bar{A}$-modules $\bar{P}$. But locally on $\opn{Spec} \bar{A}$ each 
$\bar{P}^i$ is a free $\bar{A}$-module; and hence $\bar{M}$ is geometrically 
perfect. 
\end{proof}

\begin{rem} \label{rem:80}
Possibly one could remove the pseudo-noetherian hypothesis on $A$ in condition 
(iii) of  Theorem \ref{thm:50}. The new condition on $M$ would most likely be 
this:
\begin{enumerate}
\item[(iii')] The DG $A$-module $M$ is pseudo-coherent, and it has finite flat 
dimension relative to $\cat{D}^{\mrm{b}}(A)$.
\end{enumerate}
This would require a detailed study of pseudo-coherent DG $A$-modules. 
Cf.\ \lb \cite[Expos\'e I]{SGA-6}, 
\cite[\href{http://stacks.math.columbia.edu/tag/0657}{Definition 0657}]{SP}
and
\cite[\href{http://stacks.math.columbia.edu/tag/0658}{Lemma 0658}]{SP}.
\end{rem}

Recall that a DG $A$-module $M$ is called a {\em compact object of 
$\cat{D}(A)$} if for any collection $\{ N_z \}_{z \in Z}$ of DG $A$-modules, 
the canonical homomorphism 
\begin{equation} \label{eqn:402}
\bigoplus_{z \in Z} \, \opn{Hom}_{\cat{D}(A)}(M, N_z) \to 
\opn{Hom}_{\cat{D}(A)} \Bigl( M, \bigoplus\nolimits_{z \in Z} \, N_z \Bigr) 
\end{equation}
is bijective. (In general this is only injective.)
It is known that for a ring $A$, compact and perfect are the same (see 
\cite[Section 6]{Ri}, \cite[Example 1.13]{Ne}, or 
\cite[\href{http://stacks.math.columbia.edu/tag/07LT}{Proposition 07LT}]{SP}).
It turns out that this is also true for a DG ring. 

First we need to know that being compact is a local property on 
$\opn{Spec} \bar{A}$. This is very similar to arguments found in \cite{Ne}.

\begin{lem} \label{lem:420}
Let $A$ be a DG ring, let $M$ be a DG $A$-module, and let 
$(s_1, \ldots, s_n)$ be a covering sequence of $\bar{A}$.
The following conditions are equivalent.
\begin{enumerate}
\rmitem{i} $M$ is a compact object of $\cat{D}(A)$. 

\rmitem{ii} For every $i$ the DG $A_{s_i}$-module $A_{s_i} \ot_A M$
 is a compact object of $\cat{D}(A_{s_i})$. 
\end{enumerate}
\end{lem}

\begin{proof}
(i) $\Rightarrow$ (ii): This is the easy implication. 
We write $A_i := A_{s_i}$ and $M_i := A_{i} \ot_A M$. 
Let $F_i : \cat{D}(A_{i}) \to \cat{D}(A)$ be the restriction functor. It 
commutes with all direct sums. Given a collection 
$\{ N_z \}_{z \in Z}$ in $\cat{D}(A_{i})$, we have canonical isomorphisms
\[ \begin{aligned}
& \opn{Hom}_{\cat{D}(A_{i})} \bigl( M_i , \,
\bigoplus\nolimits_{z} \, N_z \bigr) \cong 
\opn{Hom}_{\cat{D}(A_{})} \bigl( M , F_i \bigl( \bigoplus\nolimits_{z} \, 
N_z \bigr) \bigr)
\\
& \qquad 
\cong \opn{Hom}_{\cat{D}(A_{})} \bigl( M , \, \bigoplus\nolimits_{z}  
F_i(N_z) \bigr)
\cong \bigoplus\nolimits_{z} \opn{Hom}_{\cat{D}(A_{})} ( M , F_i(N_z))
\\
& \qquad 
\cong \bigoplus\nolimits_{z}
\opn{Hom}_{\cat{D}(A_{i})} ( M_i , N_z ) \ . 
\end{aligned} \]
We use the adjunction for $F_i$ and the fact that $M$ is compact. 
The conclusion is that $M_i$ is compact. 

\medskip \noindent 
(ii) $\Rightarrow$ (i): For any DG $A$-module $N$ we have the 
\v{C}ech resolution $\opn{c}_N : N \to \opn{C}(N; \bsym{s})$
from Definition \ref{dfn:72} and Proposition \ref{prop:70}.
Because 
\begin{equation*} 
\opn{C}(N; \bsym{s}) = \opn{C}(A^0; \bsym{s}) \ot_{A^0} N \cong
\opn{C}(A; \bsym{s}) \ot_A N , 
\end{equation*}
this functor commutes with all direct sums. Thus the canonical homomorphism
\begin{equation} \label{eqn:104}
\bigoplus\nolimits_{z \in Z} \opn{C}(N_z; \bsym{s}) \to 
\opn{C} \Bigl( \bigoplus\nolimits_{z \in Z} N_z \, ; \bsym{s} \Bigr)
\end{equation}
is an isomorphism in $\cat{C}(A)$. Using (\ref{eqn:104}) we obtain 
a commutative diagram of $\bar{A}$-modules
\begin{equation} \label{eqn:73}
\UseTips \xymatrix @C=5ex @R=5ex {
\displaystyle\bigoplus\nolimits_{z} \, \opn{Hom}_{\cat{D}(A)}(M, N_z) 
\ar[r]
\ar[d]^{\cong}
&
\opn{Hom}_{\cat{D}(A)} \Bigl( M, \displaystyle\bigoplus\nolimits_{z} N_z \Bigr)
\ar[d]^{\cong}
\\
\displaystyle\bigoplus\nolimits_{z} \, \opn{Hom}_{\cat{D}(A)} \bigl( M, 
\opn{C}(N_z; \bsym{s}) \bigr) 
\ar[r]
&
\opn{Hom}_{\cat{D}(A)} \Bigl( M, \displaystyle\bigoplus\nolimits_{z}
\opn{C}(N_z; \bsym{s}) \Bigr) 
} 
\end{equation}
where the vertical arrows are bijections. So it suffices to prove that the 
lower horizontal arrow is a bijection.

Consider the finite length filtration
$\bigl\{ \mu^k(\opn{C}(N_z; \bsym{s})) \bigr\}_{k \in \Z}$ 
on the DG module \lb $\opn{C}(N_z; \bsym{s})$, as in formula (\ref{eqn:410}). 
Passing to the associated distinguished triangles,
and using induction on $k$, as was done in the proof of Lemma \ref{lem:400}, 
we reduce the problem to the verification that 
\begin{equation} \label{eqn:413}
\bigoplus\nolimits_{z} \, 
\opn{Hom}_{\cat{D}(A)}(M, A_{\bsym{i}} \ot_{A} N_z ) 
\to 
\opn{Hom}_{\cat{D}(A)} \Bigl( M,  \bigoplus\nolimits_{z}
A_{\bsym{i}} \ot_{A} N_z  \Bigr)
\end{equation}
is a bijection, where 
$A_{\bsym{i}} :=  A^0_{s_{i_0}} \ot_{A^0} \cdots \ot_{A^0} A^0_{s_{i_k}}$
for some strictly increasing sequence
$\bsym{i} = (i_0, \ldots, i_k)$ in the integer interval $[1, n]$.
Let's write $A' := A_{s_{i_0}}$. Adjunction for the DG ring homomorphism 
$A \to A'$ allows us to replace (\ref{eqn:413}) with the homomorphism
\begin{equation} \label{eqn:74}
\bigoplus\nolimits_{z} \, 
\opn{Hom}_{\cat{D}(A')}(A' \ot_{A} M, A_{\bsym{i}} \ot_{A} N_z ) \to 
\opn{Hom}_{\cat{D}(A')} \Bigl( A' \ot_{A} M,  
\bigoplus\nolimits_{z} A_{\bsym{i}} \ot_{A} N_z  \Bigr) .
\end{equation}
But we are assuming that $A' \ot_A M$ is compact in $\cat{D}(A')$; so
(\ref{eqn:74}) is bijective.
\end{proof}

\begin{lem} \label{lem:73}
If $M$ is a compact object of $\cat{D}(A)$, then it belongs to 
$\cat{D}^-(A)$.
\end{lem}

\begin{proof}
This is an argument from \cite{Ri}, slightly improved in the proof of
\cite[\href{http://stacks.math.columbia.edu/tag/07LT}{Proposition 07LT}]{SP}. 

Suppose  $\{ N_z \}_{z \in Z}$ is a collection of DG $A$-modules. 
Given a morphism $\psi : M \to \bigoplus_{z \in Z} N_z$ in $\cat{D}(A)$,
there is a finite subset $Z_0 \subseteq Z$ such that $\psi$ factors through 
$\bigoplus_{z \in Z_0} N_z$. So for any $z \notin Z_0$, the component 
$\psi_z : M \to N_z$ of $\psi$ is zero. 

For every $k \geq 0$ consider the smart truncation $\opn{smt}^{\geq k}(M)$
from (\ref{eqn:425}).
There is a canonical surjective homomorphism 
$\phi_k : M \to \opn{smt}^{\geq k}(M)$ in $\cat{C}(A)$,
and we know that $\opn{H}^l(\phi_k)$ is an isomorphism for all $l \geq k$.
Consider the homomorphism
\[ \phi : M \to  \bigoplus_{k \in \N} \, \opn{smt}^{\geq k}(M) \ , \ 
\phi := \sum \phi_k   \]
in $\cat{C}(A)$. Let $\psi := \opn{Q}(\phi)$; so the $k$-th component of $\psi$ 
is 
$\psi_k := \opn{Q}(\phi_k) : M \to \opn{smt}^{\geq k}(M)$.
As explained in the paragraph above, there is an integer $k_0$ such that 
$\psi_{k_0 + 1} = 0$. Therefore 
\[ \opn{H}^l(\psi_{k_0 + 1}) = \opn{H}^l(\phi_{k_0 + 1}) : 
\opn{H}^l(M) \to \opn{H}^l(\opn{smt}^{\geq k_0 + 1}(M)) \]
is zero for all $l$. We see that 
$\opn{H}^l(M) = 0$ for all  $l \geq k_0 + 1$.
\end{proof}

\begin{thm} \label{thm:74}
Let $A$ be a DG ring, and let $L$ be a DG $A$-module.
The following three conditions are equivalent\tup{:} 
\begin{enumerate}
\rmitem{i} $L$ is a geometrically perfect DG $A$-module.  

\rmitem{ii} $L$ is a compact object of $\cat{D}(A)$. 

\rmitem{iii} For any $M, N \in \cat{D}(A)$, the canonical morphism 
\[ \psi_{L, M, N} : \opn{RHom}_A(L, M) \ot^{\mrm{L}}_A N \to 
\opn{RHom}_A(L, M \ot^{\mrm{L}}_A N) \]
in $\cat{D}(A)$ is an isomorphism.
\end{enumerate}
\end{thm}

See the text just after formula (\ref{eqn:400}) for a description of the  
morphism $\psi_{L, M, N}$ in condition (iii).

\begin{proof}
(i) $\Rightarrow$ (ii): Since a finite semi-free DG module 
is clearly compact, this follows from Lemma \ref{lem:420}.

\medskip \noindent 
(ii) $\Rightarrow$ (i): Assume $L$ is compact in $\cat{D}(A)$. 
Consider the DG $\bar{A}$-module $\bar{L} := \bar{A} \ot^{\mrm{L}}_A L$.
Adjunction shows that 
\[ \opn{Hom}_{\cat{D}(\bar{A})}(\bar{L}, M) \cong 
\opn{Hom}_{\cat{D}(A)}(L, F(M)) \]
functorially for $M \in \cat{D}(\bar{A})$. Here $F$ is the forgetful functor, 
that commutes with all direct sums. Thus $\bar{L}$ is a compact object of 
$\cat{D}(\bar{A})$. Now by \cite[Section 6]{Ri},  \cite[Example 1.13]{Ne} or 
\cite[\href{http://stacks.math.columbia.edu/tag/07LT}{Proposition 07LT}]{SP})
the DG $\bar{A}$-module $\bar{L}$ is algebraically perfect.
So there is an isomorphism $L \cong \bar{P}$ in $\cat{D}(\bar{A})$, where 
$\bar{P}$ is a bounded complex of finite projective $\bar{A}$-modules. But 
locally on $\opn{Spec} \bar{A}$ each $\bar{P}^i$ is a free module. 
Thus $\bar{L}$ is geometrically perfect.
By the lemma above we know that $L \in \cat{D}^-(A)$.
The implication (ii) $\Rightarrow$ (i) in Theorem \ref{thm:50} says that $L$ 
is geometrically perfect.

\medskip \noindent 
(i) $\Rightarrow$ (iii): This is Lemma \ref{lem:400}.

\medskip \noindent 
(iii) $\Rightarrow$ (ii): Take any collection of DG $A$-modules 
$\{ N_z \}_{z \in Z}$, and define 
$N := \lb \bigoplus_{z \in Z} N_z$. 
By assumption, the any $z$ the morphism 
\[ \psi_{L, A, N_z} : \opn{RHom}_A(L, A) \ot^{\mrm{L}}_A N_z \to 
\opn{RHom}_A(L, N_z) \]
is an isomorphism. Since derived tensor products commute with all direct sums, 
we get an isomorphism 
\[ \phi : \opn{RHom}_A(L, A) \ot^{\mrm{L}}_A N \iso
\bigoplus_{z \in Z} \, \opn{RHom}_A(L, N_z) . \]
Now the functor $\opn{H}^0$ also commutes with all direct sums. So we get a 
commutative diagram of $\bar{A}$-modules 
\[ \UseTips \xymatrix @C=13ex @R=5ex {
\opn{H}^0  \bigl( \opn{RHom}_A(L, A) \ot^{\mrm{L}}_A N \bigr)
\ar[r]^(0.54){ \opn{H}^0(\psi_{L, A, N}) }  
\ar[d]_{ \opn{H}^0(\phi) }
&
\opn{H}^0  \bigl( \opn{RHom}_A(L, N) \bigr) 
\ar[d]^{\cong}
\\
\displaystyle\bigoplus_{z \in Z} \, \opn{Hom}_{\cat{D}(A)}(L, N_z) 
\ar[r]^{ \opn{can} }
&
\opn{Hom}_{\cat{D}(A)}(L, N)
} \]
in which the vertical arrows are isomorphisms. 
Our assumption says that \lb $\opn{H}^0(\psi_{L, A, N})$ is an isomorphism.
Therefore the bottom arrow (marked ``$\opn{can}$'') is an isomorphism too.
But this is the morphism (\ref{eqn:402}).  
\end{proof}

\begin{cor} \label{cor:550}
Let $A$ be a DG ring and $M$ a DG $A$-module. The following two conditions are 
equivalent:  
\begin{enumerate}
\rmitem{i} $M$ is geometrically perfect \tup{(}Definition \tup{\ref{dfn:25})}. 
\rmitem{ii} $M$ is algebraically perfect \tup{(}Definition \tup{\ref{dfn:545})}.
\end{enumerate}
\end{cor}

\begin{proof}
By Theorem \ref{thm:74}, $M$ is geometrically perfect iff it is a compact 
object of $\cat{D}(A)$. On the other hand, it is well-known (see 
\cite[Proposition 2.2.4]{BV}) that $M$ is algebraically 
perfect iff it is a compact object of $\cat{D}(A)$.
\end{proof}

\begin{conv}
From here on we use the expression ``perfect DG module'', rather than the 
two longer yet equivalent expressions. 
\end{conv}

\begin{rem} \label{rem:540}
Suppose $A$ is a noncommutative DG ring. Definition \ref{dfn:25} is worthless 
here: even if $A$ happens to be nonpositive, still the ring $\bar{A}$ is 
noncommutative, so we cannot localize on $\opn{Spec} \bar{A}$.

However, Definition \ref{dfn:545} is fine when $A$ is noncommutative, and also 
when $A$ has nontrivial positive components. So we can talk about algebraically 
perfect DG $A$-modules for any $A \in \cat{DGR}$. Indeed, this is the  
definition of perfect DG module that was used in \cite{ABIM}. 
Results of \cite{Ri}, \cite{BV} and \cite{ABIM} say that a DG $A$-module $M$ 
is algebraically perfect iff it is a compact object of $\cat{D}(A)$, iff it is 
a direct summand, in $\cat{D}(A)$, of a finite semi-free DG $A$-module. 
\end{rem}

\section{Tilting DG Modules} \label{sec:tilting}

Recall that all DG rings here are commutative (Convention \ref{conv:100}). In 
particular all rings are commutative. 

\begin{dfn}
Let $A$ be a DG ring.
A DG $A$-module $P$ is called a {\em tilting DG module} if there 
exists some DG $A$-module $Q$ such that $P \ot^{\mrm{L}}_A Q \cong A$
in $\cat{D}(A)$.  
\end{dfn}

The DG module $Q$ in the definition is called a {\em quasi-inverse} of $P$.
Due to the symmetry of the operation $- \ot^{\mrm{L}}_A- $, $Q$ is also 
tilting. It is easy to see that the quasi-inverse $Q$ is unique, up to a 
nonunique isomorphism. If $P_1$ and $P_2$ are tilting then so is 
$P_1 \ot^{\mrm{L}}_A P_2$; this is because of the associativity of  
$- \ot^{\mrm{L}}_A- $. Hence the next definition makes sense. 

\begin{dfn} \label{dfn:350}
The {\em commutative derived Picard group} of $A$ is the abelian group \lb 
$\opn{DPic}(A)$, whose elements are the isomorphism classes, in $\cat{D}(A)$,
of tilting DG $A$-modules. The product is induced by the operation
$- \ot^{\mrm{L}}_A -$, and the unit element is the class of  $A$. 
\end{dfn}

\begin{lem} \label{lem:80}
Let $f : A \to B$ be a homomorphism of DG rings.
\begin{enumerate}
\item For any $M, N \in \cat{D}(A)$ there is an isomorphism 
\[  (B \ot^{\mrm{L}}_A M) \ot^{\mrm{L}}_B (B \ot^{\mrm{L}}_A N) \cong 
B \ot^{\mrm{L}}_A (M \ot^{\mrm{L}}_A N) \]
in $\cat{D}(A)$. 

\item If $P$ is a tilting DG $A$-module, then $B \ot^{\mrm{L}}_A P$ is a 
tilting DG $B$-module.

\item If $f$ is a quasi-isomorphism and $Q$ is a tilting DG $B$-module, then 
$\opn{rest}_f(Q)$ is a tilting DG $A$-module. 
\end{enumerate}
\end{lem}

\begin{proof}
(1) Choose K-flat resolutions $\til{M} \to M$ and $\til{N} \to N$
over $A$. Then $B \ot_A \til{M}$ and $B \ot_A \til{N}$ are K-flat over $B$, 
$\til{M} \ot_A \til{N}$ is K-flat over $A$, and 
\[ \begin{aligned}
& (B \ot^{\mrm{L}}_A M) \ot^{\mrm{L}}_B (B \ot^{\mrm{L}}_A N) \cong 
(B \ot_A \til{M}) \ot_B (B \ot_A \til{N}) 
\\
& \quad 
\cong B \ot_A (\til{M} \ot_A \til{N}) \cong
B \ot^{\mrm{L}}_A (M \ot^{\mrm{L}}_A N) .
\end{aligned} \]

\medskip \noindent
(2) Let $P, Q \in \cat{D}(A)$ be such that 
$P \ot^{\mrm{L}}_A Q \cong A$. By (1)  we have 
\[ (B \ot^{\mrm{L}}_A P) \ot^{\mrm{L}}_B (B \ot^{\mrm{L}}_A Q) \cong  B . \]

\medskip \noindent
(3) Say $Q_1, Q_2 \in \cat{D}(B)$ satisfy
$Q_1 \ot^{\mrm{L}}_B Q_2 \cong B$. Let 
$P_i := \opn{rest}_f(Q_i) \in \cat{D}(A)$. 
By the equivalence for DG ring  
quasi-isomorphisms (see \cite[Proposition 1.4]{YZ1}) we have 
\[ P_1 \ot^{\mrm{L}}_A P_2 \cong 
\opn{rest}_f( Q_1 \ot^{\mrm{L}}_B Q_2 ) \cong 
\opn{rest}_f(B) \cong A . \]
\end{proof}

\begin{prop} \label{prop:32}
Let $f : A \to B$ be a homomorphism of DG rings.
\begin{enumerate}
\item There is a group homomorphism
\[ \opn{DPic}(f) : \opn{DPic}(A) \to \opn{DPic}(B)  \]
with formula
$P \mapsto B \ot^{\mrm{L}}_A P$.

\item If $f$ is a quasi-isomorphism then $\opn{DPic}(f)$ is bijective.
\end{enumerate}
\end{prop}

\begin{proof}
(1) This follows from parts (1) and (2) of the lemma above. 

\medskip \noindent
(2) Part (3) of Lemma \ref{lem:80} shows that in case $f$ is a 
quasi-isomorphism, the function 
$Q \mapsto \opn{rest}_f(Q)$ is an inverse of  $\opn{DPic}(f)$.
\end{proof}

\begin{thm} \label{thm:76}
Let $A$ be a DG ring and $P \in \cat{D}(A)$. The following 
four conditions are equivalent.
\begin{enumerate}
\rmitem{i} The DG $A$-module $P$ is tilting.

\rmitem{ii} The functor $P \ot^{\mrm{L}}_A -$ is an equivalence of 
$\cat{D}(A)$. 

\rmitem{iii}  The functor $\opn{RHom}_{A}(P, -)$ is an equivalence of 
$\cat{D}(A)$. 

\rmitem{iv}  The DG $A$-module $P$ is perfect, and the adjunction morphism 
$A \to \opn{RHom}_A(P, P)$ in $\cat{D}(A)$ is an isomorphism.
\end{enumerate} 
\end{thm}

\begin{proof}
(i) $\Rightarrow$ (ii): Let $Q$ be a quasi-inverse of $P$. Then the functor  
$G(M) := Q \ot^{\mrm{L}}_A M$ is a quasi-inverse of the functor 
$F(M) := P \ot^{\mrm{L}}_A M$.

\medskip \noindent
(ii) $\Rightarrow$ (i): The functor $F := P \ot^{\mrm{L}}_A -$ is 
essentially surjective on objects, so there is some $Q \in \cat{D}(A)$
such that $F(Q) \cong A$. Then $Q$ is a quasi-inverse of $P$. 

\medskip \noindent
(ii) $\Leftrightarrow$ (iii): The functors 
$F := P \ot^{\mrm{L}}_A -$ and 
$G := \opn{RHom}_{A}(P, -)$
are adjoints, so $F$ is an equivalence iff $G$ is an equivalence.

\medskip \noindent 
(ii) $\Rightarrow$ (iv): 
Consider the auto-equivalence $F := P \ot^{\mrm{L}}_A -$ of $\cat{D}(A)$.
Since $A$ is compact and $P = F(A)$, it follows that $P$ is compact. 
Now according to Theorem \ref{thm:74}, perfect is the same as compact. 

For any $M \in \cat{D}(A)$, the adjunction morphism 
$A \to \opn{RHom}_A(M, M)$ is an isomorphism iff the canonical 
graded ring homomorphism 
\[ \al_M : \opn{H}(A) \to 
\bigoplus_{k \in \Z} \, \opn{Hom}_{\cat{D}(A)}(M, M[k]) \]
is bijective. The equivalence $F$ induces a commutative diagram of graded rings 
\[ \UseTips \xymatrix @C=5ex @R=5ex {
\opn{H}(A)
\ar[d]_{\al_M}
\ar[dr]^{\al_{F(M)}}
\\
\displaystyle\bigoplus_{k \in \Z} \, \opn{Hom}_{\cat{D}(A)}(M, M[k]) 
\ar[r]^(0.44){F}
&
\displaystyle\bigoplus_{k \in \Z} \, 
\opn{Hom}_{\cat{D}(A)} \bigl( F(M), F(M)[k] \bigr)
} \]
in which the horizontal arrow is an isomorphism. 
Now take $M := A$. Because $\al_A$ is an isomorphism, so is $\al_P$.

\medskip \noindent
(iv) $\Rightarrow$ (i): 
Define $Q := \opn{RHom}_A(P, A) \in \cat{D}(A)$.
The implication (i) $\Rightarrow$ (iii) of Theorem \ref{thm:74} 
shows that 
\[ Q \ot^{\mrm{L}}_A P =  \opn{RHom}_A(P, A) \ot^{\mrm{L}}_A P
\cong \opn{RHom}_A(P, P) \cong A \]
in $\cat{D}(A)$. So $P$ is tilting, with quasi-inverse $Q$. 
\end{proof}

\begin{cor} \label{cor:405}
Let $P$ be a tilting DG $A$-module. Then the DG $A$-module 
$Q := \lb \opn{RHom}_A(P, A)$ is a quasi-inverse of $P$. 
\end{cor}

\begin{proof}
This was shown in the proof of the implication 
(iv) $\Rightarrow$ (i) above.
\end{proof}

\begin{rem} \label{rem:510}
A DG $A$-module $M$ for which the adjunction 
morphism $A \to \lb \opn{RHom}_A(M, M)$ is an isomorphism, is sometimes called 
{\em semidualizing}; cf.\ \cite{AIL}. Indeed, this condition is part of the 
definition of a {\em dualizing DG module} (see Definition \ref{dfn:253} below). 
But as we saw in Theorem \ref{thm:76}, this condition is also characteristic of 
tilting DG modules -- so the name semidualizing might be confusing. 
\end{rem}

\begin{cor} \label{cor:73}
Let $P$ be a tilting DG $A$-module, and let 
$F := P \ot^{\mrm{L}}_A -$
be the corresponding auto-equivalence of $\cat{D}(A)$.
\begin{enumerate}
\item The functor $F$ has finite cohomological dimension, and it preserves 
$\cat{D}^+(A)$, $\cat{D}^-(A)$ and 
$\cat{D}^{\mrm{b}}(A)$.

\item If $A$ is cohomologically pseudo-noetherian, then the auto-equivalence 
$F$ preserves the subcategory $\cat{D}_{\mrm{f}}(A)$. 
\end{enumerate}
\end{cor}

\begin{proof}
(1) The theorem says that $P$ is perfect. 
Let $s_1, \ldots, s_n \in \bar{A}$ and $P_1, \ldots, P_n$ be as in Definition 
\ref{dfn:25}.
Let $d_0 \leq d_1$ be integers such that each $P_i$ is generated in 
the integer interval $[d_0, d_1]$. Then the functor $F$
has cohomological displacement at most $[d_0, d_1]$ relative to 
$\cat{D}(A)$, and cohomological dimension at most $d_1 - d_0$. The claim about 
$\cat{D}^{\star}(A)$ is now clear. 

\medskip \noindent
(2) Use Theorem \ref{thm:30}(2), noting that 
$F(A) = P \in \cat{D}^-_{\mrm{f}}(A)$,
by Theorems \ref{thm:76} and \ref{thm:50}.
\end{proof}

\begin{prop} \label{prop:313}
Let $A$ be a DG ring, and let 
$\bsym{e} = (e_1, \ldots, e_n)$ be an idempotent covering sequence of 
$\bar{A} = \opn{H}^0(A)$. For 
any $i$ we have the localized DG ring $A_i = A_{e_i}$ from Definition 
\tup{\ref{dfn:40}}, and the DG ring homomorphism
$\la_i : A \to A_i$.
Then the group homomorphisms 
\[ \opn{DPic}(\la_i) : \opn{DPic}(A) \to \opn{DPic}(A_i) \]
induce a group isomorphism
\[ \opn{DPic}(A) \iso \prod_{i = 1}^n \, \opn{DPic}(A_i) . \]
\end{prop}

\begin{proof}
(1) According to Proposition \ref{prop:20} there is a DG ring quasi-isomorphism 
$\la : A \to \prod_{i = 1}^n A_i$. By Proposition \ref{prop:32}(2) there is a 
group isomorphism 
\[ \opn{DPic}(\la) : \opn{DPic}(A) \iso 
\opn{DPic} \Bigl( \, \prod\nolimits_{i = 1}^n A_i \, \Bigr) . \]
And there is an obvious group isomorphism 
\[ \opn{DPic} \Bigl( \, \prod\nolimits_{i = 1}^n A_i \, \Bigr) \cong 
\prod_{i = 1}^n \, \opn{DPic}(A_i) . \]
\end{proof}

From here to Theorem \ref{thm:315} we consider a ring $A$. The first lemma 
goes back to \cite[Lemma V.3.3]{RD}, and it reappeared, in a noncommutative 
guise, in \cite{Ye2} and \cite{RZ}. But in the prior treatments the ring $A$ 
was assumed to be noetherian. Therefore we give a full proof of the general 
case. 

\begin{lem} \label{lem:575}
Let $A$ be a local ring, and let $P$ be a tilting DG $A$-module. 
Then $P \cong A[k]$ in $\cat{D}(A)$ for some integer $k$. 
\end{lem}

\begin{proof}
Since $P$ is a perfect complex of $A$-modules (Theorem \ref{thm:76}), it is 
iso\-morphic in $\cat{D}(A)$ to a bounded complex of finite projective 
$A$-modules. Say  $\opn{con}(\opn{H}(P)) = [i_0, i_1]$ for some integers $i_0 
\leq i_1$. By replacing $P$ with such a resolution, and then splitting off 
extra 
terms in high and low degrees, we can assume that $P$ is a complex of finite 
projective $A$-modules, concentrated in the degree interval $[i_0, i_1]$. Then 
$P' := \opn{H}^{i_1}(P)$ is a nonzero finitely presented $A$-module.  

Let $Q$ be a quasi-inverse of $P$. By the same reasoning, we can assume that 
$Q$ is a complex of finite projective $A$-modules, concentrated in a degree 
interval $[j_0, j_1]$, and $Q' := \opn{H}^{j_1}(Q)$ is a nonzero finitely 
presented $A$-module.

Now by the K\"unneth trick we have isomorphisms of $A$-modules
\[ P' \ot_A Q' \cong \opn{H}^{i_1 + j_1}(P \ot^{\mrm{L}}_A Q) \cong 
\opn{H}^{i_1 + j_1}(A) . \]  
On the other hand $P' \ot_A Q' \neq 0$, by Nakayama. The conclusion is that 
$i_1 + j_1 = 0$ and $P' \ot_A Q' \cong A$. 
This tells us that $P'$ and $Q'$ are flat $A$-modules. 
Again by Nakayama, we see that $P' \cong Q' \cong A$. 

Finally, we can split the complexes $P$ and $Q$ into 
$P \cong P'[-i_1] \oplus P^{\mrm{ex}}$
and 
$Q \cong Q'[-j_1] \oplus Q^{\mrm{ex}}$,
where $P^{\mrm{ex}}$ and $Q^{\mrm{ex}}$ are complexes of finite projective 
$A$-modules, concentrated in the degree intervals $[i_0, i_1 - 1]$ and 
$[j_0, j_1 - 1]$ respectively. 
Then $\opn{H}(P^{\mrm{ex}} \ot_A Q'[-j_1])$
is a direct summand of the graded module
$\opn{H}(P \ot^{\mrm{L}}_A Q) \cong \opn{H}(A)$. 
Since $\opn{H}(P^{\mrm{ex}} \ot_A Q'[-j_1])$ is concentrated in the interval
$[i_0 + j_1, -1]$,
it must be zero. So $P^{\mrm{ex}} = 0$ in $\cat{D}(A)$, and
$P \cong P'[-i_1]$. 
\end{proof}

Recall that an $A$-module $P$ is called {\em invertible} if it is projective of 
rank $1$. In other words, if $P$ is locally free of rank $1$.  
See \cite[Section II.5.2, Theorem 1]{Bo}. 


\begin{prop} \label{prop:590}
The following conditions are equivalent for an $A$-module $P$~\tup{:}
\begin{enumerate}
\rmitem{i} $P$ is invertible.
\rmitem{ii} When viewed as a DG module, $P$ is tilting. 
\end{enumerate}
\end{prop}

\begin{proof}
The implication (i) $\Rightarrow$ (ii) is trivial. For the other direction: 
the arguments in the first paragraph of the proof of Lemma 
\ref{lem:575} show that $P$ is a finitely presented $A$-module. 
Take any prime ideal $\p$. By Lemma \ref{lem:80}, the DG $A_{\p}$-module 
$P_{\p} := A_{\p} \ot_A P$ is tilting. Hence, by Lemma \ref{lem:575},
we have $P_{\p} \cong A_{\p}$.  
According to \cite[Section II.5.2, Theorem 1]{Bo} the module $P$ is 
invertible.
\end{proof}

\begin{lem} \label{lem:570}
Let $A$ be a ring, and let $P$ be a tilting DG $A$-module. 
Write $X := \opn{Spec} A$, and for any integer $i$ define the set 
\[ Y_i := \{ \p \in X \mid \opn{H}^i(P)_{\p} \neq 0 \} . \]
Then:
\begin{enumerate}
\item Only finitely many of the $Y_i$ are nonzero.
\item $Y_i \cap Y_j = \varnothing$ if $i \neq j$.
\item $X = \bigcup\nolimits_{i} \, Y_i$.
\item For any $i$, the set $Y_i$ is open-closed.
\end{enumerate}
\end{lem}

\begin{proof}
We may assume that  $A \neq 0$, i.e.\ that $X \neq \varnothing$.
This implies that $P \neq 0$. 
We know from Theorem \ref{thm:76} that $P$ is a perfect complex of 
$A$-modules. Therefore the cohomology $\opn{H}(P)$ is bounded, and thus
$\opn{con}(\opn{H}(P)) = [i_0, i_1]$
for some integers $i_0 \leq i_1$. This proves claim (1). 

Take any prime $\p \in X$. By Lemma \ref{lem:80} the DG $A_{\p}$-module 
$P_{\p} := A_{\p} \ot_A P$ is tilting. So by Lemma \ref{lem:575} we have 
$P_{\p} 
\cong A_{\p}[k]$ for some integer $k$. We see that $\p \in Y_{-k}$, and 
$\p \notin Y_{i}$ for $i \neq -k$. This proves claims (2) and (3). 

It remains to prove claim (4). By induction on 
$i_1 - i_0 = \opn{amp}(\opn{H}(P))$, it suffices to prove that $Y' := Y_{i_1}$
is open-closed in $X$. The reason is this: once we know that $Y'$ is 
open-closed, then so is its complement $Y'' := \bigcup_{i_0 \leq i < i_1} Y_i$. 
So $X = Y' \coprod Y''$ as schemes, and correspondingly 
$A = A' \times A''$ as rings, and $P \cong P' \oplus P''$ in $\cat{D}(A)$. 
But $\opn{amp}(\opn{H}(P'')) < i_1 - i_0$. 

As explained in the proof of Lemma \ref{lem:575}, we can 
assume that $P$ is a complex of finite projective $A$-modules, concentrated in 
the degree interval $[i_0, i_1]$. So $Q := \opn{H}^{i_1}(P)$ is a finitely 
presented $A$-module. Because $Y'$ is the support of the module 
$Q$, it is closed in $X$. 

Take a prime $\p \in Y'$. As we have already seen above, the 
$A_{\p}$-module $Q_{\p}$ is isomorphic to $A_{\p}$. According to \cite[Section 
II.5.1, Corollary]{Bo} there is an open neighborhood $U$ of $\p$ in $X$ such 
that $Q_{\q} \cong A_{\q}$ for all $\q \in U$. This shows that $U \subset Y'$. 
Therefore $Y'$ is open in $X$. 
\end{proof}

For a ring $A$ we have the following theorem, due to Negron 
\cite{Ng}. It is a refinement of earlier results, that are due (independently) 
to the author \cite{Ye2} and to Rouquier-Zimmermann \cite{RZ}. 
The earlier results focused on noetherian rings, and hence there was an 
assumption that $\opn{Spec} A$ has finitely many connected components. 
Negron recently noticed that this assumption is superfluous. 

As usual, $\opn{Pic}(A)$ denotes the (commutative) Picard group of $A$, whose 
elements are the isomorphism classes of invertible $A$-modules. 

The abelian group $\opn{F}_{\mrm{lc}}(\opn{Spec} A, \Z)$
was introduced in Definition \ref{dfn:560}.
As shown in Proposition \ref{prop:550}, each
function $f \in \opn{F}_{\mrm{lc}}(\opn{Spec} A, \Z)$ 
determines a decomposition 
$A = \prod_{i = 1}^n \, A_i$ of the ring $A$, 
and a nondecreasing sequence of integers $\bsym{k} = (k_1, \ldots, k_n)$. 
The relation is this: $f(\opn{Spec} A_i) = k_i$. 

\begin{thm} \label{thm:315}
Let $A$ be a commutative ring. 
There is a canonical group isomorphism
\[ \opn{DPic}(A) \cong 
\opn{Pic}(A) \times \opn{F}_{\mrm{lc}}(\opn{Spec} A, \Z) \, , \]
characterized as follows\tup{:}
\begin{itemize}
\item The homomorphism 
$\opn{Pic}(A) \to \opn{DPic}(A)$
sends the class of an invertible $A$-module $P$ to its class as a tilting DG 
$A$-module. 

\item Let $f \in \opn{F}_{\mrm{lc}}(\opn{Spec} A, \Z)$,
with corresponding ring decomposition
$A = \prod_{i = 1}^n \, A_i$
and integer sequence $(k_1, \ldots, k_n)$.
The tilting DG $A$-module associated to $f$ is 
$\bigoplus_{i = 1}^n  A_i[k_i]$.
\end{itemize}
\end{thm}

\begin{proof}
For an invertible $A$-module $Q$ and a function 
$f \in \opn{F}_{\mrm{lc}}(\opn{Spec} A, \Z)$
let 
\[ G(Q, f) := Q \ot_A \bigl( \bigoplus\nolimits_{i = 1}^n  A_i[k_i] \bigr) , \]
which is a tilting DG $A$-module. We have to prove that the group homomorphism 
\[ \opn{Pic}(A) \times \opn{F}_{\mrm{lc}}(\opn{Spec} A, \Z) \to 
\opn{DPic}(A) \]
induced by $G$ is bijective. 
It is certainly injective: if 
$G(Q, f) \cong A$, then $f$ must be the constant function $0$, as can be 
checked in $\opn{DPic} (A_{\p})$ for each $\p \in \opn{Spec} A$. 
But then $Q \cong A$ in $\cat{D}(A)$, which implies that 
$Q \cong A$ in $\cat{Mod} A$. 

It remains to prove that the group homomorphism induced by $G$ is surjective.
Take any tilting DG $A$-module $P$. Let 
$\opn{Spec} A = \coprod_{i = i_0}^{i_1} \, Y_i$ 
be the decomposition into open-closed sets induced by $P$, from Lemma 
\ref{lem:570}. Consider the locally constant function $f : X \to \Z$ 
defined by $f|_{Y_i} := i$. 
Then the tilting DG module 
$Q := P \ot^{\mrm{L}}_A G(A, f)$ 
has the property that 
$\opn{H}^i(Q) = 0$ for all $i \neq 0$.
By Proposition \ref{prop:590} we know that $Q$ is isomorphic to an invertible 
$A$-module, say $Q'$. But then $P \cong G(Q', -f)$. 
\end{proof}

Let $A$ be a noetherian ring, $\a$-adically complete with respect to some ideal 
$\a$, with reduction $\bar{A} := A / \a$. It is known that the group 
homomorphism $\opn{Pic}(A) \to \opn{Pic}(\bar{A})$ is bijective. For a proof 
see \cite[Exercises II.9.6 and III.4.6]{Ha}. 

The next theorem is a DG analogue of this fact. 
Recall that for a DG ring $A$ there is a canonical homomorphism 
$A \to \bar{A}$.

\begin{thm} \label{thm:40}
Let $A$ be a commutative DG ring.
Then the canonical group homomorphism 
\[ \opn{DPic}(A) \to \opn{DPic}(\bar{A}) \]
is bijective.
\end{thm}
 
\begin{proof}
Let us denote by $\pi : A \to \bar{A}$ the canonical DG ring homomorphism.
We begin by proving that the homomorphism $\opn{DPic}(\pi)$ is injective.
Suppose $P$ is a tilting DG $A$-module such that 
$\bar{A} \ot^{\mrm{L}}_A P \cong \bar{A}$ in $\cat{D}(\bar{A})$. 
By Corollary \ref{cor:73}(1) we know that $P \in \cat{D}^{-}(A)$.
Then Proposition \ref{prop:305}(1) says that $P \cong A$ in $\cat{D}(A)$. 

Now let us prove that  $\opn{DPic}(\pi)$ is surjective.
Take any tilting DG $\bar{A}$-module $\bar{P}$. By Theorem \ref{thm:315} there 
is an isomorphism 
$\bar{P} \cong G(\bar{P}_0, f)$ for some invertible $\bar{A}$-module 
$\bar{P}_0$ and some function $f \in \opn{F}_{\mrm{lc}}(\opn{Spec} A, \Z)$.

By Proposition \ref{prop:310}, there exists some DG module 
$P_0 \in \cat{D}^{-}(A)$ such that  \lb
$\bar{A} \ot^{\mrm{L}}_A P_0 \cong \bar{P}_0$ in $\cat{D}(\bar{A})$.
Let $\bar{Q}_0 \in \cat{Mod} \bar{A}$ be a quasi-inverse of the invertible 
module $\bar{P}_0$. By the same reason, there exists
$Q_0 \in \cat{D}^{-}(A)$ such that  
$\bar{A} \ot^{\mrm{L}}_A Q_0 \cong \bar{Q}_0$ in $\cat{D}(\bar{A})$.
Now by Lemma \ref{lem:80}(1) we have
\[ \bar{A} \ot^{\mrm{L}}_A (P_0 \ot^{\mrm{L}}_A Q_0) \cong 
\bar{P}_0 \ot^{\mrm{L}}_{\bar{A}} \bar{Q}_0 \cong \bar{A} \]
in $\cat{D}(\bar{A})$. Since 
$P_0 \ot^{\mrm{L}}_A Q_0 \in \cat{D}^{-}(A)$,
by Proposition \ref{prop:305}(1) we know that 
$P_0 \ot^{\mrm{L}}_A Q_0 \cong A$ in $\cat{D}(A)$. This shows that $P_0$ is a 
tilting DG $A$-module. 

Finally, let $\bsym{e} = (e_1, \ldots, e_n)$ and 
$\bsym{k} = (k_1, \ldots, k_n)$ be the data corresponding to $f$ from 
Proposition \ref{prop:550}, and let 
$A \to \prod_{i = 1}^n A_i$ be the $\bsym{e}$-induced decomposition of 
$A$ from Definition \ref{dfn:343}.
Consider the tilting DG $A$-module 
\[ P := P_0 \ot_A \bigl( \bigoplus\nolimits_{i = 1}^n  A_i[k_i] \bigr) . \]
Then 
$\bar{P} \cong \bar{A} \ot^{\mrm{L}}_A P$. 
\end{proof}

\begin{cor} \label{cor:330}
Let $A$ be a DG ring. There is a canonical group isomorphism 
\[ \opn{DPic}(A) \cong  \opn{Pic}(\bar{A}) \times 
\opn{F}_{\mrm{lc}}(\opn{Spec} \bar{A}, \Z) \, . \]
\end{cor}

\begin{proof}
Combine Theorems \ref{thm:40} and \ref{thm:315}.
\end{proof}

\begin{dfn} \label{dfn:330}
Let $A$ be a DG ring. We define $\opn{DPic}^0(A)$ to be the subgroup of 
$\opn{DPic}(A)$ that corresponds to $\opn{Pic}(\bar{A})$, under the 
canonical group isomorphism of Corollary \ref{cor:330}.
\end{dfn}

\begin{cor} \label{cor:315}
If $\bar{A}$ is a local ring, then $\opn{DPic}(A) \cong \Z$.
\end{cor}

\begin{proof}
We know that $\opn{Pic}(\bar{A})$ is trivial, and 
$\opn{Spec} \bar{A}$ is connected. Now use Corollary \ref{cor:330}.
\end{proof}

\section{Dualizing DG Modules} \label{sec:dualizing}

Recall that all our DG rings are now commutative (Convention \ref{conv:100}),
and $\bar{A} = \opn{H}^0(A)$.
In this section we concentrate on cohomologically pseudo-noetherian  
DG rings (Definition \ref{dfn:391}). 

Here is a generalization of Grothendieck's definition of dualizing 
complex. \lb When $A$ is a ring, this is identical to the definition in
\cite[Section V.2]{RD}.

\begin{dfn} \label{dfn:253}
Let $A$ be a cohomologically pseudo-noetherian DG ring.
A DG $A$-module $R$ is called a {\em dualizing} if it satisfies these three 
conditions:
\begin{enumerate}
\rmitem{i} Each $\opn{H}^i(R)$ is a finite $\bar{A}$-module. 

\rmitem{ii} $R$ has finite injective dimension relative to 
$\cat{D}(A)$. 

\rmitem{iii} The adjunction morphism 
$A \to \opn{RHom}_A(R, R)$ in $\cat{D}(A)$
is an isomorphism. 
\end{enumerate}
\end{dfn}

In other words, condition (i) says that $R \in \cat{D}_{\mrm{f}}(A)$;
condition (ii) says that the functor 
$\opn{RHom}_A(-, R)$ has finite cohomological dimension relative to
$\cat{D}(A)$, as in Definitions \ref{dfn:10} and \ref{dfn:11}(2); and condition 
(iii) says that the canonical graded ring homomorphism 
\[ \opn{H}(A) \to 
\bigoplus_{k \in \Z} \, \opn{Hom}_{\cat{D}(A)}(R, R[k]) \]
is bijective. 

\begin{prop} \label{prop:250}
Suppose $R$ is a dualizing DG module over $A$, and consider the functor 
\[ D :=  \opn{RHom}_A(-, R) : \cat{D}(A) \to \cat{D}(A) . \]
Let $\cat{D}^{\star}$ denote either $\cat{D}$, 
$\cat{D}^{\mrm{b}}$, $\cat{D}^{+}$ or $\cat{D}^{-}$; and 
correspondingly let  $\cat{D}^{- \star}$ denote either $\cat{D}$, 
$\cat{D}^{\mrm{b}}$, $\cat{D}^{-}$ or $\cat{D}^{+}$.
\begin{enumerate}
\item For any $M \in \cat{D}^{\star}_{\mrm{f}}(A)$
we have $D(M) \in \cat{D}^{- \star}_{\mrm{f}}(A)$.
In particular, $R \in \cat{D}^{+}_{\mrm{f}}(A)$.

\item For any $M \in \cat{D}_{\mrm{f}}(A)$ the canonical morphism 
$M \to D(D(M))$ in $\cat{D}_{\mrm{f}}(A)$ is an isomorphism. 

\item The functor 
\[ D : \cat{D}^{\star}_{\mrm{f}}(A)^{\mrm{op}} \to 
\cat{D}^{- \star}_{\mrm{f}}(A) \]
is an equivalence of triangulated categories. 
\end{enumerate}
\end{prop}

\begin{proof}
(1) The functor $D$ has finite cohomological dimension,
so we can apply Theorem \ref{thm:30}(2).
Since $A \in \cat{D}^{-}_{\mrm{f}}(A)$, 
we get $R = D(A) \in \cat{D}^{+}_{\mrm{f}}(A)$. 

\medskip \noindent
(2) There is a morphism of triangulated functors 
$\eta : \opn{id}_{\cat{D}_{\mrm{f}}(A)} \to D \circ D$.
Both functors have finite cohomological dimensions, 
and $\eta_A$ is an isomorphism.  We can apply Theorem \ref{thm:10}(2). 

\medskip \noindent
(3) Combine items (1) and (2).  
\end{proof}

\begin{cor} \label{cor:380} 
Let $A$ be a cohomologically pseudo-noetherian DG ring, and let $R$ be a 
dualizing DG $A$-module. The following two conditions are equivalent\tup{:}
\begin{itemize}
\rmitem{i} The DG ring $A$ is cohomologically bounded.
\rmitem{ii} The DG module $R$ is cohomologically bounded.
\end{itemize}
\end{cor}

\begin{proof}
This is by Proposition \ref{prop:250}, since $R \cong D(A)$ and $A \cong D(R)$.
\end{proof}

In Example \ref{exa:400} we demonstrate the unbounded option. 

Given a homomorphism $f : A \to B$  of DG rings, 
we denote by $\bar{f} := \opn{H}^0(f)$ the induced ring homomorphism.

\begin{dfn} \label{dfn:252}
Let $f : A \to B$ be a homomorphism between cohomologically pseudo-noetherian 
DG rings. We say that $f$ is a {\em cohomologically pseudo-finite 
homomorphism} if $\bar{f} : \bar{A} \to \bar{B}$ is a finite ring homomorphism, 
i.e.\ $\bar{f}$ makes $\bar{B}$ into a finite $\bar{A}$-module. 
\end{dfn}

Clearly if $f$ is cohomologically finite, then $\opn{rest}_f$ sends 
$\cat{D}^{\star}_{\mrm{f}}(B)$ into $\cat{D}^{\star}_{\mrm{f}}(A)$,
where  $\star$ is either $+, -, \mrm{b}$ or blank.

\begin{prop} \label{prop:252}
Let $f : A \to B$ be a cohomologically pseudo-finite homomorphism between 
cohomologically pseudo-noetherian DG rings.
\begin{enumerate}
\item If $R_A$ is a dualizing DG $A$-module, then 
$R_B := \opn{RHom}_A(B, R_A)$ is a dualizing DG $B$-module.

\item If $f$ is a quasi-isomorphism and $R_B$ is a dualizing DG $B$-module, 
then $R_A := \lb \opn{rest}_f(R_B)$ is a dualizing DG $A$-module.
\end{enumerate}
\end{prop}

\begin{proof}
(1) The proof is almost the same as in the case of a ring; see 
\cite[Proposition V.2.4]{RD}. Viewing $B$ as an object of 
$\cat{D}^{-}_{\mrm{f}}(A)$, Proposition \ref{prop:250}(1) tells us that 
$R_B \in \cat{D}^{+}_{\mrm{f}}(A)$; and hence 
$R_B \in \cat{D}^{+}_{\mrm{f}}(B)$.
For any $N \in \cat{D}(B)$ we have 
\[ \opn{RHom}_B(N, R_B) \cong \opn{RHom}_A(N, R_A) \]
by adjunction, and hence the injective dimension of $R_B$ relative to 
$\cat{D}(B)$ is at most the injective dimension of A relative to 
$\cat{D}(A)$, which is finite. And finally 
\[ \opn{RHom}_B(R_B, R_B) \cong 
\opn{RHom}_A \bigl( \opn{RHom}_A(B, R_A), R_A \bigr) \cong B \]
by Proposition \ref{prop:250}(2). These isomorphisms are actually inside 
$\cat{D}(B)$, and they are compatible with the 
canonical morphism from $B$. 

\medskip \noindent
(2) This is because here
$\opn{rest}_f : \cat{D}(B) \to \cat{D}(A)$
is an equivalence, preserving boundedness and finiteness of cohomology. 
\end{proof}

In commutative ring theory, many good properties of a ring $A$ can be deduced 
if it is ``not far'' from a ``nice'' ring $\K$ (such as a field). This is the 
underlying reason for the next definition. 

Recall that a ring homomorphism $A \to B$ is called of {\em essentially finite 
type} if it can be factored as $A \to B_{\mrm{ft}} \to B$, where $A \to 
B_{\mrm{ft}}$ is finite type (i.e.\ $B$ is finitely generated as $A$-ring), and 
$B_{\mrm{ft}} \to B$ is the localization at some multiplicatively closed subset 
of $B_{\mrm{ft}}$. 

\begin{dfn} \label{dfn:510}
Let $\K$ be a noetherian ring, let $A$ be a cohomologically pseudo-noetherian 
DG ring, and let $u : \K \to A$ a homomorphism of DG 
rings. We say that $u$ is of {\em cohomologically essentially finite type}, and 
that $A$ is a {\em cohomologically essentially finite type DG $\K$-ring}, if 
the ring homomorphism $\bar{u} : \K \to \bar{A}$ is of essentially finite type.
\end{dfn}

\begin{dfn} \label{dfn:46}
A DG ring $A$ is called {\em tractable} if it is cohomologically 
pseudo-noetherian, and there exists a cohomologically 
essentially finite type homomorphism $\K \to A$ from some finite dimensional 
regular noetherian ring $\K$. Such a homomorphism $\K \to A$ is called a {\em 
traction} for $A$.
\end{dfn}

\begin{lem} \label{lem:380}
Let $A$ be a DG ring, let $\K$ be a noetherian ring, and suppose there is a 
cohomologically essentially finite type homomorphism $u : \K \to A$. Then there 
is a commutative diagram of DG rings 
\[  \UseTips \xymatrix @C=5ex @R=5ex {
\K
\ar[r]^{u}
\ar[d]_{v}
&
A 
\ar[r]^{\pi}
\ar[d]_{f}
& 
\bar{A}
\\
A_{\mrm{eft}}
\ar[r]^{g}
&
A_{\mrm{loc}}
\ar[ur]
} \]
such that $\pi$ is the canonical homomorphism\tup{;} 
$f$ and $g$ are quasi-isomorphism\tup{;} and
$\K \to A^0_{\mrm{eft}}$ is essentially finite type.
\end{lem}

\begin{proof}
Let $S := \pi^{-1}(\bar{A}^{\times}) \cap A^0$, namely $s \in S$ iff 
$\pi(s)$ is invertible in the ring $\bar{A}$.  
Define the DG ring 
$A_{\mrm{loc}} := (S^{-1} \cd A^0) \ot_{A^0} A$.
Then $\pi$ factors via $f$, and $f$ is a quasi-isomorphism. 

Since $\K \to \bar{A}$ is essentially finite type, there is a polynomial ring 
$\K[\bsym{t}]$ in finitely many variables of degree $0$, and a 
homomorphism $h : \K[\bsym{t}] \to \bar{A}$ which is essentially surjective,
i.e.\ it is surjective after a localization.
Thus, letting $T := h^{-1}(\bar{A}^{\times}) \lb \subseteq \K[\bsym{t}]$,
and defining 
$A_{\mrm{eft}}^0 := T^{-1} \cd \K[\bsym{t}]$,
the homomorphism  
$h_T : A_{\mrm{eft}}^0 \to \bar{A}$ 
is surjective.

Now the ring  $A_{\mrm{eft}}^0$ is noetherian.
The homomorphism $h_T : A_{\mrm{eft}}^0 \to \bar{A}$ factors via a homomorphism 
$g^0 : A_{\mrm{eft}}^0 \to A^0_{\mrm{loc}}$, 
and the composed homomorphism 
$A_{\mrm{eft}}^0 \to  \opn{H}^0(A_{\mrm{loc}})$ is surjective. 
Since the the modules $\opn{H}^i(A_{\mrm{loc}})$ are finite over 
$A_{\mrm{eft}}^0$, we can extend $A_{\mrm{eft}}^0$ to a DG ring 
$A_{\mrm{eft}}$, and simultaneously extend $g^0$
to a quasi-isomorphism 
$g : A_{\mrm{eft}} \to A_{\mrm{loc}}$,
by inductively introducing finitely many new variables (free ring generators) 
in negative degrees.
The process is the same as in the proof of 
\cite[Proposition 1.7(2)]{YZ1}. 
\end{proof}

\begin{thm} \label{thm:51}
Let $A$ be a tractable  DG ring.  Then $A$ has a dualizing DG module. 
\end{thm}

\begin{proof}
Let $\K \to A$ be a traction for $A$. 
Consider the diagram of homomorphisms in Lemma \ref{lem:380}.
Since $A \to A_{\mrm{loc}}$ is a quasi-isomorphism, and 
$A_{\mrm{eft}}^0 \to A_{\mrm{eft}} \to A_{\mrm{loc}}$
are cohomologically pseudo-finite, it suffices (by Proposition \ref{prop:252})
to show that $A_{\mrm{eft}}^0$ has a dualizing DG module 
(which is the same as a dualizing complex over this ring, in the sense of 
\cite{RD}). But the ring homomorphism $\K \to A_{\mrm{eft}}^0$ can be factored 
into 
$\K \to \K[\bsym{t}] \to B \to A_{\mrm{eft}}^0$,
where $\K[\bsym{t}]$ is a polynomial ring in $n$ variables, 
$\K[\bsym{t}] \to B$ is surjective, and $B \to A_{\mrm{eft}}^0$
is a localization. Thus, using \cite[Theorem V.8.3]{RD} and Proposition 
\ref{prop:252}(1) above, the DG module 
\[ A_{\mrm{eft}}^0 \ot_B \opn{RHom}_{\K[\bsym{t}]}
(B, \Om^n_{\K[\bsym{t}] / \K}[n]) \]
is a dualizing DG module over $A_{\mrm{eft}}^0$.
\end{proof}

\begin{thm} \label{thm:43}
Let $A$ be a cohomologically pseudo-noetherian DG ring, and let
$R$ be a dualizing DG module over $A$. 
\begin{enumerate}
\item If $P$ is a tilting DG module, then
$P \ot^{\mrm{L}}_A R$ is a dualizing DG module. 

\item If $R'$ is a dualizing DG module, then 
$P := \opn{RHom}_A(R, R')$ is a tilting DG module, and 
$R' \cong P \ot^{\mrm{L}}_A R$ in $\cat{D}(A)$.

\item If $P$ is a tilting DG module, and if $R \cong P \ot^{\mrm{L}}_A R$
in $\cat{D}(A)$, then $P \cong A$ in $\cat{D}(A)$.
\end{enumerate}
\end{thm}

This is similar to \cite[Theorem V.3.1]{RD}, and the strategy of the proof 
is the same; cf.\ also \cite[Theorem 4.5]{Ye2}. 

\begin{proof}
(1) Assume $P$ is a tilting DG module, and let $R' := P \ot^{\mrm{L}}_A R$. 
According to Corollary \ref{cor:73} the functor $P \ot^{\mrm{L}}_A -$
is an auto-equivalence of $\cat{D}(A)$, it has finite cohomological dimension,
and it preserves $\cat{D}^{+}_{\mrm{f}}(A)$.
Therefore the DG module $R'$ is dualizing. 

\medskip \noindent (2)
Define the objects 
$P := \opn{RHom}_A(R, R')$ and 
$P' := \opn{RHom}_A(R', R)$, 
and the functors 
$D := \opn{RHom}_A(-, R)$, $D' := \opn{RHom}_A(-, R')$,
$F := P \ot^{\mrm{L}}_A -$ and $F' := P' \ot^{\mrm{L}}_A -$.
We know that the functors $D, D', D' \circ D, D \circ D'$ have finite 
cohomological dimensions relative to $\cat{D}(A)$; the DG modules 
$P, P' \in \cat{D}^{-}_{\mrm{f}}(A)$; and the functors $F, F'$ have 
bounded above cohomological displacements relative to 
$\cat{D}(A)$.
For any $M \in \cat{D}(A)$ there is a canonical morphism 
\[ \opn{RHom}_A(R, R') \ot^{\mrm{L}}_A M \to 
\opn{RHom}_A ( \opn{RHom}_A(M, R), R') , \]
so we get a morphism of triangulated functors 
$\eta : F \to D' \circ D$. By definition $\eta_A$ is an isomorphism, and 
Theorem \ref{thm:10}(1) says that $\eta_M$ is an isomorphism for every 
$M \in \cat{D}^{-}_{\mrm{f}}(A)$. Likewise there is 
an isomorphism $\eta'_M : F'(M) \to (D \circ D')(M)$ for every 
$M \in \cat{D}^{-}_{\mrm{f}}(A)$. 

Let us calculate $P \ot^{\mrm{L}}_A P'$~:
\[ \begin{aligned}
& P \ot^{\mrm{L}}_A P' \cong F(P') \cong (D' \circ D)(P') 
\\ & \quad 
\cong (D' \circ D)(F'(A)) \cong (D' \circ D \circ D \circ D')(A) \cong A .
\end{aligned} \]
This proves $P$ is tilting. And 
\[ P \ot^{\mrm{L}}_A R \cong F(R) 
\cong (D' \circ D)(D(A)) \cong D'(A) \cong R' . \]

\medskip \noindent (3)
If $P$ is tilting and $R \cong P \ot^{\mrm{L}}_A R$, then 
\[ \begin{aligned}
& A \cong \opn{RHom}_A(R, R) \cong 
\opn{RHom}_A(P \ot^{\mrm{L}}_A R, R)
\\
& \quad \cong^{*} \opn{RHom}_A(P, \opn{RHom}_A(R,  R)) 
\cong \opn{RHom}_A(P, A)  , 
\end{aligned}  \]
where the isomorphism $\cong^{*}$ is by adjunction. 
But then 
\[ P \cong A \ot^{\mrm{L}}_A P \cong 
\opn{RHom}_A(P, A) \ot^{\mrm{L}}_A P \cong^{\dag} 
\opn{RHom}_A(P, P) \cong^{\dag \dag} A , \]
where the isomorphism $\cong^{\dag}$ is by a combination of Theorems 
\ref{thm:76} and \ref{thm:74}, and the isomorphism $\cong^{\dag \dag}$ is by  
Theorem \ref{thm:76}.
\end{proof}

\begin{cor} \label{cor:50}
Assume $A$ has some dualizing DG module. 
The formula $R \mapsto P \ot^{\mrm{L}}_A R$
induces a simply transitive action of the group $\opn{DPic}(A)$ on 
the set of isomorphism classes of dualizing DG $A$-modules. 
\end{cor}

\begin{proof}
Clear from the theorem.
\end{proof}

\begin{cor} \label{cor:80}
Assume $A$ has some dualizing DG module  \tup{(}e.g.\ $A$ is tractable\tup{)}. 
The formula 
$R \mapsto \opn{RHom}_A(\bar{A}, R)$ 
induces a bijection 
\[ \frac{ \{  \tup{dualizing DG} \, A \tup{-modules} \} }
{ \tup{isomorphism} } \iso 
\frac{ \{ \tup{dualizing DG} \, \bar{A} \tup{-modules} \}  }
{ \tup{isomorphism} } . \]
\end{cor}

\begin{proof}
By Corollary \ref{cor:50} the actions of the groups $\opn{DPic}(A)$ and 
$\opn{DPic}(\bar{A})$ on these two sets, respectively, are simply transitive.
And by Theorem \ref{thm:40} the group homomorphism 
$\opn{DPic}(A) \to \opn{DPic}(\bar{A})$
induced by $P \mapsto \bar{A} \ot^{\mrm{L}}_A P$ is bijective. 
Thus it suffices to prove that the function induced by 
$R \mapsto \opn{RHom}_A(\bar{A}, R)$ is equivariant for the action of 
$\opn{DPic}(A)$. Here is the calculation:
\[ \begin{aligned}
& \opn{RHom}_A(\bar{A}, P \ot^{\mrm{L}}_A R) \cong^{\ddag}
P \ot^{\mrm{L}}_A  \opn{RHom}_A(\bar{A}, R) 
\\ & \qquad \cong 
(\bar{A} \ot^{\mrm{L}}_A P) \ot^{\mrm{L}}_{\bar{A}}  \opn{RHom}_A(\bar{A}, R) . 
\end{aligned} \]
The isomorphism $\cong^{\ddag}$ comes from Lemma \ref{lem:45}  
below, noting that the tilting DG module $P$ satisfies condition ($*$) of the 
lemma, since it is perfect.
\end{proof}

We say that a DG $A$-module $N$ has {\em bounded below generation} if it is 
generated in the integer interval $[i_0, \infty]$ for some integer $i_0$;
see Definition \ref{dfn:390}. 

\begin{lem} \label{lem:45}
Let $L \in \cat{D}^{-}_{\mrm{f}}(A)$ and $M, N \in \cat{D}^{+}(A)$.
Assume that $N$ satisfies this condition\tup{:}
\begin{enumerate}
\item[($*$)] There is a covering sequence $(s_1, \ldots, s_n)$ of 
$\bar{A}$, and for every $i$ there is 
an isomorphism $A_{s_i} \ot_A N \cong \til{N}_i$ in $\cat{D}(A_{s_i})$,
where $\til{N}_i$ is a K-flat DG $A_{s_i}$-module with bounded below 
generation.
\end{enumerate}
Then the canonical morphism 
\[ \psi_{L, M, N} : \opn{RHom}_A(L, M) \ot^{\mrm{L}}_A N \to 
\opn{RHom}_A(L, M \ot^{\mrm{L}}_A N) \]
in $\cat{D}(A)$, from formula \tup{(\ref{eqn:400})}, is an isomorphism.
\end{lem}

\begin{proof}
Step 1. Here we assume that $N \cong \til{N}$ in $\cat{D}(A)$, where $\til{N}$ 
is a K-flat DG $A$-module of bounded below generation. Using smart truncation 
if needed, we can assume that the DG 
$B$-module $M$ is bounded below. Let $\til{L} \to L$ be a pseudo-finite 
semi-free resolution over $A$ (see Proposition \ref{prop:101}).
The morphism $\psi_{L, M, N}$ is represented by the homomorphism 
\[ \til{\psi}_{\til{L}, M, \til{N}} :  \opn{Hom}_A(\til{L}, M) \ot^{}_A \til{N} 
\to \opn{Hom}_A(\til{L}, M \ot^{}_A \til{N}) \]
in $\cat{C}(A)$. 
Because the semi-free DG $A$-module $\til{L}$ is bounded above and has 
finitely many basis elements in each degree, and both $M$ and $M \ot_A \til{N}$ 
are bounded below, we see that $\til{\psi}_{\til{L}, M, \til{N}}$ is bijective. 

\medskip \noindent 
Step 2. 
Here $N$ satisfies condition ($*$).  We claim that the obvious morphisms
\begin{equation} \label{eqn:415}
\bigl( \opn{RHom}_A(L, M) \ot^{\mrm{L}}_A N \bigr) \ot_{A^0} A^0_{s_i} \to
\opn{RHom}_A(L, M) \ot^{\mrm{L}}_{A} \til{N}_i 
\end{equation}
and 
\begin{equation} \label{eqn:416}
\opn{RHom}_A(L, M \ot^{\mrm{L}}_A N) \ot_{A^0} A^0_{s_i} \to
\opn{RHom}_A(L, M \ot^{\mrm{L}}_{A} \til{N}_i) 
\end{equation}
in $\cat{D}(A)$, that come from the given isomorphisms 
$A_{s_i} \ot_A N \iso \til{N}_i$, are isomorphisms.
That  (\ref{eqn:415}) is an isomorphism is trivial.
As for the morphism  (\ref{eqn:416}): condition ($*$) implies that 
$M \ot^{\mrm{L}}_A N$ belongs to $\cat{D}^{+}(A)$. Since $A^0_{s_i}$ is a 
K-flat DG module over $A^0$ generated in degree $0$, we can use Step 1.

\medskip \noindent 
Step 3. Now we are in the general situation. 
Let $\psi_i$ be the morphism gotten from $\psi_{L, M, N}$ by the localization 
$A^0_{s_i} \ot_{A^0} - $, so $\psi_i$  goes from the first object in 
(\ref{eqn:415}) to the first object in (\ref{eqn:416}). 
Since $\bar{A} \to \prod_i \bar{A}_{s_i}$ is faithfully flat, it suffices to 
prove that all the $\psi_i$ are isomorphisms. 
But by step 2 it suffices to show that 
\[ \psi_{L, M, \til{N}_i} : \opn{RHom}_A(L, M) \ot^{\mrm{L}}_{A} \til{N}_i \to 
\opn{RHom}_A(L, M \ot^{\mrm{L}}_{A} \til{N}_i)  \]
is an isomorphism. Since $\til{N}_i$ is a K-flat DG $A$-module with bounded 
below generalization, we can use step 1.
\end{proof}

\begin{cor} \label{cor:41}
If the ring $\bar{A}$ is local, then any two dualizing DG $A$-modules 
$R$ and $R'$ satisfy $R' \cong R[m]$ for some integer $m$. 
\end{cor}

\begin{proof}
By Corollary \ref{cor:315} we have 
$\opn{DPic}(A) \cong \Z$, generated by the class of $A[1]$. Now use 
Corollary \ref{cor:50}.
\end{proof}

\begin{prop} \label{prop:80}
Let $A$ be a cohomologically noetherian  DG ring, and let 
$R$ be a dualizing DG $A$-module. Assume $A$ is cohomologically bounded.
Then $R$ is dualizing in the sense of 
\cite[Definition 1.8]{FIJ}.
\end{prop}

\begin{proof}
There are four conditions in \cite[Definition 1.8]{FIJ}. Condition (1) -- 
the existence of resolutions -- is trivial in our commutative situation. 
Condition (2) says that if $M \in \cat{D}^{\mrm{b}}_{\mrm{f}}(A)$ then 
$\opn{RHom}_A(M, R) \in  \cat{D}^{\mrm{b}}_{\mrm{f}}(A)$; and this is true by 
Proposition \ref{prop:250}(1). 
Condition (3) requires that for any $M \in \cat{D}^{\mrm{b}}_{\mrm{f}}(A)$,
letting $N$ be either $M$ or $M \ot^{\mrm{L}}_{A} R$, 
the adjunction morphisms 
\[ N \to \opn{RHom}_A \bigl( \opn{RHom}_A(N, R), R \bigr) \]
are isomorphisms. Now by Corollary \ref{cor:380} we know that 
$R \in \cat{D}^{\mrm{b}}_{\mrm{f}}(A)$.
A combination of Proposition \ref{prop:35} and Theorem \ref{thm:30}(1)
tells us that $M \ot^{\mrm{L}}_{A} R \in \cat{D}^{-}_{\mrm{f}}(A)$.
Thus in both cases $N \in \cat{D}^{-}_{\mrm{f}}(A)$, and 
according to Proposition \ref{prop:250}(2) the morphism in question is an 
isomorphism. Condition (4) is part of condition (3) in the commutative 
situation. 
\end{proof}

\begin{dfn} \label{dfn:405}
A cohomologically noetherian cohomologically bounded DG ring $A$ is 
called {\em Gorenstein} if the DG module $A$ has finite injective dimension 
relative to $\cat{D}(A)$.
\end{dfn}

\begin{prop} \label{prop:45}
Let $A$ be a cohomologically noetherian cohomologically bounded DG ring. The 
following conditions are equivalent\tup{:}
\begin{enumerate}
\rmitem{i} $A$ is Gorenstein. 
\rmitem{ii} The DG $A$-module $A$ is dualizing. 
\end{enumerate}
\end{prop}

\begin{proof}
Since conditions (i) and (iii) of Definition \ref{dfn:253}
are automatic for $R := A$, this is clear. 
\end{proof}

\begin{rem}
For DG rings that are not cohomologically bounded, a comparison like in 
Proposition \ref{prop:80}  does not seem to work nicely.  
Corollary \ref{cor:41} is very similar to \cite[Theorem III]{FIJ}; 
but of course the assumptions are not the same. 

We do not know a reasonable definition of Gorenstein DG rings that are not 
cohomologically bounded. 
\end{rem}

Here is a rather surprising result, that was pointed out to us 
by J{\o}rgensen. 

\begin{thm} \label{thm:73}
Let $A$ be a DG ring, which is cohomologically bounded and cohomologically 
essentially finite type over some noetherian ring $\K$. If $\bar{A}$ is a 
perfect DG $A$-module, then the canonical homomorphism $A \to \bar{A}$ is a 
quasi-isomorphism. 
\end{thm}

Example \ref{exa:400} shows that the assumption that $A$ is cohomologically 
bounded is really needed. 

\begin{proof}
We will prove that $\bar{A}_{\p} \ot_{\bar{A}} \opn{H}^i(A) = 0$
for every $i < 0$ and every $\p \in \opn{Spec} \bar{A}$. 

Fix such $i$ and $\p$. Because the assertion is invariant under DG ring 
quasi-iso\-morphisms, we may assume, by Lemma \ref{lem:380} (replacing 
$A$ with $A_{\mrm{eft}}$), that $A^0$ 
is a noetherian ring. Consider the ring 
$A_{\p}^0 := \til{S}^{-1} \cd A^0$, 
 where $\pi : A \to \bar{A}$ is the canonical homomorphism, 
$S := \bar{A} - \p$, and $\til{S} := \pi^{-1}(S) \cap A^0$. 
Next define the DG ring 
$A_{\p} := A_{\p}^0 \ot_{A^0} A$.
Then $A^0_{\p} \to \bar{A}_{\p}$ is surjective, and
$A^0_{\p}$ is a noetherian local ring. By Proposition \ref{prop:81}(2) 
the DG $A_{\p}$-module $\bar{A}_{\p} \cong A_{\p} \ot_A \bar{A}$ is perfect; 
and by Theorem \ref{thm:74} this is a compact object of 
$\cat{D}(A_{\p})$. Also $\bar{A}_{\p}$ is nonzero. 
According to \cite[Theorem 0.2]{Jo} we have 
$\opn{amp}(\opn{H}(A_{\p})) \leq \opn{amp}(\opn{H}(\bar{A}_{\p})) = 0$.
Therefore 
$\opn{H}^i(A_{\p}) = 0$ for all $i < 0$. 
But $\opn{H}^i(A_{\p}) \cong \bar{A}_{\p} \ot_{\bar{A}} \opn{H}^i(A)$.
\end{proof}

We conclude this section with several examples and remarks. 

\begin{exa}
Suppose $A$ is a Gorenstein noetherian ring, and 
$\bsym{a} = (a_1, \ldots, a_n)$ is a sequence of elements in $A$. 
Let $B := \opn{K}(A; \bsym{a})$, the Koszul complex, which is cohomologically 
noetherian and cohomologically bounded. The DG ring homomorphism 
$A \to B$ is cohomologically pseudo-finite, $R_A := A[n]$ is a dualizing DG 
$A$-module, and hence 
$R_B := \opn{RHom}_A(B, R_A)$ is a dualizing DG $B$-module. But $B$ is 
semi-free as DG $A$-module, and therefore 
\[ R_B = \opn{RHom}_A(B, A[n]) \cong \opn{Hom}_A(B, A[n]) \cong B \]
in $\cat{D}(B)$. We see that $B$ is Gorenstein, in the sense of Definition 
\ref{dfn:405}.
\end{exa}

\begin{exa} \label{exa:113}
Here is a comparison to Hinich's notion of dualizing DG module from \cite{Hi1}.
Let $A$ be a noetherian local ring, with maximal ideal $\m$. 
For the sake of simplicity, let us assume that $A$ contains a field $\K$, such 
that $\K \to A / \m$ is finite. 
Let $\bsym{a} = (a_1, \ldots, a_n)$ be a sequence of elements in $A$ that 
generates an $\m$-primary ideal, and let 
$B := \opn{K}(A; \bsym{a})$ be the associated Koszul complex. 
Thus $\K \to B$ is a cohomologically pseudo-finite homomorphism of DG rings, 
and according to Proposition \ref{prop:252} the DG $B$-module
$R_B := \opn{Hom}_{\K}(B, \K)$ is dualizing. Now 
$\opn{Hom}_{\K} \bigl( \opn{H}(R_B), \K \bigr) \cong \opn{H}(B)$
as graded $\opn{H}(B)$-modules, so $R_B$ is a dualizing DG module in the sense 
of Hinich. Taking Corollary \ref{cor:41} into consideration, we see that any 
dualizing DG $B$-module $R$ (in our sense) satisfies the condition of Hinich. 
\end{exa}

\begin{rem} \label{rem:317}
The results in our paper so far suggest an analogy between the 
following two scenarios:
\begin{itemize}
\item[(dg)] The {\em DG scenario}: a cohomologically pseudo-noetherian 
DG ring $A$, with reduction $\bar{A} := \opn{H}^0(A)$. 
\item[(ad)] The {\em adic scenario}: a noetherian ring $A$, 
$\a$-adically complete w.r.t.\ an ideal $\a$,  with reduction 
$\bar{A} := A / \a$.
\end{itemize}
We refer to this as the {\em DG vs.\ adic analogy}.
This analogy restricts to the ``degenerate cases'' in these scenarios:
\begin{itemize}
\item[(dg)] The cohomology $\opn{H}(A)$ is bounded.
\item[(ad)] The defining ideal $\a$ is nilpotent. 
\end{itemize}
Of course, this observation is not new (cf.\ \cite{Lu1}, \cite{Lu2}, \cite{TV} 
and \cite{AG}). 

The DG vs.\ adic analogy holds also for ``finite homomorphisms'':
\begin{itemize}
\item[(dg)] A cohomologically pseudo-finite homomorphism $f : A \to B$
between cohomologically pseudo-noetherian DG rings 
(Definition \ref{dfn:252}).
\item[(ad)] A {\em formally finite} or {\em pseudo-finite} 
homomorphism $f : A \to B$ between adically complete noetherian  
rings, as in \cite{Ye1} and  \cite{AJL2} respectively.
\end{itemize}
There is a further analogy between ``dualizing objects'' in the two scenarios:
\begin{itemize}
\item[(dg)] A dualizing DG module $R$ over a cohomologically pseudo-noetherian 
 DG ring $A$ (Definition \ref{dfn:253}).
\item[(ad)] A {\em t-dualizing complex} $R$ over an adically complete 
noetherian ring $A$, as in \cite{Ye1} and \cite{AJL1}.
\end{itemize}

The analogies above raise two questions: 
\begin{enumerate}
\item Is there a DG analogue of the {\em c-dualizing complex} of \cite{AJL1}?
\item Is there a DG analogue of the {\em GM Duality} of \cite{AJL1} and the
{\em MGM Equivalence} of \cite{PSY}?
\end{enumerate}
\end{rem}

\begin{rem} \label{rem:70}
Recall that a noetherian ring $A$ of finite Krull dimension is regular (i.e.\ 
all its local rings $A_{\p}$ are regular) iff it has finite global 
cohomological dimension. 

Now suppose $A$ is a cohomologically pseudo-noetherian DG ring. 
By ``Krull dimension'' we could mean that of $\bar{A}$, but 
``regular local ring'' has no apparent meaning here.
Hence we propose this definition:  $A$ is  called
{\em regular} if it has {\em finite global cohomological dimension}. By this we 
mean that there is a natural number $d$, such that if
$M \in \cat{C}(A)$ is generated in the integer interval $[i_0, i_1]$, then 
$M$ has projective dimension at most $d + i_1 - i_0$; cf.\ Definitions 
\ref{dfn:390} and \ref{dfn:11}, and Examples \ref{exa:390} and \ref{exa:520}.
According to Theorem \ref{thm:50}, we see that any 
$M \in \cat{D}^{\mrm{b}}_{\mrm{f}}(A)$, including $M = \bar{A}$, is perfect.

Now assume that $A$ is a regular DG ring, but with {\em bounded 
cohomology}. Then, taking $M = \bar{A}$, Theorem \ref{thm:73} says that 
$A \to \bar{A}$ is a quasi-isomorphism. The conclusion is that {\em 
the only regular DG rings with bounded cohomology are the regular rings} (up to 
quasi-isomorphism). 

Under the DG vs.\ adic analogy of Remark \ref{rem:317}, this corresponds 
to an adic ring $A$ with a nilpotent defining ideal $\a$. If $A$ is 
regular, then it cannot have nonzero nilpotent elements. Therefore $\a = 0$ 
here, and $A \to \bar{A}$ is bijective.
\end{rem}

\begin{exa} \label{exa:400}
Take a field $\K$, and let $A := \K[t]$, the polynomial ring in a variable $t$ 
of degree $-2$. We view $A$ as a DG ring with zero differential, so 
$\opn{H}(A) \cong A$, and it is cohomologically noetherian, but not 
cohomologically bounded below. The DG ring 
homomorphism $\K \to A$ is cohomologically pseudo-finite. Hence the DG 
$A$-module $R := \opn{Hom}_{\K}(A, \K)$ is dualizing. This DG module is not 
bounded above.
 
Note that here $\bar{A} \cong \K$ is a perfect DG $A$-module.
To show this, we shall produce a finite semi-free resolution of $\bar{A}$.
The DG module $A^{\leq -2}$, which is both the stupid 
and the smart truncation of $A$ at $-2$, is free, since 
$A^{\leq -2} \cong A[2]$ as DG $A$-modules. Let 
$\phi : A^{\leq -2} \to A$ be the inclusion, and let 
$P := \opn{cone}(\phi)$. There is an obvious quasi-isomorphism 
$P \to \bar{A}$.   

The adic analogue is the ring of formal power series 
$A := \K[[t]]$, with ideal of definition $\a := (t)$. The corresponding 
t-dualizing complex is $R := \opn{Hom}_{\K}^{\mrm{cont}}(A, \K)$,
which is an artinian $A$-module of infinite length. 
\end{exa}

\begin{rem} \label{rem:316}
Our definition of dualizing DG modules, Definition \ref{dfn:253}, might seem 
an almost straightforward generalization of Grothendieck's original definition 
in \cite{RD}. However there are at least two subtle points: (a) Finding 
the correct notion of injective dimension of a DG module (condition (ii) of 
Definition \ref{dfn:253}). (b) Allowing a dualizing DG module to have unbounded 
above cohomology (condition (i) of Definition \ref{dfn:253}; 
cf.\ Corollary \ref{cor:380} and Remark \ref{rem:317}).

All results in this section, up to and including Corollary 
\ref{cor:50}, might also seem to be straightforward generalizations of 
Grothendieck's corresponding results in \cite{RD}. But the technical 
difficulties (mainly when $A$ is cohomologically unbounded) cannot be 
neglected. 

We should mention that Theorem \ref{thm:51} can be made a bit stronger, by 
replacing the condition that $A$ is tractable with the weaker condition that 
there is a cohomologically essentially finite type homomorphism $\K \to A$, 
where $\K$ is a noetherian ring with a dualizing complex.
Theorem \ref{thm:73} can be similarly strengthened.

Some earlier papers, notably \cite{Hi1} and \cite{FIJ}, had adopted
other definitions of dualizing DG modules; see Example \ref{exa:113} and
Proposition \ref{prop:80} respectively. These definitions are not consistent 
with our definition in general, and there does not appear to be a 
well-developed theory for them.  

In \cite[Definition 4.2.5]{Lu2}, Lurie gives a definition of a
dualizing $\mrm{E}_{\infty}$ module 
over an $\mrm{E}_{\infty}$ ring. Now any DG ring $A$ can be viewed as an 
$\mrm{E}_{\infty}$ ring, and DG $A$-modules can be viewed $\mrm{E}_{\infty}$ 
$A$-modules. Under this correspondence, it seems that
a dualizing DG $A$-module, in the sense of Definition \ref{dfn:253} above,
becomes a dualizing module in the sense of \cite{Lu2}. (We are being careful, 
because a precise comparison of the definitions is not so easy.)
Thus our results in this section, up to and including Corollary 
\ref{cor:50}, might be viewed as special instances of Lurie's  
statements. Still, an attempt to produce a full proof of our results
based on the corresponding results in \cite{Lu2} (e.g.\ deducing our Theorem 
\ref{thm:51} from \cite[Theorem 4.3.14]{Lu2}, or deducing our Theorem 
\ref{thm:43} from \cite[Proposition 4.2.9]{Lu2}) might be nontrivial, and most 
likely it would be longer than our own direct proofs. This is because, as far 
as we know, there do not exist full comparison results for the monoidal 
operations between the $\mrm{E}_{\infty}$ and the DG setups. 
 
Our Corollary \ref{cor:80} appears to be totally new. We could not find 
anything resembling it in Lurie's papers, nor elsewhere in the literature. 
Likewise for Theorem \ref{thm:73} (except for J{\o}rgensen's original local 
result).
\end{rem}

\begin{rem} \label{rem:315}
A result that is noticeably missing from our paper is a DG analogue of 
\cite[Theorem 4.3.5]{Lu2}. Translated to the DG terminology, it states that if 
the ring $\bar{A}$ admits a dualizing DG module, then the DG ring $A$ 
admits a dualizing DG module. We do not know whether this result can be 
proved within the DG framework; this is a question that we find interesting.

Note however that the corresponding result in the adic case, namely when 
$A$ is a complete $\a$-adic ring extension of $\bar{A}$ (cf.\ Remark 
\ref{rem:317}), was proved a long time ago by Faltings \cite{Fa}.
The proof of the nilpotent case in \cite{Fa} is quite easy; but the  
passage to the complete adic case is somewhat involved there. The proof can be 
greatly simplified by first proving the existence of a t-dualizing complex
$R'_A$ over $A$, and then applying derived completion to obtain a c-dualizing 
complex $R_A := \mrm{L} \Lambda_{\a}(R'_A)$. 
See \cite{Ye1} \cite{AJL1}, \cite{AJL2} and \cite{PSY} for information on 
derived completion, and on dualizing complexes over adic rings.
\end{rem}

\section{Cohen-Macaulay DG Modules} \label{sec:cm}

In this section we work with cohomologically pseudo-noetherian
commutative DG rings (see Convention \ref{conv:100} and Definition 
\ref{dfn:391}).  

Let $A$ be such a DG ring. Recall that $\bar{A} = \opn{H}^0(A)$,  
and $\cat{D}^0(A)$ is the full subcategory of $\cat{D}(A)$ 
consisting of the DG modules $M$ such that $\opn{H}^i(M) = 0$ for all 
$i \neq 0$. 
Inside $\cat{D}^{0}(A)$ we have 
$\cat{D}^{0}_{\mrm{f}}(A) =
\cat{D}_{\mrm{f}}(A) \cap \cat{D}^{0}(A)$.

\begin{lem} \label{lem:530} 
Consider the canonical DG ring homomorphism $\pi :A \to \bar{A}$. 
The functor 
\[ \opn{Q} \circ \opn{rest}_{\pi} : \cat{Mod} \bar{A} \to 
\cat{D}^0(A) \]
is an equivalence. It restricts to an equivalence 
\[ \opn{Q} \circ \opn{rest}_{\pi} : \cat{Mod}_{\mrm{f}} \bar{A} \to 
\cat{D}^0_{\mrm{f}}(A) . \]
\end{lem}

\begin{proof}
Smart truncation shows that any object of $\cat{D}^0(A)$ is isomorphic to an 
object of $\cat{Mod} \bar{A}$. Finiteness of $\bar{A}$-modules is preserved. It 
remains to show that $\opn{Q} \circ \opn{rest}_{\pi}$ is a fully faithful 
functor. 

So take $M, N \in \cat{Mod} \bar{A}$, and let 
$\til{M} \to M$ be a semi-free resolution over $A$ with 
$\opn{sup}(\til{M}) \leq 0$. Then 
\[ \opn{Hom}_{\cat{D}(A)}(M, N) \cong 
\opn{H}^0 \bigl( \opn{Hom}_{A}(\til{M}, N) \bigr) \cong
\opn{Hom}_{\bar{A}}(\opn{H}^0 (\til{M}), N) \cong
\opn{Hom}_{\bar{A}}(M, N) . \]
\end{proof}

\begin{dfn}
Let let $R$ be a dualizing DG $A$-module. A 
DG module $M \in \cat{D}^{\mrm{b}}_{\mrm{f}}(A)$ is called
{\em Cohen-Macaulay with respect to $R$} if 
$\opn{RHom}_A(M, R) \in \cat{D}^0_{\mrm{f}}(A)$.
\end{dfn}

In other words, the condition is that $\opn{RHom}_A(M, R)$ is 
isomorphic, in $\cat{D}(A)$, to an object of $\cat{Mod}_{\mrm{f}} \bar{A}$. 
As usual ``Cohen-Macaulay'' is abbreviated to ``CM''. Let us denote by 
$\cat{D}^{\mrm{b}}_{\mrm{f}}(A)_{\mrm{CM} : R}$
the full  subcategory of $\cat{D}^{\mrm{b}}_{\mrm{f}}(A)$ consisting of DG 
modules that are CM w.r.t.\ $R$. 

\begin{rem}
Observe that the functor $\opn{RHom}_A(-, R)$
gives rise to a duality between 
$\cat{D}^{\mrm{b}}_{\mrm{f}}(A)_{\mrm{CM} : R}$
and $\cat{D}^0_{\mrm{f}}(A)$. And the latter is equivalent to 
$\cat{Mod}_{\mrm{f}} \bar{A}$. Therefore 
$\cat{D}^{\mrm{b}}_{\mrm{f}}(A)_{\mrm{CM} : R}$
is an artinian abelian category.

If $A \to \bar{A}$ is not a quasi-isomorphism, then $A$ does not belong to 
$\cat{D}^0(A)$, and therefore $R$ is not a CM DG module w.r.t.\ itself. 

We do not know any definition of Cohen-Macaulay DG rings; except when 
$A \to \bar{A}$ is a quasi-isomorphism, in which case the condition is that the 
ring $\bar{A}$ should be CM. 

For a comparison to Cohen-Macaulay modules and Grothendieck's notion of 
Cohen-Macaulay complexes, see \cite[Theorem 6.2]{YZ3} and 
\cite[Section 7]{YZ4}.
\end{rem}

The groups $\opn{DPic}^0(A) \subseteq \opn{DPic}(A)$ were introduced in 
Definitions \ref{dfn:350} and \ref{dfn:330}.

\begin{lem} \label{lem:525} 
Let $P$ be a tilting DG $A$-module. The following are equivalent\tup{:}
\begin{enumerate}
\rmitem{i} The auto-equivalence $P \ot^{\mrm{L}}_A -$ of $\cat{D}(A)$ preserves 
the subcategory $\cat{D}^0_{\mrm{f}}(A)$. 

\rmitem{ii} The class of $P$ is in $\opn{DPic}^0(A)$.
\end{enumerate}
\end{lem}

\begin{proof}
(i) $\Rightarrow$ (ii): 
Let $\bar{P} := \bar{A} \ot^{\mrm{L}}_A P$. By assumption it belongs to 
$\cat{D}^0_{\mrm{f}}(A)$. But then the class of $\bar{P}$ is in  
$\opn{Pic}(\bar{A}) = \opn{DPic}^0(\bar{A})$, 
so by definition the class of $P$ is in $\opn{DPic}^0(A)$.

\medskip \noindent 
(ii) $\Rightarrow$ (i): Here $\bar{P} := \bar{A} \ot^{\mrm{L}}_A P$
is isomorphic to an invertible $\bar{A}$-module, so 
$\bar{P} \ot^{\mrm{L}}_{\bar{A}} -$ preserves 
$\cat{D}^0_{\mrm{f}}(\bar{A})$.
Now take any $M \in \cat{D}^0_{\mrm{f}}(A)$. 
By Lemma \ref{lem:530} we can assume that $M \in 
\cat{Mod}_{\mrm{f}}(\bar{A})$. Then 
\[ P \ot^{\mrm{L}}_A M \cong 
P \ot^{\mrm{L}}_A \bar{A} \ot^{\mrm{L}}_{\bar{A}} M \cong 
\bar{P} \ot^{\mrm{L}}_{\bar{A}} M \in \cat{D}^0_{\mrm{f}}(\bar{A}) . \]
We see that $\opn{H}^i(P \ot^{\mrm{L}}_A M) = 0$ for all $ i \neq 0$. 
\end{proof}

Recall the connected component idempotent functors of a DG ring from Definition 
\ref{dfn:343}.

\begin{thm} \label{thm:250}
Let $f : A \to B$ be a cohomologically pseudo-finite homomorphism between 
cohomologically pseudo-noetherian DG rings. Assume that $A$ and 
$B$ have dualizing DG modules $R_A$ and $R_B$ respectively, and that $\bar{B}$ 
is nonzero. Let $E_1, \ldots, E_n$ be the connected component decomposition 
functors of $B$. 
\begin{enumerate}
\item There are unique integers $k_1, \ldots, k_n$ such that, letting 
\[ R'_B := \bigoplus_{i = 1}^n E_i(R_{B})[k_i] \in \cat{D}(B) , \] 
the class of the tilting DG $B$-module $\opn{RHom}_A(R'_B, R_A)$ 
is inside $\opn{DPic}^0(B)$. 

\item Let $M \in \cat{D}^{\mrm{b}}_{\mrm{f}}(B)$. 
The following conditions are equivalent\tup{:}
\begin{enumerate}
\rmitem{i} $M$ is CM w.r.t.\ to $R'_B$. 
  
\rmitem{ii} $\opn{rest}_f(M)$ is CM w.r.t.\ to $R_A$. 
\end{enumerate}
\end{enumerate}
\end{thm}

Note that $R'_B$ and $R''_B := \opn{RHom}_A(B, R_A)$ are dualizing DG 
$B$-modules, so 
\[ \opn{RHom}_A(R'_B, R_A) \cong \opn{RHom}_B ( R'_B, R''_B)  \]
is a tilting DG $B$-module (by Theorem \ref{thm:43}(2)), and hence item (1) 
above makes sense.

\begin{proof}
(1) Let $R''_B$ be as above.
By the classifications in Corollaries \ref{cor:50} and \ref{cor:330},
there are unique $k_1, \ldots, k_n \in \Z$, and a tilting DG $B$-module 
$Q$ whose class in $\opn{DPic}^0(B)$ is unique, such that 
$R''_B \cong Q \ot^{\mrm{L}}_B R'_B$. 
Let $P := \opn{RHom}_B(Q, B)$, which by Corollary \ref{cor:405}
is the quasi-inverse of $Q$. Then  
$R'_B \cong P \ot^{\mrm{L}}_B R''_B$. 
Using adjunction we get isomorphisms
\[  \begin{aligned}
& \opn{RHom}_A(R'_B, R_A) \cong \opn{RHom}_B(R'_B, R''_B)
\cong \opn{RHom}_B(P \ot^{\mrm{L}}_B R''_B, R''_B)
\\
& \quad \cong 
\opn{RHom}_B \bigl(P, \opn{RHom}_B(R''_B, R''_B) \bigr)
\cong \opn{RHom}_B(P, B) \cong Q
\end{aligned} \]
in $\cat{D}(B)$. This proves (1). 

\medskip \noindent 
(2) By Lemma \ref{lem:525} the
functor $Q \ot^{\mrm{L}}_B -$ preserves 
$\cat{D}^0_{\mrm{f}}(B)$.
Take any $M \in \cat{D}^{\mrm{b}}_{\mrm{f}}(B)$.
Then, using Lemma \ref{lem:45}, we get isomorphisms
\[ \begin{aligned}
& \opn{RHom}_A(M, R_A) \cong \opn{RHom}_B(M, R''_B)
\\
& \quad 
\cong \opn{RHom}_B(M, Q \ot^{\mrm{L}}_B R'_B) 
\cong \opn{RHom}_B(M, R'_B) \ot^{\mrm{L}}_B Q . 
\end{aligned} \]
This gives (i) $\Rightarrow$ (ii). The converse is very similar. 
\end{proof}

\begin{cor}
Let $f : A \to B$ be a cohomologically finite homomorphism between tractable DG 
rings. Assume $\opn{Spec} \bar{B}$ is connected. The following are equivalent 
for $M \in \cat{D}^{\mrm{b}}_{\mrm{f}}(B)$~\tup{:}
\begin{enumerate}
\rmitem{i} $M$ is CM w.r.t.\ some dualizing DG $B$-module.
\rmitem{ii} $\opn{rest}_f(M)$ is CM w.r.t.\ some dualizing DG $A$-module.
\end{enumerate}
\end{cor}

\begin{proof}
The implication (ii) $\Rightarrow$ (i) comes from Theorem \ref{thm:250}(2) -- 
and does not require $\opn{Spec} \bar{B}$ to be connected.

For the implication (i) $\Rightarrow$ (ii), let $R_B$ be a dualizing DG 
$B$-module such that $M$ is CM with respect to it. Take any dualizing DG 
$A$-module $R_A$, and let $k \in \Z$ be such that 
$\opn{RHom}_A(R_B[k], R_A)$ 
is inside $\opn{DPic}^0(B)$. See Theorem \ref{thm:250}(1). Then by 
Theorem \ref{thm:250}(2) we know that $\opn{rest}_f(M)$ is CM w.r.t.\ 
$R_A[-k]$. 
\end{proof}

\begin{thm} \label{thm:251}
Let $f : A \to B$ be a homomorphism between 
cohomologically pseudo-noetherian DG rings, such that 
$\bar{f} : \bar{A}  \to \bar{B}$
is surjective. Assume $A$ and $B$ have dualizing DG modules.
Let $R_B$ be a dualizing DG $B$-module and let  
$M, N \in \cat{D}^{\mrm{b}}_{\mrm{f}}(B)$. 
\begin{enumerate}
\item If $M$ is CM w.r.t.\ $R_B$, and there is an isomorphism
$\opn{rest}_f(M) \cong \opn{rest}_f(N)$ in $\cat{D}(A)$, then 
$N$ is also CM w.r.t.\ $R_B$. 

\item If $M$ and $N$ are CM w.r.t.\ $R_B$, then the homomorphism 
\[ \opn{rest}_f : \opn{Hom}_{\cat{D}(B)}(M, N) \to 
\opn{Hom}_{\cat{D}(A)} \bigl( \opn{rest}_f(M), \opn{rest}_f(N) \bigr) \]
is bijective. 
\end{enumerate}
\end{thm}

\begin{proof}
(1)  We may assume that $\bar{B} \neq 0$. 
Let $E_1, \ldots, E_n$ be the connected component decomposition functors 
of $B$. Write $F := \opn{rest}_f$, $M_i := E_i(M)$ and $N_i := E_i(N)$.
Choose some dualizing DG $A$-module $R_A$, and let 
$k_1, \ldots, k_n$ be the integers from Theorem \ref{thm:250}(1). 
Thus the dualizing DG $B$-module 
$R'_B := \bigoplus_{i = 1}^n E_i(R_{B})[k_i]$
satisfies this: the tilting DG $B$-module $Q := \opn{RHom}_A(R'_B, R_A)$ 
is inside $\opn{DPic}^0(B)$. 
Define $M' := \bigoplus_i M_i[k_i]$ and $N' := \bigoplus_i N_i[k_i]$.

We are given that $M$ is CM w.r.t.\ $R_B$. 
Using the equivalence of Corollary \ref{cor:22} we see that 
$M'$ is CM w.r.t.\ $R'_B$. Thus by Theorem \ref{thm:250}(2) 
the DG $A$-module $F(M')$ is CM w.r.t.\ $R_A$.
Proposition \ref{prop:21} implies that 
$F(M') \cong F(N')$ in $\cat{D}(A)$, and hence  $F(N')$ is CM w.r.t.\ $R_A$.
Using Theorem \ref{thm:250}(2) once more we conclude that 
$N'$ is CM w.r.t.\ $R'_B$; and hence $N$ is CM w.r.t.\ $R_B$. 

\medskip \noindent
(2) Let us write
$R''_B := \opn{RHom}_A(B, R_A)$, $D''_B := \opn{RHom}_B(-, R''_B)$ and 
$D_A := \lb \opn{RHom}_A(-, R_A)$. 
Define $M' := \bigoplus_i M_i[k_i]$ and $N' := \bigoplus_i N_i[k_i]$
as above, so these are CM w.r.t.\ $R''_B$. There are isomorphisms 
\[ \begin{aligned}
& \opn{Hom}_{\cat{D}(B)}(M, N) \cong^1
\opn{Hom}_{\cat{D}(B)}(M', N') 
\cong^2 \opn{Hom}_{\cat{D}(B)} \bigl( D''_B(N'), D''_B(M') \bigr) 
\\ & \quad 
\cong^3 \opn{Hom}_{\cat{Mod} \bar{B}} \bigl( D''_B(N'), D''_B(M') \bigr) 
\cong^{4} \opn{Hom}_{\cat{Mod} \bar{A}} \bigl( F(D''_B(N')), F(D''_B(M')) 
\bigr) 
\\ & \quad 
\cong^5 \opn{Hom}_{\cat{Mod} \bar{A}} \bigl( D_A(F(N')), D_A(F(M')) \bigr)
\cong^{2, 3} \opn{Hom}_{\cat{D} (A)} \bigl( F(M'), F(N') \bigr)
\\ & \quad 
\cong^{6} \opn{Hom}_{\cat{D} (A)} \bigl( F(M), F(N) \bigr) \ .
\end{aligned} \]
They are gotten as follows:
the isomorphism $\cong^{1}$ is by Corollary \ref{cor:22} and Proposition 
\ref{prop:90};
the isomorphism $\cong^{2}$ is by Proposition \ref{prop:250}(3);
the isomorphism $\cong^{3}$ is by Lemma \ref{lem:530};
the isomorphism $\cong^{4}$ is because
$\opn{H}^0(f) : \bar{A}  \to \bar{B}$ is surjective;
the isomorphism $\cong^{5}$ is because 
$F \circ D''_B \cong D_A \circ F$ as functors;
and isomorphism $\cong^{6}$ is due to Proposition \ref{prop:21}. 
The composition of all these isomorphisms is $F$.
\end{proof}

\begin{cor}
Let $f : A \to B$ be a homomorphism between 
cohomologically pseudo-noetherian DG rings, such that 
$\bar{f} : \bar{A}  \to \bar{B}$ is bijective. 
Let $R_A$ be a dualizing DG $A$-module, and define 
$R_B := \opn{RHom}_{A}(B, R_A)$. Then the functor
\[ \opn{rest}_f : \cat{D}^{\mrm{b}}_{\mrm{f}}(B)_{\mrm{CM} : R_B} \to 
\cat{D}^{\mrm{b}}_{\mrm{f}}(A)_{\mrm{CM} : R_A} \]
is an equivalence.
\end{cor}

\begin{proof}
In view of the last theorem, it suffices to show that $\opn{rest}_f$ is 
essentially surjective on objects. Take any 
$M \in \cat{D}^{\mrm{b}}_{\mrm{f}}(A)_{\mrm{CM} : R_A}$.
Then 
$N :=  \opn{RHom}_{A}(M, R_A)$ is (isomorphic to) a module in 
$\cat{Mod}_{\mrm{f}} \bar{A}$. Because 
$\opn{rest}_{\bar{f}} : \cat{Mod}_{\mrm{f}} \bar{B} \to
\cat{Mod}_{\mrm{f}} \bar{A}$
is an equivalence, there is $N' \in \cat{Mod}_{\mrm{f}} \bar{B}$
that is sent to $N$. Therefore the DG module 
\[ M' := \opn{RHom}_{A}(N', R_B) \in 
\cat{D}^{\mrm{b}}_{\mrm{f}}(B)_{\mrm{CM} : R_B} \]
satisfies 
$\opn{rest}_f(M') \cong M$. 
\end{proof}

\begin{rem} \label{rem:120}
Here is a quick explanation of the role of CM DG modules in \cite{Ye5}. 
Suppose $\K \to A \to B$ are ring homomorphisms, $M \in \cat{D}(A)$ and 
$N \in \cat{D}(B)$. Under suitable assumptions we want to have a canonical 
isomorphism 
\[  \smallsmile \, : \opn{Sq}_{A / \K}(M) \ot^{\mrm{L}}_A   
\opn{Sq}_{B / A}(N) \iso
\opn{Sq}_{B / \K} (M \ot^{\mrm{L}}_A N) \]
in $\cat{D}(B)$, that we call the {\em cup product}. 
This isomorphism was already constructed in \cite[Theorem 4.11]{YZ1}; but 
unfortunately this part of \cite{YZ1} also contained a serious mistake. 

The construction in \cite{Ye5} goes like this. We choose a semi-free DG ring 
resolution $\K \to \til{A}$ of $\K \to A$, and then a semi-free DG ring 
resolution $\til{A} \to \til{B}$ of $\til{A} \to B$. 
So 
\[ \opn{Sq}_{A / \K}(M) = 
\opn{RHom}_{\til{A} \ot_{\K} \til{A}} (A, M \ot^{\mrm{L}}_{\K} M) \]
etc. The construction goes through a few ``standard moves'' (adjunction 
formulas mostly), until we arrive at the following situation. 
Consider the surjective DG ring homomorphism 
$f : \til{B} \ot_{\K} \til{B} \to \til{B} \ot_{\til{A}} \til{B}$,
and the DG modules 
\[ K := N \ot^{\mrm{L}}_A N \ot^{\mrm{L}}_A \opn{Sq}_{A / \K}(M) \]
and 
\[ L := \opn{RHom}_{\til{B} \ot_{\K} \til{B}}
\bigl( \til{B} \ot_{\til{A}} \til{B} , 
(M \ot^{\mrm{L}}_{A} N) \ot^{\mrm{L}}_{\K}
(M \ot^{\mrm{L}}_{A} N) \bigr) \]
in $\cat{D}(\til{B} \ot_{\til{A}} \til{B})$.
The ``standard moves'' give us a canonical isomorphism  \lb 
$\chi : \opn{rest}_f(K) \iso \opn{rest}_f(L)$
in $\cat{D}(\til{B} \ot_{\K} \til{B})$; but what we need to continue the 
construction is a canonical isomorphism
$\bar{\chi} : K \iso L$ in $\cat{D}(\til{B} \ot_{\til{A}} \til{B})$
such that $\opn{rest}_f(\bar{\chi}) = \chi$. 
The only conceivable hope was that something like Theorem \ref{thm:251} should 
appear. 

Fortunately, in the situation where we require the cup product, the ring $\K$ 
is a regular noetherian ring; $\K \to A$ is essentially finite type; 
$A \to B$ is essentially Gorenstein (i.e.\ essentially finite type, 
flat, and the fibers are Gorenstein rings); $M$ is a rigid dualizing complex 
over $A$ relative to $\K$; and $N$ is a tilting complex over $B$ (and hence it 
is a relative dualizing complex for $A \to B$). These assumptions imply that 
$K$ is a dualizing DG module over the DG ring 
$\til{B} \ot_{\til{A}} \til{B}$, and therefore it is a CM DG module w.r.t.\ 
itself. Now Theorem \ref{thm:251} says that there 
exists a unique isomorphism $\bar{\chi} : K \iso L$ satisfying 
$\opn{rest}_f(\bar{\chi}) = \chi$. 
\end{rem}

\section{Rigid DG Modules} \label{sec:rigid}

Recall that all our DG rings are commutative (Convention \ref{conv:100}); 
namely we work inside the category $\cat{DGR}^{\leq 0}_{\mrm{sc}}$.
At the beginning of this section we do not make any finiteness assumptions on 
DG rings. 

In our new paper \cite{Ye5} we introduced the {\em squaring operation} for
commutative DG rings. It is summarized in Theorem \ref{thm:500}  below,
which is a combination of \cite[Theorems 0.3.4, 0.3.5 and 7.16]{Ye5}.

Given a homomorphism $u : A \to B$ of DG rings, a {\em K-flat resolution} of 
$u$ is a commutative diagram 
\[ \UseTips \xymatrix @C=8ex @R=6ex {
\til{A}
\ar[r]^{\til{u}} ="tilu"
\ar[d]_{v} ="v"
{} \save 
[]+<-12ex,-1ex> *+[F-:<3pt>]{\scriptstyle \text{qu-isom}} 
\ar@(r,l)@{..} "v" 
\restore
&
\til{B}
\ar[d]^{w} ="w"
{} \save 
[]+<10ex,4ex> *+[F-:<3pt>]{\scriptstyle \text{K-flat}} 
\ar@(l,ur)@{..} "tilu" 
\restore
&
{} \save 
[]+<4ex,-1ex> *+[F-:<3pt>]{\scriptstyle \text{surj qu-isom}} 
\ar@(l,r)@{..} "w" 
\restore
\\
A
\ar[r]_{u} ="u"
&
B
} \]
in $\cat{DGR}^{\leq 0}_{\mrm{sc}}$, where $v$ is a quasi-isomorphism, $w$ is a 
surjective quasi-isomorphism, and $\til{B}$ is K-flat as a DG $\til{A}$-module. 
Such resolutions always exist. 

A K-flat resolution $\til{A} \to \til{B}$ of $A \to B$ gives rise to a functor 
$\opn{Sq}_{B / A}^{\til{B} / \til{A}} : \cat{D}(B) \to \cat{D}(B)$,
\[ \opn{Sq}_{B / A}^{\til{B} / \til{A}}(M) := 
\opn{RHom}_{\til{B} \ot_{\til{A}} \til{B}}(B, M \ot^{\mrm{L}}_{\til{A}} M) . \]

\begin{thm}[\cite{Ye5}] \label{thm:500}
Let $A \to B$ be a homomorphism of DG rings. 
\begin{enumerate}
\item There is a functor 
\[ \opn{Sq}_{B / A} : \cat{D}(B) \to \cat{D}(B) \ , \]
called the squaring operation, together with a compatible system of 
isomorphisms of functors
$\opn{Sq}_{B / A} \cong \opn{Sq}_{B / A}^{\til{B} / \til{A}}$,
where $\til{A} \to \til{B}$ runs over all K-flat resolutions  of $A \to B$
in $\cat{DGR}^{\leq 0}_{\mrm{sc}}$.

\item The functor $\opn{Sq}_{B / A}$ is quadratic, in the following 
sense\tup{:} for any morphism $\phi : M \to N$ in $\cat{D}(B)$, and any element
$b \in \bar{B}$, there is equality 
\[ \opn{Sq}_{B / A}(b \cd \phi) = b^2 \cd \phi \ . \]
\end{enumerate}
\end{thm}

The compatible system of isomorphisms of functors in item (1) of the theorem is 
stated in detail in \cite[Theorem 0.3.4]{Ye5}. These isomorphisms make the 
functor $\opn{Sq}_{B / A}$ unique up to a unique isomorphism. 
In part (2) we use the fact that $\cat{D}(B)$ is a 
$\bar{B}$-linear category (see Proposition \ref{prop:311}). 
 
\begin{dfn} \label{dfn:500}
Let $A \to B$ be a homomorphism of DG rings.
\begin{enumerate}
\item Let $M$ be a DG $B$-module. A {\em rigidifying isomorphism} for $M$ 
relative to $A$ is an isomorphism 
\[ \rho : M \iso \opn{Sq}_{B / A}(M) \]
in $\cat{D}(B)$.

\item A {\em rigid DG module over $B$ relative to $A$} is a pair 
$(M, \rho)$, consisting of a DG $B$-module $M$, and a rigidifying isomorphism
$\rho : M \iso \opn{Sq}_{B / A}(M)$ in $\cat{D}(B)$.
\end{enumerate}
\end{dfn}

Unlike in \cite{YZ1}, here we do not impose any finiteness conditions on $M$ 
as part of the definition of rigidity. 

\begin{dfn} \label{dfn:501}
Let $A \to B$ be a homomorphism of DG rings.
\begin{enumerate}
\item Let $(M, \rho)$ and $(M', \rho')$ be rigid DG modules over $B$ relative 
to $A$. A {\em rigid morphism} 
\[ \phi : (M, \rho) \to (M', \rho') \]
is a morphism $\phi : M \to M'$ in $\cat{D}(B)$, such that the diagram 
\[ \UseTips \xymatrix @C=5ex @R=5ex {
M 
\ar[r]^(0.35){\rho}
\ar[d]_{\phi}
&
\opn{Sq}_{B / A}(M)
\ar[d]^{\opn{Sq}_{B / A}(\phi)}
\\
M' 
\ar[r]^(0.35){\rho'}
&
\opn{Sq}_{B / A}(M')
} \]
in $\cat{D}(B)$ is commutative. 

\item The category of rigid DG modules over $B$ relative to $A$ is denoted by 
\lb $\cat{D}(B)_{\mrm{rig} / A}$.
\end{enumerate}
\end{dfn}

\begin{thm}[Uniqueness of Rigid Automorphisms] \label{thm:501}
Let $A \to B$ be a homomorphism of DG rings, and let 
$(M, \rho)$ be a rigid DG module over $B$ relative to $A$. Assume that the 
adjunction morphism 
$B \to \opn{RHom}_{B}(M, M)$ is an isomorphism. Then the only automorphism of 
$(M, \rho)$ in $\cat{D}(B)_{\mrm{rig} / A}$ is the identity. 
\end{thm}

\begin{proof}
The idea of the proof is already in \cite[Theorem 5.2]{Ye2}.
The adjunction condition implies that the ring homomorphism 
$\bar{B} \to \opn{End}_{\cat{D}(B)}(M)$
is bijective. Take any automorphism $\phi$ of $(M, \rho)$ in 
$\cat{D}(B)_{\mrm{rig} / A}$. Then $\phi = b \cd \opn{id}_M$, for a unique 
invertible element $b \in \bar{B}$. By item (2) of Theorem \ref{thm:500}
we know that 
$\opn{Sq}_{B / A}(\phi) = b^2 \cd \opn{Sq}_{B / A}(\opn{id}_M)$. 
On the other hand, because $\phi$ and $\opn{id}_M$ are both rigid, we get 
\[ \begin{aligned}
& \opn{Sq}_{B / A}(\phi) = \rho \circ \phi \circ \rho^{-1} = 
\rho \circ (b \cd \opn{id}_M) \circ \rho^{-1} 
\\
& \qquad = b \cd (\rho \circ \opn{id}_M \circ \rho^{-1}) = 
b \cd \opn{Sq}_{B / A}(\opn{id}_M) .
\end{aligned} \]
This shows that
$b^2 \cd \opn{id}_M = b \cd \opn{id}_M$, so $b^2 = b$, and hence $b = 1$. 
\end{proof}

\begin{dfn} \label{dfn:503}
Let $A$ be a tractable cohomologically pseudo-noetherian DG ring, with traction 
$\K \to A$ (see Definition \ref{dfn:46}). 
A {\em rigid dualizing DG module over $A$ relative to $\K$} is a 
rigid DG module $(R, \rho)$ over $A$ relative to $\K$ (Definition 
\ref{dfn:500}), such that $R$ is a dualizing DG module over $A$
 (Definition \ref{dfn:253}).
\end{dfn}

\begin{lem} \label{lem:526}
Suppose $\K$ is a noetherian ring, $A$ is a cohomologically pseudo-noetherian 
DG ring, and $\K \to A$ is a cohomologically essentially finite type 
homomorphism. Let $\K \to \til{A}$ be a K-flat resolution of $\K \to A$, and 
define 
$\til{A}^{\mrm{en}} := \til{A} \ot_{\K} \til{A}$. 
\begin{enumerate}
\item The DG ring $\til{A}^{\mrm{en}}$ is cohomologically pseudo-noetherian,
and the homomorphism $\K \to \til{A}^{\mrm{en}}$ is a cohomologically 
essentially finite type. 

\item Let $P \in \cat{D}(A)$ be a tilting DG $A$-module. Then 
$P \ot^{\mrm{L}}_{\K} P \in \cat{D}(\til{A}^{\mrm{en}})$
is a tilting DG $\til{A}^{\mrm{en}}$-module.
\end{enumerate}
\end{lem}

The derived tensor product $P \ot^{\mrm{L}}_{\K} P$ in item (2) is calculated 
as follows: we take any K-flat resolution $\til{P} \to P$ in 
$\cat{C}(\til{A})$, and define 
$P \ot^{\mrm{L}}_{\K} P := \til{P} \ot_{\K} \til{P}$.
Because $\til{P}$ is K-flat over $\K$, this is independent of the resolution, 
up to a canonical isomorphism in $\cat{D}(\til{A}^{\mrm{en}})$. 

\begin{proof}
(1) As in the proof of Lemma \ref{lem:380}, we can find a commutative diagram 
in $\cat{DGR}^{\leq 0}_{\mrm{sc}}$
\[  \UseTips \xymatrix @C=5ex @R=5ex {
\K
\ar[r]
\ar[d]
&
\til{A}
\ar[r]
\ar[d]^{\til{f}}
&
A 
\ar[r]
\ar[d]
& 
\bar{A}
\\
\til{A}_{\mrm{eft}}
\ar[r]^{\til{g}}
&
\til{A}_{\mrm{loc}}
\ar[r]
&
A_{\mrm{loc}}
\ar[ur]
} \]
such that $\til{f}$ is a K-flat quasi-isomorphism, $\til{g}$ is a 
quasi-isomorphism, the DG ring $\til{A}_{\mrm{eft}}$ is K-flat over $\K$, 
the ring $\til{A}^0_{\mrm{eft}}$ is essentially finite type over $\K$, and each 
module $\til{A}^i_{\mrm{eft}}$ is finite over $\til{A}^0_{\mrm{eft}}$.
Now $\til{A}$, $\til{A}_{\mrm{eft}}$ and $\til{A}_{\mrm{loc}}$
are all K-flat over $\K$. Hence, letting 
$\til{A}_{\mrm{loc}}^{\mrm{en}} := \til{A}_{\mrm{loc}} \ot_{\K} 
\til{A}_{\mrm{loc}}$
and
$\til{A}_{\mrm{eft}}^{\mrm{en}} := \til{A}_{\mrm{eft}} \ot_{\K} 
\til{A}_{\mrm{left}}$,
we get DG ring quasi-isomorphisms 
$\til{A}_{\mrm{eft}}^{\mrm{en}} \to \til{A}_{\mrm{loc}}^{\mrm{en}}$
and 
$\til{A}_{}^{\mrm{en}} \to \til{A}_{\mrm{loc}}^{\mrm{en}}$.
Therefore it suffices to prove the required assertions for 
$\til{A}_{\mrm{eft}}^{\mrm{en}}$. 

Consider the ring 
$(\til{A}_{\mrm{eft}}^{\mrm{en}})^0 = \til{A}^0_{\mrm{eft}} \ot_{\K} 
\til{A}^0_{\mrm{eft}}$. 
It is essentially finite type over $\K$, and thus it is also noetherian. 
For any $i \leq 0$ we have 
\[ (\til{A}_{\mrm{eft}}^{\mrm{en}})^i = 
\bigoplus_{i \leq p \leq 0} \til{A}^p_{\mrm{eft}} \ot_{\K} 
\til{A}^{i - p}_{\mrm{eft}} , \]
and this is a finite module over  
$(\til{A}_{\mrm{eft}}^{\mrm{en}})^0$. 
Therefore the ring $\opn{H}^0(\til{A}_{\mrm{eft}}^{\mrm{en}})$,
being a quotient of $(\til{A}_{\mrm{eft}}^{\mrm{en}})^0$, is 
essentially finite type over $\K$. Each 
$\opn{H}^i(\til{A}_{\mrm{eft}}^{\mrm{en}})$,
being a subquotient of $(\til{A}_{\mrm{eft}}^{\mrm{en}})^i$, 
is a finite module over $\opn{H}^0(\til{A}_{\mrm{eft}}^{\mrm{en}})$. 

\medskip \noindent
(2) Let $Q \in \cat{D}(A)$ be a quasi-inverse of $P$. Then 
\[ (Q \ot^{\mrm{L}}_{\K} Q) \ot^{\mrm{L}}_{A} (P \ot^{\mrm{L}}_{\K} P) \cong
(Q \ot^{\mrm{L}}_{A} P)  \ot^{\mrm{L}}_{\K} (Q \ot^{\mrm{L}}_{A} P)
\cong A  \ot^{\mrm{L}}_{\K} A \cong \til{A}^{\mrm{en}} \]
in $\cat{D}(\til{A}^{\mrm{en}})$. 
\end{proof}

\begin{thm}[Uniqueness of Rigid Dualizing DG Modules] \label{thm:502}
In the situation of Definition \tup{\ref{dfn:503}}, suppose that
$(R, \rho)$ and $(R', \rho')$ are rigid dualizing DG modules over $A$ relative 
to $\K$. Then there is a unique isomorphism 
$(R, \rho) \cong (R', \rho')$ in $\cat{D}(A)_{\mrm{rig} / \K}$.
\end{thm}

\begin{proof}
The idea of first half of the proof goes back to the original work of Van den 
Bergh \cite{VdB}. According to Theorem \ref{thm:43}(2), there is a tilting DG 
module $P$ such that $R' \cong P \ot^{\mrm{L}}_{A} R$
in $\cat{D}(A)$. 
Choose a K-flat DG ring resolution $\K \to \til{A}$ of $\K \to A$,
and let $\til{A}^{\mrm{en}} := \til{A} \ot_{\K} \til{A}$. 
Then we have isomorphisms 
\[ \begin{aligned}
& P \ot^{\mrm{L}}_{A} R \cong^1 R' \cong^{2}
\opn{RHom}_{\til{A}^{\mrm{en}}}(A, R' \ot^{\mrm{L}}_{\K} R')
\\ & \quad 
\cong^1 \opn{RHom}_{\til{A}^{\mrm{en}}} \bigl( A, (P \ot^{\mrm{L}}_{A} R) 
\ot^{\mrm{L}}_{\K} (P \ot^{\mrm{L}}_{A} R) \bigr)
\\ & \quad 
\cong^3 \opn{RHom}_{\til{A}^{\mrm{en}}} \bigl( A, (R \ot^{\mrm{L}}_{\K} R) 
\ot^{\mrm{L}}_{\til{A}^{\mrm{en}}} (P \ot^{\mrm{L}}_{\K} P) \bigr)
\\ & \quad 
\cong^4 \opn{RHom}_{\til{A}^{\mrm{en}}} ( A, R \ot^{\mrm{L}}_{\K} R) 
\ot^{\mrm{L}}_{\til{A}^{\mrm{en}}} (P \ot^{\mrm{L}}_{\K} P) 
\\ & \quad 
\cong^2 R \ot^{\mrm{L}}_{\til{A}^{\mrm{en}}} (P \ot^{\mrm{L}}_{\K} P) 
\cong^3 R \ot^{\mrm{L}}_{A} P \ot^{\mrm{L}}_{A} P 
\end{aligned} \]
in $\cat{D}(A)$. The isomorphisms $\cong^1$ come from the given isomorphism 
$R' \cong P \ot^{\mrm{L}}_{A} R$
in $\cat{D}(A)$; the isomorphisms $\cong^2$ come from the rigidifying 
isomorphisms $\rho'$ and $\rho$, together with the isomorphisms from 
Theorem \ref{thm:500}(1); the isomorphism $\cong^3$
is by a standard tensor product identity (that is valid also in this derived 
setting, because $\til{A}$ is K-flat over $\K$); and the 
the isomorphism $\cong^4$ is by Lemma \ref{lem:45}, that 
applies because $\til{A}^{\mrm{en}}$ is cohomologically pseudo-noetherian,
$R \ot^{\mrm{L}}_{\K} R$ is in $\cat{D}^+(\til{A}^{\mrm{en}})$, and 
$P \ot^{\mrm{L}}_{\K} P$ is perfect over $\til{A}^{\mrm{en}}$ -- thanks to 
Lemma \ref{lem:526}. 
Now Theorem \ref{thm:43}(3) tells us that 
$P \ot^{\mrm{L}}_{A} P \cong P$ in $\cat{D}(A)$. This implies that $P \cong A$.
We conclude that $R' \cong R$ in $\cat{D}(A)$. 

The remainder of the proof is the same as in the proof of 
\cite[Theorem 5.2]{Ye2}.
Let $\phi : R \iso R'$ be any isomorphism in $\cat{D}(A)$. The calculation in 
the proof of Theorem \ref{thm:501} shows that 
$\opn{Sq}_{A / \K}(\phi) = 
a \cd (\rho' \circ \phi \circ \rho^{-1})$
for a unique invertible element $a \in \bar{A}$. Then 
$a^{-1} \cd \phi : R \to R$ is a rigid isomorphism; and it is unique by 
Theorem \ref{thm:501}.
\end{proof}

Existence of rigid dualizing DG modules in this generality is not known, but we 
believe it it true -- this is Conjecture \ref{conj:505} in the Introduction. 
Conjecture \ref{conj:501} below implies Conjecture \ref{conj:505}.

Suppose $u : A \to B$ is a homomorphism between cohomologically 
pseudo-noetherian DG rings. We say that $u$ is {\em 
cohomologically essentially smooth of relative dimension $n$} if 
$\bar{u} : \bar{A} \to \bar{B}$ is essentially smooth of relative dimension 
$n$, 
and the induced homomorphism 
$\opn{H}(A) \ot_{\bar{A}} \bar{B} \to \opn{H}(B)$
is bijective. In this case, the module of differentials 
$\Omega^1_{\bar{B} / \bar{A}}$ is projective of rank $n$, and hence 
$\Omega^n_{\bar{B} / \bar{A}}$ is projective of rank $1$. 
We define 
$\Omega^n_{B / A}[n] \in \cat{D}(B)$ to be the lift of 
$\Omega^n_{\bar{B} / \bar{A}}[n] \in \cat{D}(\bar{B})$,
under the group isomorphism 
$\opn{DPic}(B) \cong \opn{DPic}(\bar{B})$
of Theorem \ref{thm:40}. There is an alternative description of 
$\Omega^n_{B / A}[n]$ in terms of the {\em cotangent complex} of $B / A$. 
 
\begin{conj} \label{conj:501}
In the situation of Definition \tup{\ref{dfn:503}}, suppose 
$(R_A, \rho_A)$ is a rigid dualizing DG module over $A$ relative to $\K$.
Let $B$ be another DG ring, and let $A \to B$ be a cohomologically 
essentially finite type homomorphism. 
\begin{enumerate}
\item If $A \to B$ is cohomologically pseudo-finite, then the DG module 
\[ R_B := \opn{RHom}_{A}(B, R_A ) \in  \cat{D}(B) \]
has an induced rigidifying isomorphism $\rho_B$ relative to $\K$. 

\item If $A \to B$ is cohomologically essentially smooth of relative dimension 
$n$, then the DG module 
\[ R_B := \Omega^n_{B / A}[n] \ot^{\mrm{L}}_{A} R_A \in  \cat{D}(B) \]
has an induced rigidifying isomorphism $\rho_B$ relative to $\K$. 
\end{enumerate}
\end{conj}

In the case when  $A$ and $B$ are rings, Conjectures \ref{conj:505} and 
\ref{conj:501} were already proved in \cite{YZ1}.
The proofs will be repeated, with improvements, in \cite{Ye7}. 
In the case when the cohomology of $A$ is bounded, Conjectures \ref{conj:505} 
and \ref{conj:501} were very recently proved by Shaul \cite{Sh2}. 
What remains to prove is the case when $\opn{H}(A)$ is unbounded. 

\begin{rem} \label{rem:520}
Finally, a few words on Conjecture \ref{conj:506} from the Introduction. 
We can see three possible routes to try to prove it.
The first route is by generalizing the original proof of 
\cite[Theorem 1.2]{YZ2}, that talked about a regular tractable ring $A$, to the
DG ring case. The second route is by generalizing the proof 
of \cite[Theorem 4]{AIL}, which applied to any tractable ring $A$, to the DG 
ring case. 

There is also a more conceptual route. 
In \cite{Ye5} we introduced the {\em rectangle operation} 
\[ \opn{Rect}_{A / \K}(M, N) := 
\opn{RHom}_{\til{A} \ot_{\K} \til{A}}(A, M \ot^{\mrm{L}}_{\K} N) , \]
for $M, N \in \cat{D}(A)$. Here $\K \to \til{A}$ is a K-flat resolution of 
$\K \to A$, but the rectangle operation is independent of this resolution. 
Notice that the square is a special instance of the rectangle:
$\opn{Sq}_{A / \K}(M) =  \opn{Rect}_{A / \K}(M, M)$.

Recently Shaul \cite{Sh1} proved that when $A$ is a tractable ring, and under 
certain finiteness conditions on $M$ and $N$, there is a canonical isomorphism 
\[ \opn{Rect}_{A / \K}(M, N) \cong 
D \bigl( D(M) \ot^{\mrm{L}}_{A} D(N) \bigr) , \]
where $D := \opn{RHom}_{A}(-, R_A)$, and $R_A$ is the rigid dualizing complex 
of $A$ relative to $\K$. Thus, writing 
$M \ot_{A / \K}^! N := \opn{Rect}_{A / \K}(M, N)$,
this operation becomes a symmetric monoidal structure on (a suitable 
subcategory of) $\cat{D}^{+}_{\mrm{f}}(A)$.

Suppose now that $M \in \cat{D}^{\mrm{b}}_{\mrm{f}}(A)$ is rigid relative to 
$\K$. Define 
$L := D(M) \in \cat{D}^{\mrm{b}}_{\mrm{f}}(A)$. 
Then $L$ satisfies 
$L \ot^{\mrm{L}}_{A} L \cong L$. If $M$ is nonzero on each connected component 
of $\opn{Spec} \bar{A}$, then the same holds for $L$. 
It is not hard to show that in this case $L \cong A$ in $\cat{D}(A)$. Therefore
$M \cong R_A$, so it is a rigid dualizing complex. 

If this work of Shaul could be extended to cover tractable DG rings, we would 
have a proof of Conjecture \ref{conj:506}, at least when $A$ is cohomologically 
bounded. 
\end{rem}


\end{document}